\documentclass[11pt]{article}
\usepackage{amsfonts}

\usepackage{graphics}
\usepackage{indentfirst}
\usepackage{cite}
\usepackage{latexsym}
\usepackage{amsmath,amsthm}
\usepackage{amssymb}
\usepackage{amscd}
\usepackage{mathrsfs}

\usepackage{color}
\usepackage[a4paper,margin=24mm]{geometry}
\usepackage[unicode,hypertexnames=false,colorlinks=true,linkcolor=blue,citecolor=blue,urlcolor=blue]{hyperref}
\allowdisplaybreaks

\newtheorem{theorem}{Theorem}[section]
\newtheorem{remark}{Remark}[section]

\newtheorem{definition}{Definition}[section]
\newtheorem{lemma}[theorem]{Lemma}

\newtheorem{proposition}[theorem]{Proposition}

\newcommand{\n}{\rho}

\renewcommand{\div}{ {\rm div }  }

\newcommand{\na}{\nabla }

\newcommand{\bt}{\begin{theorem}}
\newcommand{\bl}{\begin{lemma}}
\newcommand{\el}{\end{lemma}}
\newcommand{\et}{\end{theorem}}
\newcommand{\ga}{\gamma}

\newcommand{\OM}{\Omega}

\newcommand{\curl}{{\rm curl} }

\newcommand{\la}{\label}
\newcommand{\si}{\sigma}

\newcommand{\ol}{\overline}

\newcommand{\bn}{\begin{eqnarray}}
\newcommand{\en}{\end{eqnarray}}
\newcommand{\bnn}{\begin{eqnarray*}}
\newcommand{\enn}{\end{eqnarray*}}

\newcommand{\bnnn}{\begin{eqnarray*}}
\newcommand{\ennn}{\end{eqnarray*}}
\newcommand{\ben}{\begin{enumerate}}
\newcommand{\een}{\end{enumerate}}

\newcommand{\rs}{\rho^*}

\newcommand{\ba}{\begin{aligned}}
\newcommand{\ea}{\end{aligned}}
\newcommand{\be}{\begin{equation}}
\newcommand{\ee}{\end{equation}}

\def\p{\partial}
\def\norm[#1]#2{\|#2\|_{#1}}

\def\lam{\lambda}
\def\ep{\varepsilon}

\def\rr{\mathbb{R}^2}

\makeatletter      
\@addtoreset{equation}{section}
\makeatother       

\title{Global Existence and Large-Time Asymptotic Behavior of Strong Solutions to the Two-Dimensional Nematic Liquid Crystal Flows with Large Initial Data and Vacuum}

\date{}

\author{$\text{Qinghao L{\small EI}}^{a,b}, \text{Lu W{\small{ANG}}}^{b}\thanks{Email addresses:  leiqinghao22@mails.ucas.ac.cn (Q. H. Lei), wlu1130@163.com (L. Wang) }$\\
a. School of Mathematical Sciences,\\ University of Chinese Academy of Sciences,
Beijing 100049, P. R. China;\\
b. Institute of Applied Mathematics,\\ Academy of Mathematics and Systems Science, \\
Chinese Academy of Sciences, Beijing 100190, P. R. China}

\begin{document}
\maketitle

\begin{abstract}
This paper studies the two-dimensional nonhomogeneous incompressible nematic liquid crystal flows with planar orientation fields taking values in $\mathbb{S}^1$.
For the initial-boundary value problem in bounded domains with density-dependent viscosity, we establish the global existence and exponential decay of strong solutions with initial density allowing vacuum, provided that $\|\nabla \mu(\rho_0)\|_{L^q}$ is sufficiently small for some $q>2$.
The smallness condition is automatically satisfied when $\mu$ is constant, and hence the result yields the global existence of strong solutions for arbitrarily large initial data in the constant viscosity case.
Furthermore, for the Cauchy problem with constant viscosity and either vacuum or non-vacuum far-field density, we establish the global existence and large-time decay rates of strong solutions for arbitrarily large initial data.
These results are obtained without imposing any geometric angle conditions on the initial orientation field. 
\end{abstract}

\section{Introduction and main results}

Liquid crystals are a state of matter that exhibits properties intermediate between those of conventional liquids and solids.
They are typically classified according to molecular arrangement, with each phase defined by a distinct structural configuration.
The simplest of these is the nematic phase, notable for its stability at room temperature and wide applicability.
In the 1960s, Ericksen \cite{E1961,E1962} and Leslie \cite{Les1968} formulated a mathematical model for nematic liquid crystal flows, now known as the Ericksen-Leslie system.

In this paper, we investigate the nonhomogeneous incompressible nematic liquid crystal flows with density-dependent viscosity, described by the following equations:
\be\ba\la{nlc}
\begin{cases}
\rho_t + \div(\rho u) = 0,\\
(\n u)_t + \div(\n u\otimes u) -\div(2\mu(\rho) Du) + \na P = - \lam \div (\na d \odot \na d),\\
d_t + u \cdot \na d = \ga (\Delta d + |\na d|^2 d),\\
\div u = 0,  \quad |d|=1,
\end{cases}
\ea\ee
where $\n=\n(x,t)$, $u(x,t)=(u^1(x,t),u^2(x,t))$, and $P=P(x,t)$ denote the 
density, velocity, and pressure of the fluid, respectively.
The vector field $d=(d^1(x,t),d^2(x,t))$ represents the macroscopic molecular orientation of the liquid crystal material. The deformation tensor is defined
by
\be\nonumber\ba
Du=\frac{1}{2} \left[ \na u + (\na u)^{T} \right],
\ea\ee
and the viscosity $\mu(\n)$ satisfies the following hypothesis:
\be\la{vc}\ba
\mu \in C^1([0,\infty)), \quad \mu(\n)>0.
\ea\ee
The positive constants $\lam$ and $\ga$ denote the elastic coupling coefficient and the relaxation coefficient for the molecular orientation field, respectively.
The notation $\na d \odot \na d$ denotes the $2 \times 2$ matrix whose
$(i,j)$-th entry is given by $\p_i d \cdot \p_j d \ (1 \le i,j \le 2)$.

For simplicity, we assume that $\lam=\ga=1$.
We supplement \eqref{nlc} with the initial data
\be\la{cztj}\ba
\n(x,0)=\n_0(x), \ \n u(x,0)=\n_0 u_0(x), \ d(x,0)=d_0(x), \ x \in \OM,
\ea\ee
and consider the following two types of boundary conditions:

(1) Dirichlet and Neumann boundary conditions:
$\OM$ is a simply connected bounded smooth domain in $\rr$ and
\be\la{bjtj1}\ba
u=0,\quad \frac{\p d}{\p n}=0 \ \text{ on } \ \p \OM,
\ea\ee
where $n=(n_1,n_2)$ denotes the unit outer normal vector of the boundary $\partial \Omega$.

(2) Cauchy problem: $\OM=\rr$ with the far-field behavior
\be\la{bjtj3}\ba
(\n,u,d)(x,t) \to (\tilde{\n},0,\mathbf{1}) \ \text{ as } |x| \to \infty, \  t>0,
\ea\ee
where $\tilde{\n} \ge 0$ is a constant and $\mathbf{1}$ is a given unit vector.

This system couples the nonhomogeneous incompressible Navier–Stokes equations, incorporating the density-dependent viscous stress and the elastic stress induced by the orientation field, with an evolution equation for the orientation field that accounts for convective effects.
The resulting strong coupling poses substantial analytical challenges.
For further background on system \eqref{nlc}, we refer the reader to the monograph by Stewart \cite{Ste2004} and the references therein.

There is an extensive literature on the homogeneous case, i.e., $\rho$ is constant.
For initial data satisfying $u_0 \in L^2$ and $\nabla d_0 \in L^2$, Lin-Liu \cite{LL1995,LL1996} proved the global existence of weak solutions to an approximate model in dimensions $N=2,3$ via a Ginzburg-Landau approximation, in which the term $|\nabla d|^2 d$ in the director equation is replaced by $\frac{1-|d|^2}{\varepsilon^2}d$.
Hong \cite{H2011} and Hong-Xin \cite{H2012} obtained the global existence of Leray-Hopf type weak solutions to the two-dimensional Cauchy problems for the simplified system and the system with a general Oseen-Frank energy, respectively, by passing to the limit as $\varepsilon \to 0$ in the Ginzburg-Landau approximation.
Without resorting to this approximation, Lin et al. \cite{LLW2010} directly established the global existence of weak solutions to the original system.
Lin-Wang \cite{LW2010} further proved the uniqueness of such weak solutions, and this result was later extended by Xu-Zhang \cite{XZ2012}.
Subsequently, Wang \cite{W2011} obtained global strong solutions with higher regularity for the two-dimensional system, provided that $\left\|u_0\right\|_{\mathcal{BMO}^{-1}}$ and $\left\|d_0\right\|_{\mathcal{BMO}}$ are sufficiently small.
Under a natural geometric angle condition, Lei-Li-Zhang \cite{LLZ2014} provided a new proof of the global well-posedness of smooth solutions for a class of large initial data in the energy space.
For further results on general liquid crystal systems, we refer the reader to \cite{HW2012,LW2016}.

The well-posedness theory for the nonhomogeneous incompressible nematic liquid crystal system with constant viscosity has also been established under various assumptions.
For example, Wen-Ding \cite{WD2011} proved the local existence and uniqueness of strong solutions in two-dimensional bounded domains for general initial data allowing vacuum and satisfying the compatibility condition
\be\nonumber\ba
-\mu \Delta u_0+\nabla P_0+\operatorname{div}\left(\nabla d_0 \odot \nabla d_0\right)=\sqrt{\rho_0} g, \quad \text { for some }\ \left(P_0, g\right) \in H^1(\Omega) \times L^2(\Omega).
\ea\ee
Under the additional assumptions that the initial density has a positive lower bound and that both $\|\nabla d_0\|_{L^2}$ and $\|\sqrt{\n_0} u_0\|_{L^2}$ are sufficiently small, they also established the global existence and uniqueness in two-dimensional bounded domains.
Subsequently, Li-Liu-Zhong \cite{LLZ2017} extended the global existence result to the two-dimensional whole space with vacuum far-field density under analogous smallness conditions.
Li \cite{L2014} obtained the global existence of strong solutions under the assumptions that the initial density is away from vacuum and that the initial orientation field satisfies a natural geometric angle condition.
Later, Liu et al. \cite{LLTZ} proved the global existence of strong solutions for the two-dimensional Cauchy problem, allowing the initial density to vanish and assuming that the initial orientation field satisfies a geometric angle condition.
Further results on the well-posedness of nonhomogeneous incompressible nematic liquid crystal flows with constant viscosity can be found in \cite{L2015,Li2017,LZ2016} and the references therein.

The main purpose of this paper is to establish the global existence and large-time asymptotic behavior of strong solutions to the problem \eqref{nlc}--\eqref{cztj} with large initial data.
More precisely, we consider both the bounded domain problem with density-dependent viscosity and the Cauchy problem with constant viscosity.

Before stating the main results, we first explain the notations and conventions used throughout this paper.
We denote
\be\ba\nonumber
\int f dx \triangleq \int_{\OM} f dx.
\ea\ee
When $|\OM|<\infty$, we further write
\be\ba\nonumber
\ol{f} \triangleq \frac{1}{|\OM|}\int f dx.
\ea\ee
For $R>0$, we set
\be\ba\nonumber
B_R \triangleq \{x \in \rr: |x|<R \}.
\ea\ee
Moreover, the material derivative is given by
\be\ba\nonumber
\frac{d}{dt} f = \dot{f} \triangleq \frac{\p}{\p t} f + u \cdot \na f.
\ea\ee
For a positive integer $k$ and $1\leq r\leq\infty$, we define the standard homogeneous and inhomogeneous Sobolev
spaces as follows:
\be\ba\nonumber
{\left\{\begin{array}{ll}
L^r=L^r(\OM),\quad W^{k,r} = W^{k,r}(\OM), \quad H^k = W^{k,2}, \\
\| \cdot \|_{B_1 \cap B_2} = \| \cdot \|_{B_1} + \| \cdot \|_{B_2}, \mbox{ for two Banach spaces } B_1 \mbox{ and } B_2, \\
D^{k,r}=D^{k,r}(\rr )=\{v\in L^1_\mathrm{loc}(\rr )| \nabla ^k v\in L^r(\rr )\}, \quad D^1=D^{1,2}, \\
C_{0,\sigma}^\infty=\{f\in C_0^\infty(\OM)~|~\div f=0\},\quad H_{0,\sigma}^1=\overline{C_{0,\sigma}^\infty}~\mbox{closure~in~the~norm~of}~H^{1}, \\
D_{0,\sigma}^1=\overline{C_{0,\sigma}^\infty}~\mbox{closure~in~the~norm~of}~D^{1}.
\end{array}\right.}
\ea\ee

We now give the definition of strong solutions to \eqref{nlc}.
\begin{definition}
If all derivatives involved in \eqref{nlc} for $(\n,u,P,d)$ are regular distributions,
and equations \eqref{nlc} hold almost everywhere in 
$\OM \times(0,T)$, then $(\n,u,P,d)$ is called a strong solution.
\end{definition}

Our first result concerns the global existence and exponential decay of strong solutions with density-dependent viscosity in bounded domains.

\begin{theorem}\la{thnp1}
For constants $\hat{\n}>0$ and $q\in(2,\infty)$, assume that the initial data $(\n_0,u_0,d_0)$ satisfy
\be\la{npsol1}\ba
0 \le \n_0 \le \hat{\n},\ \n_{0}\in W^{1,q}, \ u_0 \in H_{0,\sigma}^{1},\ \nabla d_{0} \in H^{1},\  |d_{0}|=1,
\ (n \cdot \na d_0)|_{\p \OM}=0.
\ea\ee
Define
\be\nonumber\ba
\underline{\mu} \triangleq \min_{ \n \in [0,\hat{\n}] } \mu(\n), \quad
\bar{\mu} \triangleq \max_{ \n \in [0,\hat{\n}] } \mu(\n).
\ea\ee
Then there exists a small positive constant $\varepsilon_0$ depending only on $q$, $\underline{\mu}$, $\bar{\mu}$, $\hat{\rho}$, $\| u_0 \|_{H^1}$, $\|\na d_0 \|_{H^1}$, and $\OM$ such that if
\begin{equation}\label{npsol01}
\left\|\nabla \mu\left(\rho_0\right)\right\|_{L^q} \leq \varepsilon_0,
\end{equation}
the problem \eqref{nlc}--\eqref{bjtj1} admits a unique global strong solution
$(\n,u,P,d)$ in $\OM \times (0,\infty)$ satisfying for any $0<T<\infty$ and $r\in [1,q)$,
\be\la{npsol2}\ba
\begin{cases}
\rho\in C([0,T];W^{1,q} ), \quad \n_t\in L^\infty(0,T;L^2), \\
\nabla u \in L^\infty(0,T;L^2) \cap L^2(0,T;H^1), \quad P \in L^2(0,T;H^1), \\
\sqrt{t}\nabla u,\ \sqrt{t} P \in  L^\infty(0,T;H^1)\cap L^2(0,T; W^{1,r}), \\
\nabla d \in L^\infty (0,T;H^1)\cap L^{2}(0,T;H^2),\\
\sqrt{t} \nabla d \in L^\infty(0,T;H^2) \cap L^2(0,T;H^3), \\
\sqrt{t}u_t,\ \sqrt{t} \na d_t \in L^2(0,T;H^1), \\
\n u\in C([0,T];L^2), \quad \sqrt{\n} u_t\in L^2(\OM \times(0,T)),
\end{cases}
\ea\ee
with the pressure normalized by $\ol{P} = 0$.
Moreover, there exists a positive constant $\eta_0$ depending only on 
$\underline{\mu}$, $\hat{\n}$, and $\OM$ such that, for all $t \ge 1$,
\be\la{npsol3}\ba
\|\nabla u(\cdot,t)\|^2_{H^{1}} + \|P(\cdot,t)\|^2_{H^{1}}
+ \|\nabla d(\cdot, t)\|^2_{H^{2}} + \|\nabla d_t(\cdot, t)\|^2_{L^{2}}
\leq C e^{-\eta_0 t},
\ea\ee
where $C$ depends only on $q$, $\underline{\mu}$, $\bar{\mu}$, $\hat{\n}$, $\| u_0 \|_{H^1}$, $\| \na d_0 \|_{H^1}$, and $\OM$.
\end{theorem}

The second result establishes the global existence and large-time asymptotic behavior of strong solutions to the Cauchy problem \eqref{nlc}--\eqref{cztj}, \eqref{bjtj3} with constant viscosity and vacuum far-field density.

\begin{theorem}\la{thcp1}
Let $\OM=\rr$ and $\mu(\n) \equiv \mu > 0$.
Assume that the initial data $(\n_0,u_0,d_0)$ satisfy, for some $a>1$ and $q\in(2,\infty),$
\be\la{cpsol1}\ba
\begin{cases}
\n_{0}\geq 0, \quad \bar{x}^{a}\n_{0}\in L^{1}\cap H^{1}\cap W^{1,q}, \quad M_0 \triangleq \| \n_0 \|_{L^1(\rr)}>0, \quad \sqrt{\n_0} u_{0}\in L^{2}, \\ u_0 \in D_{0,\sigma}^{1}, \quad \bar{x}^{\frac{a}{2}} \nabla d_{0} \in L^{2},\quad \nabla^{2} d_{0}\in L^{2},
\quad d_0-\mathbf{1}\in L^2,\quad |d_{0}|=1,
\end{cases}
\ea\ee
where
\be\la{cpsol2}\ba
{\bar{x}}\triangleq (e+|x|^2)^{\frac{1}{2}} \log ^2 (e+|x|^2).
\ea\ee
Then the problem \eqref{nlc}--\eqref{cztj}, \eqref{bjtj3} with $\tilde{\n}=0$ admits a unique global strong solution $(\n,u,P,d)$ in $\rr \times (0,\infty)$ satisfying for any $0<T<\infty$,
\be\la{cpsol3}\ba
\begin{cases}
\rho \in C([0,T];L^1 \cap H^1\cap W^{1,q} ), \quad d-\mathbf{1} \in C([0,T];L^2) \\
{\bar{x}}^a\rho \in L^\infty ( 0,T ;L^1\cap H^1\cap W^{1,q} ), \\
\sqrt{\rho} u,\,\nabla u,\, {\bar{x}}^{-1}u, \, \sqrt{t} \sqrt{\rho } u_t,\ \sqrt{t} \na P,\ \sqrt{t} \na^2 u \in L^\infty (0,T;L^2), \\
\nabla d,\  \nabla d \bar{x}^{\frac{a}{2}},\  \nabla^{2} d,\  \sqrt{t} \nabla d_{t},\  \sqrt{t} \nabla^{3}d,\ 
\sqrt{t} \na^2 d \bar{x}^{\frac{a}{2}} \in L^\infty (0,T;L^2), \\
\nabla u\in L^2(0,T;H^1)\cap L^{(q+1)/q}(0,T; W^{1,q}), \\
\nabla P \in L^2(0,T;L^2)\cap L^{(q+1)/q}(0,T; L^q), \\
\nabla^{3} d,\, \nabla d_{t},\, \nabla^{2} d \bar{x}^{\frac{a}{2}}\in L^{2}(\mathbb{R}^{2}\times (0,T)), \\
\sqrt{t}\nabla u\in L^2(0,T; W^{1,q} ), \\
\sqrt{\rho } u_t, \, \sqrt{t} \nabla u_t,\ \sqrt{t} \na^2 d_t, \, \sqrt{t} {\bar{x}}^{-1}u_t\in L^2({\mathbb {R}^2 }\times (0,T)),
\end{cases}
\ea\ee
and
	\be\la{cpsol4}\ba
	\inf_{0 \le t \le T} \int _{B_{N_1}}\rho (x,t) dx \ge \frac{1}{4}M_0,
	\ea\ee
	for some positive constant $N_1$ depending only on
	$T$, $\mu$, $M_0$, $\|\bar{x}^{a}\rho_0\|_{L^1}$,
	$\| \sqrt{\n_0} u_0 \|_{L^2}$, $\| \na d_0 \|_{L^2}$, and some  $N_0$ satisfying $\int_{B_{N_0}}\rho_0 dx\geq \frac{M_0}{2}$. Moreover, the following decay estimates hold for all $t \ge 1$:
\be\la{cpsol5}\ba
\begin{cases}
\|\nabla u(\cdot,t)\|_{L^{2}} + \|\nabla^{2} d(\cdot, t)\|_{L^{2}}\leq C t^{-\frac{1}{2}}, \\
\|\nabla^{2} u(\cdot,t)\|_{L^{2}} + \|\nabla P(\cdot,t)\|_{L^{2}}
+ \||\nabla d||\nabla^{2} d|(\cdot, t)\|_{L^{2}} \leq C t^{-1},
\end{cases}
\ea\ee
where $C$ depends only on $\mu$, $\|\rho_0\|_{L^1\cap L^\infty}$,
	$\| \sqrt{\n_0} u_0 \|_{L^2}$, $\| \na u_0 \|_{L^2}$, and $\| \na d_0 \|_{H^1}$.
\end{theorem}

Finally, for the Cauchy problem \eqref{nlc}--\eqref{cztj}, \eqref{bjtj3} with constant viscosity and non-vacuum far-field density, we have the following global existence result.

\begin{theorem}\la{th2cp1}
Let $\OM=\rr$ and $\mu(\n) \equiv \mu > 0$.
Assume that the initial data $(\n_0,u_0,d_0)$ satisfy, for some $s \in [1,\infty)$ and $q \in (2,\infty),$
\be\la{2cpsol1}\ba
	\n_0 \ge 0, \ \n_{0} - \tilde{\n} \in L^{s} \cap L^\infty \cap D^{1,q}, \
	u_0 \in H_{0,\sigma}^{1}, \
	\nabla d_{0} \in H^{1}, \ d_0-\mathbf{1}\in L^2, \ |d_{0}|=1.
\ea\ee
Then the problem \eqref{nlc}--\eqref{cztj}, \eqref{bjtj3} with $\tilde{\n}>0$ admits a unique global strong solution $(\n,u,P,d)$ in $\rr \times (0,\infty)$ satisfying, for any $0<T<\infty$,
\be\la{2cpsol3}\ba
\begin{cases}
\rho - \tilde{\rho} \in C([0,T];L^s \cap L^\infty \cap D^{1,q} ), \\
u,\ \na d \in L^\infty(0,T; H^1) \cap L^{(q+1)/q}(0,T; W^{2,q}), \\
\sqrt{t}u,\ \sqrt{t} \na d \in L^2(0,T; W^{2,q}) \cap L^\infty(0,T;H^2), \\
\nabla P \in L^2(0,T;L^2), \quad \sqrt{t} \nabla P \in L^\infty(0,T;L^2) \cap L^{(q+1)/q}(0,T;L^q), \\
d-\mathbf{1} \in C([0,T];L^2),\ \sqrt{t}u_t,\ \sqrt{t} \na d_t \in L^2(0,T;H^1), \\
\n u\in C([0,T];L^2), \quad \sqrt{\n} u_t\in L^2(\rr \times(0,T)).
\end{cases}
\ea\ee
Moreover, the following decay estimates hold for all $t \ge 1$:
\be\la{2cpsol5}\ba
\begin{cases}
\|\nabla u(\cdot,t)\|_{L^{2}} + \|\nabla^{2} d(\cdot, t)\|_{L^{2}}\leq C t^{-\frac{1}{2}}, \\
\|\nabla^{2} u(\cdot,t)\|_{L^{2}} + \|\nabla P(\cdot,t)\|_{L^{2}}
+ \||\nabla d||\nabla^{2} d|(\cdot, t)\|_{L^{2}} \leq C t^{-1},
\end{cases}
\ea\ee
	where $C$ depends only on $\mu$, $\tilde{\n}$, $s$,
	$\|\n_0-\tilde{\n}\|_{L^s\cap L^\infty}$,
	$\| u_0 \|_{H^1}$, and $\| \na d_0 \|_{H^1}$.
\end{theorem}

A few remarks are in order.

\begin{remark}
We emphasize that Theorem~\ref{thnp1} requires only that $\| \na \mu(\n_0) \|_{L^q}$ be sufficiently small.
Since this condition is automatically satisfied when $\mu(\n)$ is constant, Theorem~\ref{thnp1} also yields the global existence of strong solutions in the constant-viscosity case.
\end{remark}

\begin{remark}
For the Cauchy problem, the pressure $P$ is unique up to an additive function of time.
\end{remark}

\begin{remark}
It is worth noting that Theorems \ref{thnp1}--\ref{th2cp1} impose no restrictions on the size of the initial data in the constant-viscosity case.
These results therefore establish the global existence of strong solutions to the two-dimensional nonhomogeneous incompressible nematic liquid crystal flows with constant viscosity.
\end{remark}

\begin{remark}
When $d$ is a constant unit vector, the system \eqref{nlc} reduces to the nonhomogeneous incompressible Navier-Stokes system.
In this case, Theorem~\ref{thnp1} is consistent with the result of Huang-Wang \cite{HW} and, in addition, provides exponential decay estimates.
Moreover, Theorem~\ref{thcp1} coincides with the result of L\"u-Shi-Zhong \cite{LSZ}.
Thus, Theorems \ref{thnp1} and \ref{thcp1} generalize the previous results \cite{HW,LSZ} to nonhomogeneous incompressible nematic liquid crystal flows.
\end{remark}

\begin{remark}
It should be mentioned that, for the Cauchy problem with density-dependent viscosity in the presence of vacuum, the crucial Stokes estimates in Lemma~\ref{ste} (see also \cite[Lemma 2.1]{HW}) are no longer available.
This problem is left for future work.
\end{remark}

We now make some comments on the analysis in this paper.
Note that the local existence of strong solutions is guaranteed by Lemma~\ref{lct}.
To extend these solutions globally in time, it suffices to establish global a priori estimates in suitable higher-order norms.
We first derive the standard energy estimate \eqref{np11}.
Although this provides an $L^2(\OM \times (0,T))$ bound on $\Delta d + |\na d|^2 d$, it does not directly yield an estimate for $\na^2 d$.
However, such an estimate is crucial for handling the strong coupling term $u \cdot \na d$ and the nonlinear terms $\na d \cdot \Delta d$ and $|\na d|^2 d$.
To overcome this difficulty, we make full use of the geometric constraint $|d|=1$ by representing $d$ in polar coordinates as $d=(\cos \theta, \sin \theta )$ (see Lemma~\ref{jzb}).
A direct computation gives
\be\nonumber\ba
|\Delta d + |\na d|^2 d|^2 = |\Delta \theta|^2, \quad  |\na d|^2 = |\na \theta|^2.
\ea\ee
Combining these identities with the boundary condition \eqref{bjtj1} and Poincar\'e's inequality, we deduce that $\|\na^2 d\|_{L^2}$ can be controlled by $\|\Delta \theta\|_{L^2}$ (see \eqref{np113} and \eqref{np115}).
Together with the energy estimate \eqref{np11}, this yields the desired $L^2(\OM \times (0,T))$ estimate for $\na^2 d$.
Moreover, to obtain time-independent estimates in bounded domains, we use the fact that $d$ satisfies the parabolic equation $\eqref{nlc}_3$ and observe that the boundary condition $n \cdot \na d = 0$ on $\p \OM$ implies that, for $j =1,2$,
\be\nonumber
\na d^j \cdot \na (n \cdot \na d^j) = 0 \quad \text{ on } \p \OM.
\ee
This, together with the standard parabolic estimates, yields a uniform $L^4$-bound for $\na d$ (see \eqref{np02}), which is essential for controlling the strong nonlinear coupling between the fluid velocity and the orientation field.

For the Cauchy problem with vacuum far-field density, the presence of vacuum at infinity leads to the loss of integrability of $u$ and $u_t$.
To address this difficulty, we introduce weighted estimates in Lemma~\ref{jqgj}, inspired by \cite{LLL,LX}.
Building on these estimates and following arguments analogous to those in \cite{LSZ,LLTZ,LXZ}, we derive the required a priori estimates and establish the large-time decay rates for the gradients of the velocity field, the orientation field, and the pressure.
For the case of non-vacuum far-field density, we exploit the integrability of $\n-\tilde{\n}$ together with the Gagliardo-Nirenberg inequality to establish the Poincar\'e-type inequality \eqref{b6}.
By suitably modifying the arguments used in the vacuum case, we derive the required a priori estimates and decay rates.

The rest of this paper is organized as follows: In Section 2, we state some known results and inequalities required for our later analysis.
Sections 3, 4, and 5 are devoted to deriving the necessary a priori estimates for the bounded-domain problem, the Cauchy problem with vacuum far-field density, and the Cauchy problem with non-vacuum far-field density, respectively.
Finally, Theorems \ref{thnp1}--\ref{th2cp1} are proved in Section~6.

\section{Preliminaries}

In this section, we recall some known results and elementary inequalities
that will be used frequently in the subsequent analysis.

We first present the following local existence result for strong solutions, 
which can be proved using methods similar to those in \cite{LJK,LLTZ,LS}.
\begin{lemma}\la{lct}
Let the initial data $\left(\n _0,u_0,d_0\right)$ satisfy \eqref{npsol1}, \eqref{cpsol1}, or \eqref{2cpsol1}, respectively.
Then there exist a time $T>0$ and a unique strong solution $(\n,u,P,d)$ to the corresponding initial-boundary value problem or Cauchy problem in $\OM \times (0,T)$ satisfying \eqref{npsol2}, \eqref{cpsol3}, or \eqref{2cpsol3}, respectively. 
\end{lemma}

Next, we will frequently use the following Gagliardo-Nirenberg inequalities (see \cite{NI}).
\begin{lemma}\la{gn1}
Assume that $\OM=\rr$ or that $\OM$ is a bounded smooth domain in $\rr$.
For $1 \le r < p < \infty$, there exists a positive constant $C$ depending only on $p$, $r$, and $\OM$
such that if $f \in H_0^1(\OM)$, or $f \in H^1(\OM)$ with $\ol{f}=0$, or $f \in (H^1(\OM))^2$ with $f \cdot n|_{\p \OM}=0$, or $f \in L^r(\rr) \cap D^1(\rr)$, then
\be\ba\la{gn11}
\| f \|_{L^p} \le C \| f \|^{\frac{r}{p}}_{L^r} \| \na f \|^{1-\frac{r}{p}}_{L^2}.
\ea\ee
\end{lemma}

The following regularity results for the Stokes equations are crucial for deriving higher-order a priori estimates in bounded domains; see \cite{HW} for the proof.
\begin{lemma}\la{ste}
Let $\OM$ be a bounded smooth domain in $\rr$,
and let $\hat{\n}$, $\underline{\mu}$, and $\bar{\mu}$ be positive constants with $p \in (2,\infty)$.
Assume that $\rho \in W^{1, p}$ with $0 \leq \rho \leq \hat{\rho}$,
and that $\mu(\n)$ satisfies \eqref{vc} with $\underline{\mu} \leq \mu(\rho) \leq \bar{\mu}$ on $[0, \hat{\rho}]$.
Let $(u, P) \in H_0^1 \times L^2$ be the unique weak solution to the boundary value problem
\begin{equation}\label{sts}
-\operatorname{div}(2 \mu(\rho) Du)+\nabla P=F, \quad
\operatorname{div} u=0 \textnormal { in } \Omega, \quad
\int P d x=0.
\end{equation}
Then the following regularity estimates hold.

$(1)$ If $F \in L^2$, then $(u, P) \in H^2 \times H^1$ and there exists a positive constant $C$ depending only on $\Omega$, $\hat{\rho}$, $\underline{\mu}$, $\bar{\mu}$, and $p$ such that
\begin{equation}\label{S1}
	\begin{aligned}
		\|u\|_{H^2} & \leq C\|F\|_{L^2}\left(1+\|\nabla \mu(\rho)\|_{L^p}\right)^{\frac{p}{p-2}}, \\
		\|P\|_{H^1} & \leq C\|F\|_{L^2}\left(1+\|\nabla \mu(\rho)\|_{L^p}\right)^{\frac{2 p-2}{p-2}}.
	\end{aligned}
\end{equation}

$(2)$ If $F \in L^r$ for some $r \in(2, p)$, then $(u, P) \in W^{2, r} \times W^{1, r}$ and
\begin{equation}\label{stegj}
	\begin{aligned}
		\|u\|_{W^{2, r}} &\leq C\|F\|_{L^r}\left(1+\|\nabla \mu(\rho)\|_{L^p}\right)^{\frac{p r}{2(p-r)}}, \\
		\|P\|_{W^{1, r}} &\leq C\|F\|_{L^r}\left(1+\|\nabla \mu(\rho)\|_{L^p}\right)^{1+\frac{p r}{2(p-r)}},
	\end{aligned}
\end{equation}
where the constant $C$ depends only on $\Omega$, $\hat{\rho}$, $\underline{\mu}$, $\bar{\mu}$, $p$, and $r$.
\end{lemma}

The following elliptic estimates can be found in \cite{ADN,GT}.
\begin{lemma}\la{tygj}
Let $k \ge 0$ be an integer and $\OM$ be a bounded smooth domain in $\rr$.
Then for any $f\in H^{k+2}(\OM)$ satisfying $\frac{\p f}{\p n}=0$ on $\p \OM$,
there exists a positive constant $C$ depending only on $\OM$ and $k$ such that
\be\la{tygj1}\ba
\| \na^{k+2} f \|_{L^2} \le C \| \Delta f \|_{W^{k,2}}.
\ea\ee
\end{lemma}

By zero extension of $u$ outside $\OM$, we can obtain the following lemma, which was first established in \cite{D}.
\begin{lemma}\la{logbd}
Let $\OM \subset \rr$ be a bounded domain.
For $\n \in L^\infty(\OM)$ with $0 \le \n \le \rs $, and $u \in H_0^1(\OM)$,
there exists a positive constant $C$ depending only on $\rs$ and $\OM$ such that
\be\la{logbd1}\ba
\| \sqrt{\n} u \|^2_{L^4} \le C \left( 1+\| \sqrt{\n} u \|_{L^2} \right)
\| \na u \|_{L^2} \log^{\frac{1}{2}} \left( 2+\| \na u \|^2_{L^2} \right).
\ea\ee
\end{lemma}

The following weighted $L^p$ estimate can be found in \cite[Theorem B.1]{L1}.
\begin{lemma}\la{WPE}
For $m \in [2,\infty)$ and $\theta \in (1+\frac{m}{2},\infty)$,
there exists a generic positive constant $C$
such that for any $v \in D^1(\rr)$,
\be\la{WPE1}\ba
\left( \int _{{\mathbb {R}^2 }} \frac{|v|^m}{e+|x|^2}(\log (e+|x|^2))^{-\theta }dx \right) ^{1/m}
\le C\Vert v\Vert _{L^2(B_1)}+C\Vert \nabla v\Vert _{L^2({\mathbb {R}^2 }) }.
\ea\ee
\end{lemma}

As a consequence of Lemma~\ref{WPE}, we have the following weighted $L^p$-estimates
for functions in $D^1(\rr)$, whose proof can be found in \cite[Lemma 2.4]{LLL}.
\begin{lemma}\la{jqgj}
Assume that $\n \in L^1(\rr) \cap L^\infty(\rr)$ is a non-negative function such that
\be\la{jqgj1}\ba
\int_{B_{N_1}} \n dx \ge M_1, \quad \| \n \|_{ L^1(\rr) \cap L^\infty(\rr) } \le M_2,
\ea\ee
for positive constants $M_1$, $M_2$, and $N_1 \ge 1$ with $B_{N_1} \subset \rr$.
Then for $\ep>0$ and $\eta>0$, there is a positive constant $C$ depending only on
$\ep$, $\eta$, $M_1$, $M_2$, and $N_1$ such that every $v \in D^1(\rr)$
satisfies
\be\la{jqgj2}\ba
\| v \bar{x}^{-\eta} \|_{ L^{(2+\ep)/\tilde{\eta}} (\rr) }
\le C \| \sqrt{\n} v \|_{L^2(\rr)} + C \| \na v \|_{L^2(\rr)},
\ea\ee
where $\tilde{\eta}=\min\{1,\eta\}$.
\end{lemma}

Let $\mathcal{H}^1(\rr)$ and $\mathcal{BMO}(\rr)$ denote the standard Hardy and $\mathcal{BMO}$ spaces (see \cite[Chapter IV]{SEM}).
Then the following well-known facts play an important role in our analysis of the Cauchy problem; proofs can be found in \cite{CLMS}.
\begin{lemma}\la{hm}
$(i)$ There exists a positive constant $C$ such that
\be\ba\nonumber
\| E \cdot B \|_{\mathcal{H}^1(\rr)} \le C \| E \|_{L^2(\rr)} \| B \|_{L^2(\rr)},
\ea\ee
for all $E \in L^2(\rr)$ and $B \in L^2(\rr)$ satisfying
\be\ba\nonumber
\div E =0, \quad \na^{\bot} \cdot B=0 \ \text{ in } D'(\rr).
\ea\ee
$(ii)$ There exists a positive constant $C$ such that for all $v \in D^1(\rr)$,
\be\ba\nonumber
\| v \|_{\mathcal{BMO}(\rr)} \le C \| \na v \|_{L^2(\rr)}.
\ea\ee
\end{lemma}

Finally, we show that any unit vector field in a two-dimensional simply connected domain admits a polar representation.
\begin{lemma}\la{jzb}
Let $\OM \subset \rr$ be a simply connected domain.
Assume that $d=(d^1,d^2) \in H^2_{\textnormal{loc}}(\OM;\rr)$ with $|d|=1$.
Then there exists a function $\theta \in H^2_{\textnormal{loc}}(\OM)$ such that
\be\la{jzb1}\ba
d=(\cos \theta, \sin \theta).
\ea\ee
\begin{proof}
Define the vector field
\be\ba\nonumber
Q = (Q_1,Q_2),
\ea\ee
where
\be\ba\la{jzb101}
Q_1 \triangleq -d^2 \frac{\p d^1}{\p x_1} + d^1 \frac{\p d^2}{\p x_1}, \quad
Q_2 \triangleq -d^2 \frac{\p d^1}{\p x_2} + d^1 \frac{\p d^2}{\p x_2}.
\ea\ee
A straightforward computation shows that
\be\ba\la{jzb11}
\frac{\p Q_2}{\p x_1} - \frac{\p Q_1}{\p x_2}
= 2 \left( \frac{\p d^1}{\p x_1} \frac{\p d^2}{\p x_2} - \frac{\p d^1}{\p x_2} \frac{\p d^2}{\p x_1} \right).
\ea\ee
The condition $|d|=1$ implies
\be\ba\la{jzb12}
d^1\frac{\p d^1}{\p x_1} + d^2\frac{\p d^2}{\p x_1} = 0, \quad
d^1\frac{\p d^1}{\p x_2} + d^2\frac{\p d^2}{\p x_2} = 0.
\ea\ee
Using \eqref{jzb11}, \eqref{jzb12}, and the fact that $|d|=1$, we arrive at
\be\ba\nonumber
\curl Q = \frac{\p Q_2}{\p x_1} - \frac{\p Q_1}{\p x_2} = 0.
\ea\ee

Since $\OM$ is simply connected, Poincar\'e's theorem (see \cite{Ma}) together with $d\in H^2_{\textnormal{loc}}(\OM;\rr)$ and (\ref{jzb101}) implies that there exists $\theta \in H^2_{\textnormal{loc}}(\OM)$ such that
\be\ba\la{jzb13}
\na \theta = -d^2 \na d^1 + d^1 \na d^2.
\ea\ee

It remains to show that $\theta$ may be chosen such that \eqref{jzb1} holds.
The Sobolev embedding theorem ensures that $d$ and $\theta$ have continuous representatives.
Moreover, since $\theta$ is determined by \eqref{jzb13} up to an additive constant, we may choose the constant such that for some $x_0 \in \OM$,
\be\ba\nonumber
d(x_0)=(\cos \theta(x_0), \sin \theta(x_0)).
\ea\ee
Set
\be\ba\la{jzb14}
U_1 \triangleq d^1 - \cos \theta, \quad U_2 \triangleq d^2 - \sin \theta.
\ea\ee
A straightforward calculation combined with \eqref{jzb13} leads to
\be\ba\nonumber
\na U_1 = \left( 1 - d^2 \sin \theta \right) \na d^1 + d^1 \sin \theta \na d^2, \quad
\na U_2 = d^2 \cos \theta \na d^1 + \left( 1 - d^1 \cos \theta \right) \na d^2,
\ea\ee
which together with \eqref{jzb14} yields
\be\ba\la{jzb15}
U_1 \na U_1 + U_2 \na U_2
= \left( 1 - d^2 \sin \theta - d^1 \cos \theta \right) \left( d^1 \na d^1 + d^2 \na d^2 \right).
\ea\ee
In view of \eqref{jzb12}, we have
\be\ba\nonumber
\na \left( U^2_1 + U^2_2 \right)=0.
\ea\ee
Since $U_1(x_0)=U_2(x_0)=0$, we conclude that $U_1(x)=U_2(x)=0$ for all $x \in \OM$, which together with \eqref{jzb14} gives \eqref{jzb1} and completes the proof.
\end{proof}
\end{lemma}

\section{A priori estimates on bounded domains}
In this section, we will establish some necessary a priori bounds for local strong solutions $(\n,u,P,d)$ to the problem \eqref{nlc}--\eqref{bjtj1}
whose existence is guaranteed by Lemma~\ref{lct}.
Thus, let $T>0$ be a fixed time and $(\n,u,P,d)$ be a strong solution to \eqref{nlc}--\eqref{bjtj1} on $\OM \times (0,T]$ with initial data $(\n_0,u_0,d_0)$ satisfying \eqref{npsol1}.

We first state the standard energy estimates.

\begin{lemma}\la{Nl1}
Let $(\n,u,P,d)$ be a strong solution of \eqref{nlc}--\eqref{bjtj1} on $\OM\times (0,T]$. Then there exists a positive constant $C$ depending only on $\underline{\mu}$,
$\hat{\n}$, $\| u_0 \|_{H^1}$, $\| \na d_0 \|_{L^2}$, and $\OM$ such that 
\be\ba\la{np01}
\sup_{0\leq t\leq T} \left( \| \n \|_{L^\infty}
+ \| \sqrt{\n} u \|^2_{L^2} + \| \na d \|^2_{L^2} \right)
+ \int_0^T \left( \| \na u \|^2_{L^2} + \| \na^2 d \|^2_{L^2} \right) dt
\leq C,
\ea\ee
and
\be\la{np001}\ba
\sup_{0 \le t \le T} e^{\eta_0 t} \left( \| \sqrt{\n} u \|^2_{L^2} + \| \na d \|^2_{L^2} \right)
+ \int_0^T e^{\eta_0 t} \left( \| \na u \|^2_{L^2} + \| \na^2 d \|^2_{L^2} \right) dt \le C,
\ea\ee
where $\eta_0$ is a positive constant depending only on $\underline{\mu}$, $\hat{\n}$, and $\OM$.
\end{lemma}
\begin{proof}
First, combining the continuity equation $\eqref{nlc}_1$ with the divergence-free condition $\div u=0$
and following the argument in \cite{L1}, we obtain
\be\la{np0011}\ba
\| \n(t) \|_{L^p} = \| \n_0 \|_{L^p} \quad \text{ for all } p \in [1,\infty]
\text{ and } t \in [0,T].
\ea\ee
Next, we multiply $\eqref{nlc}_2$ by $u$, integrate by parts over $\OM$, and apply $\eqref{nlc}_1$ to derive
\be\la{np011}\ba
\frac{1}{2}\frac{d}{dt} \int \n |u|^2 dx + 2\int\mu(\n) |Du|^2 dx
= - \int u \cdot \na d \cdot \Delta d dx.
\ea\ee
On the other hand, multiplying $\eqref{nlc}_3$ by $-(\Delta d + |\na d|^2 d)$,
integrating by parts over $\OM$, and using $|d|=1$, we have
\be\la{np012}\ba
\frac{1}{2}\frac{d}{dt} \int |\na d|^2 dx + \int |\Delta d + |\na d|^2 d|^2 dx
=\int u \cdot \na d \cdot \Delta d dx.
\ea\ee
Adding \eqref{np011} and \eqref{np012}, integrating the resulting equation over $(0,T)$, and using $\mu(\rho)\geq \underline{\mu}>0$, we obtain
\be\la{np11}\ba
	&\sup_{0 \le t \le T} \left( \|\sqrt{\rho} u\|_{L^2}^2+\|\nabla d\|_{L^2}^2 \right)+\int_0^T\left(\|\nabla u\|_2^2+\left\|\Delta d+|\nabla d|^2 d\right\|_{L^2}^2\right) d t \leq C.
\ea\ee

To obtain an $L^2(\OM \times (0,T))$ estimate for $\na^2 d$, we exploit the condition $|d|=1$ by introducing a polar coordinate representation for $d$.
By Lemma~\ref{jzb}, there exists a function $\theta \in H^2(\OM)$ such that
\be\la{np13}\ba
d=(\cos \theta, \sin \theta).
\ea\ee
A direct calculation shows that for $i,j=1,2$,
\be\la{np103}\ba
\p_i d = (-\sin \theta, \cos \theta) \p_i \theta
= (-d^2,d^1) \p_i \theta,
\ea\ee
and
\be\la{np14}\ba
\p_{ij} d = (-\sin \theta, \cos \theta) \p_{ij} \theta -(\cos \theta, \sin \theta) \p_{i} \theta \p_{j} \theta.
\ea\ee
Combining \eqref{np13} and \eqref{np14}, we derive
\be\la{np16}\ba
\Delta d = (-\sin \theta, \cos \theta) \Delta \theta - d |\na \theta|^2 = (-\sin \theta, \cos \theta) \Delta \theta - d |\na d|^2,
\ea\ee
where we have used
\be\la{np17}\ba
|\na d|^2 = |\na \theta|^2,
\ea\ee
due to \eqref{np103} and $|d|=1$.

Hence, we obtain
\be\la{np18}\ba
\Delta d + |\na d|^2 d = (-\sin \theta, \cos \theta) \Delta \theta,
\ea\ee
which gives
\be\la{np19}\ba
|\Delta d + |\na d|^2 d|^2 = |\Delta \theta|^2.
\ea\ee
In view of \eqref{np11} and \eqref{np19}, we have
\be\la{np110}\ba
\int_0^T \| \Delta \theta \|^2_{L^2} dt \leq C.
\ea\ee
Moreover, we deduce from \eqref{np103} that
\be\la{np112}\ba
\na \theta = - d^2 \na d^1 + d^1 \na d^2.
\ea\ee
Combining this with the boundary condition $n \cdot \na d = 0$ on $\p \OM$,
we find that $n \cdot \na \theta = 0$ on $\p \OM$.
This, together with Lemma~\ref{tygj}, yields
\be\la{np113}\ba
\| \na^2 \theta \|^2_{L^2} \le C \| \Delta \theta \|^2_{L^2},
\ea\ee
which along with \eqref{np110} leads to
\be\la{np114}\ba
\int_0^T \| \na^2 \theta \|^2_{L^2} dt
\leq C \int_0^T \| \Delta \theta \|^2_{L^2} dt \leq C.
\ea\ee
In addition, it follows from \eqref{gn11}, \eqref{np11}, \eqref{np14}, and \eqref{np17} that
\be\la{np115}\ba
\| \na^2 d \|^2_{L^2} & \le C \left( \| \na \theta \|^4_{L^4} + \| \na^2 \theta \|^2_{L^2} \right) \\
& \le C \left( \| \na^2 \theta \|^2_{L^2} \| \na \theta \|^2_{L^2} + \| \na^2 \theta \|^2_{L^2} \right)
\le C \| \na^2 \theta \|^2_{L^2},
\ea\ee
which together with \eqref{np114} gives
\be\la{np116}\ba
\int_0^T \| \na^2 d \|^2_{L^2} dt
\leq C \int_0^T \| \na^2 \theta \|^2_{L^2} dt \leq C.
\ea\ee
From \eqref{np0011}, \eqref{np11}, and \eqref{np116}, we deduce that \eqref{np01} holds.
It remains to prove \eqref{np001}.
Adding \eqref{np011} and \eqref{np012}, we obtain
\be\la{npe11}\ba
\frac{1}{2}\frac{d}{dt} \int \left( \n |u|^2 + |\na d|^2 \right) dx
+ \int \left(2 \mu(\rho)|D u|^2 + |\Delta d + |\na d|^2 d|^2 \right) dx = 0,
\ea\ee
which along with \eqref{np19} yields
\be\la{npe12}\ba
\frac{1}{2}\frac{d}{dt} \int \left( \n |u|^2 + |\na d|^2 \right) dx
+ \underline{\mu} \| \na u \|^2_{L^2} + \| \Delta \theta \|^2_{L^2} \leq 0.
\ea\ee
	Using \eqref{np0011}, \eqref{np17}, and Poincar\'e's inequality,  we deduce that for some positive constants $C_{P}$ and $C_{N}$,
	\be\la{npe13}\ba
	\| \sqrt{\n} u \|^2_{L^2} + \| \na d \|^2_{L^2}
	&\leq \hat{\n}\|u\|^2_{L^2}+\|\nabla \theta\|^2_{L^2}\\
	&\leq C_{P} \hat{\n}\|\na u\|_{L^2}^2
	+C_{N} \|\Delta\theta\|_{L^2}^2 \\
	&\leq \eta_0^{-1}\left(
	\underline{\mu}\|\na u\|_{L^2}^2+\|\Delta\theta\|_{L^2}^2\right),
	\ea\ee
	where in the second inequality we have used $\na \theta \cdot n =0$ on $\p \OM$, and $\eta_0$ is chosen as $$\eta_0\triangleq\min\left\{\frac{\underline{\mu}}{C_{P}\hat{\rho}},\frac{1}{C_{N}}\right\}.$$ 
	Combining \eqref{npe12} and \eqref{npe13} leads to
\be\la{npe14}\ba
\frac{d}{dt} \left( \| \sqrt{\n} u \|^2_{L^2} + \| \na d \|^2_{L^2} \right)
+ \eta_0 \left( \| \sqrt{\n} u \|^2_{L^2} + \| \na d \|^2_{L^2} \right)
+  \underline{\mu} \| \na u \|^2_{L^2} + \| \Delta \theta \|^2_{L^2}  \le 0,
\ea\ee
which implies that
\be\la{npe15}\ba
\frac{d}{dt} \left( e^{\eta_0 t} \left( \| \sqrt{\n} u \|^2_{L^2} + \| \na d \|^2_{L^2} \right) \right)
+ e^{\eta_0 t} \left( \underline{\mu} \| \na u \|^2_{L^2} + \| \Delta \theta \|^2_{L^2} \right) \le 0.
\ea\ee
Integrating the above inequality over $(0,T)$ and applying \eqref{np113} and \eqref{np115},
we obtain \eqref{np001} and complete the proof of Lemma~\ref{Nl1}.
\end{proof}

\begin{lemma}\la{Nl2}
There exists a positive constant $C$ depending only on $\underline{\mu}$, $\hat{\n}$, $\| u_0 \|_{H^1}$, $\| \na d_0 \|_{H^1}$, and $\OM$ such that
\be\ba\la{np02}
\sup_{0\leq t\leq T} \| \na d \|^4_{L^4} \leq C.
\ea\ee
\end{lemma}
\begin{proof}
First, applying $\na$ to $\eqref{nlc}_3$ yields
\be\la{np21}\ba
\na d_t-\Delta \na d=-\na (u\cdot\na d)+\na(|\na d|^2 d).
\ea\ee
Multiplying \eqref{np21} by $4 |\na d|^2 \na d$ and integrating by parts over $\OM$ lead to
\be\la{np22}\ba
& \frac{d}{dt} \int |\na d|^4 dx
-4 \int |\na d|^2 \p_j d^i \p_k \p_k \p_j d^i dx \\
& \le -4 \int |\na d|^2 \p_j d^i u^k \p_k \p_j d^i dx
+ C \int |\na d|^4 |\na u| + |\na d|^6 dx \\
& \quad + 4 \int |\na d|^2 \p_j d^i \p_j(|\na d|^2) d^i dx \\
& \le C \| \na d \|^6_{L^6} + C \| |\na d|^2 \|^2_{L^4} \| \na u \|_{L^2},
\ea\ee
where in the second inequality we have used $|d|=1$.

Integration by parts over $\OM$ shows that
\be\la{np23}\ba
& -4 \int |\na d|^2 \p_j d^i \p_k \p_k \p_j d^i dx \\
& = -4 \int_{\p \OM} |\na d|^2 \p_j d^i n_k \p_k\p_j d^i ds
+ 4 \int \p_k( |\na d|^2 \p_j d^i ) \p_k (\p_j d^i) dx \\
& = 4 \int_{\p \OM} |\na d|^2 \p_j d^i \p_j n_k \p_k d^i ds
+ 4 \int |\na d|^2 |\na^2 d|^2 dx + 2 \int |\na (|\na d|^2)|^2 dx,
\ea\ee
where in the second equality we have used the following fact:
\be\la{np24}\ba
\p_j d^i \p_j n_k \p_k d^i + \p_j d^i n_k \p_k\p_j d^i 
= \na d^i \cdot \na (n \cdot \na d^i) = 0 \quad \text{ on } \p \OM,
\ea\ee
due to $n \cdot \na d = 0$ on $\p \OM$.

Substituting \eqref{np23} into \eqref{np22} and using \eqref{gn11}, \eqref{np01}, and Young's inequality, we derive
\be\ba\nonumber
& \frac{d}{dt} \int |\na d|^4 dx + 4 \int |\na d|^2 |\na^2 d|^2 dx \\
& \le C \| |\na d|^4 \|_{L^1(\p \OM)} + C \| \na d \|^6_{L^6}
+ C \| |\na d|^2 \|^2_{L^4} \| \na u \|_{L^2} \\
& \le C \| |\na d|^4 \|_{W^{1,1}(\OM)} + C \| \na d \|^4_{L^4} \| \na^2 d \|^2_{L^2}
+ C \| |\na d|^2 \|_{L^2} \| |\na d|^2 \|_{H^1} \| \na u \|_{L^2} \\
& \le C \| \na d \|^4_{L^4} + C \int |\na d|^3 |\na^2 d| dx
+ C \| \na d \|^4_{L^4} \| \na^2 d \|^2_{L^2} \\
& \quad + C \| \na d \|^4_{L^4} \| \na u \|^2_{L^2}
+ \int |\na d|^2 |\na^2 d|^2 dx \\
& \le 2 \int |\na d|^2 |\na^2 d|^2 dx + C \| \na d \|^4_{L^4}
+ C \| \na d \|^4_{L^4} ( \| \na^2 d \|^2_{L^2} + \| \na u \|^2_{L^2} ) \\
& \le 2 \int |\na d|^2 |\na^2 d|^2 dx + C \| \na^2 d \|^2_{L^2}
+ C \| \na d \|^4_{L^4} ( \| \na^2 d \|^2_{L^2} + \| \na u \|^2_{L^2} ),
\ea\ee
which implies that
\be\la{np25}\ba
\frac{d}{dt} \| \na d \|^4_{L^4}
\le C \| \na^2 d \|^2_{L^2}
+ C \| \na d \|^4_{L^4} ( \| \na^2 d \|^2_{L^2} + \| \na u \|^2_{L^2} ).
\ea\ee
Applying Gr\"onwall's inequality to \eqref{np25} and using \eqref{np01},
we obtain \eqref{np02} and complete the proof of Lemma~\ref{Nl2}.
\end{proof}

Next, we establish the following key a priori estimates on $(\rho,u,P,d)$.
\begin{proposition}\label{n1}
	There exists some positive constant $\varepsilon_0$ depending only on $q$, $\underline{\mu}$, $\bar{\mu}$, $\hat{\n}$, $\|u_0\|_{H^1}$, $\|\na d_0\|_{H^1}$, and $\OM$ such that if $(\rho,u,P,d)$ is a strong solution of \eqref{nlc}--\eqref{bjtj1} on $\OM\times(0, T]$ satisfying
	\begin{equation}\label{n2}
		\sup _{t \in[0, T]}\|\nabla \mu(\rho)\|_{L^q} \leq 1,
	\end{equation}
	the following estimate holds:
	\begin{equation}\label{np552}
		\sup _{t \in[0, T]}\|\nabla \mu(\rho)\|_{L^q} \leq \frac{1}{2},
	\end{equation}
	provided that $\left\|\na \mu(\n_0)\right\|_{L^q} \leq \varepsilon_0$.
\end{proposition}

Before the proof of Proposition~\ref{n1}, we derive some necessary a priori estimates stated in Lemmas \ref{Nl3} and \ref{Nl4}.
Hereafter, let $\eta_0$ denote the positive constant given by Lemma~\ref{Nl1}.
\begin{lemma}\la{Nl3}
	Let $(\n,u,P,d)$ be a strong solution to \eqref{nlc}--\eqref{bjtj1} satisfying $\eqref{n2}$. Then there exists a positive constant $C$ depending only on $q$, $\underline{\mu}$, $\bar{\mu}$, $\hat{\n}$, $\| u_0 \|_{H^1}$, $\| \na d_0 \|_{H^1}$, and $\OM$ such that
\be\ba\la{np03}
\sup_{0\leq t\leq T} \left( \| \na u \|^2_{L^2} + \| \na^2 d \|^2_{L^2}  \right)
& + \int_0^T \left( \|\sqrt{\n} u_t \|^2_{L^2} + \| \na^3 d \|^2_{L^2}
+ \| \na d_t \|^2_{L^2} \right) dt \\
& + \int_0^T \left( \| \na^2 u \|^2_{L^2} + \| P \|^2_{H^1} \right) dt
\leq C,
\ea\ee
and
\be\ba\la{np003}
\sup_{0\leq t\leq T} e^{\eta_0 t} \left( \| \na u \|^2_{L^2} + \| \na^2 d \|^2_{L^2}  \right)
& + \int_0^T e^{\eta_0 t} \left( \|\sqrt{\n} u_t \|^2_{L^2} + \| \na^3 d \|^2_{L^2}
+ \| \na d_t \|^2_{L^2} \right) dt \\
& + \int_0^T e^{\eta_0 t} \left( \| \na^2 u \|^2_{L^2} + \| P \|^2_{H^1} \right) dt
\leq C.
\ea\ee
\end{lemma}
\begin{proof}
First, multiplying $\eqref{nlc}_2$ by $u_t$ and integrating by parts over $\Omega$, we derive after using \eqref{gn11} and $\eqref{nlc}_1$ that
\be\la{np31}\ba
& \frac{d}{dt} \int \mu(\rho) |D u|^2 dx + \int \rho|u_t|^2 dx \\
& = -\int\n u\cdot \na u\cdot u_t dx + \int (\mu(\rho))_t |D u|^2 dx -\int \div (\na d \odot \na d)\cdot u_t dx  \\
& = -\int\n u\cdot \na u\cdot u_t dx - \int u \cdot \nabla \mu(\rho)| D u|^2 d x  +\int  (\na d \odot \na d)\cdot\na u_t dx  \\
&\leq C\left\|\sqrt{\rho} u_t\right\|_{L^2} \| \sqrt{\rho} u\|_{L^4}\|\nabla u\|_{L^4}+C\|\nabla \mu(\rho)\|_{L^q}\|u\|_{L^{\frac{2 q}{q-2}}}\|\nabla u\|_{L^4}^2+\int  (\na d \odot \na d)\cdot\na u_t dx \\
&\leq \frac{1}{8}\left\|\sqrt{\rho} u_t\right\|_{L^2}^2 + C(\| \sqrt{\rho} u\|_{L^4}^2+\|\nabla u\|_{L^2})\|\nabla u\|_{L^2}\|\nabla u\|_{H^1} +\int  (\na d \odot \na d)\cdot\na u_t dx.
\ea\ee
Moreover, it follows from \eqref{np02} and H\"older's inequality that
\be\la{np32}\ba
\int (\na d \odot \na d)\cdot \na u_t dx
& = \frac{d}{dt}\int (\na d \odot \na d)\cdot \na u dx
- \int (\na d_t \odot \na d)\cdot\na u dx \\
& \quad - \int(\na d \odot \na d_t)\cdot\na u dx \\
& \le \frac{d}{dt} \int(\nabla d \odot \nabla d) \cdot \nabla u dx
+ C \| \nabla d_t \|_{L^2} \| \na d \|_{L^4} \| \na u \|_{L^4} \\
& \le \frac{d}{dt} \int(\nabla d \odot \nabla d) \cdot \nabla u dx
+ \frac{1}{8} \| \nabla d_t \|_{L^2}^2
+ C \|\nabla u\|_{L^2}\|\na u\|_{H^1}.
\ea\ee
Next, applying $\na$ to $\eqref{nlc}_3$ yields
\be\la{np33}\ba
\na d_t-\Delta \na d=-\na (u\cdot\na d)+\na(|\na d|^2 d).
\ea\ee

Multiplying \eqref{np33} by $\na d_t - \Delta \na d$, integrating over $\OM$, and using \eqref{tygj1},
\eqref{np02}, \eqref{gn11}, and Young's inequality, we arrive at
\be\la{np34}\ba
& \frac{d}{dt} \| \Delta d \|^2_{L^2} + \| \na d_t \|^2_{L^2} + \| \na \Delta d \|^2_{L^2} \\
& \le C \int |\na (u \cdot \na d)|^2 + |\na (|\na d|^2 d)|^2 dx \\
& \le C \int |\na u|^2 |\na d|^2 + |u|^2 |\na^2 d|^2
+ |\na^2 d|^2 |\na d|^2 + |\na d|^6 dx \\
& \le C \| \na u \|^2_{L^4} \| \na d \|^2_{L^4}
+ C \| u \|^2_{L^8} \| \na^2 d \|^2_{L^\frac{8}{3}}
+ C \| \na^2 d \|^2_{L^4} \| \na d \|^2_{L^4} + \| \na d \|^6_{L^6} \\
& \le C \| \na u \|^2_{L^4}
+ C \| \na u \|^2_{L^2} \| \na d\|_{L^4} \left( \| \na^2 d\|_{L^2} + \| \na^3 d\|_{L^2} \right) \\
& \quad + C \| \na^2 d \|_{L^2} \left( \| \na^2 d \|_{L^2} + \| \na^3 d \|_{L^2} \right)
+ C \| \na d \|^4_{L^4} \| \na^2 d \|^2_{L^2} \\
& \le \frac{1}{8} \| \na \Delta d \|^2_{L^2}
+ C \| \na u \|_{L^2}\|\na u\|_{H^1} + C \| \na u \|^4_{L^2} + C \| \na^2 d \|^2_{L^2}.
\ea\ee
The combination of \eqref{np31}, \eqref{np32}, and \eqref{np34} implies
\be\la{np35}\ba
& \frac{d}{dt} ( A_1(t) )
+ \frac{1}{2} \| \sqrt{\n} u_t \|^2_{L^2} + \frac{1}{2} \| \na d_t \|^2_{L^2} + \frac{1}{2} \| \na \Delta d \|^2_{L^2} \\
& \le C(1+\| \sqrt{\n} u \|^2_{L^4}+\|\na u\|_{L^2}) \| \na u \|_{L^2}\|\na u\|_{H^1}+C\| \na u \|^4_{L^2} + C \| \na^2 d \|^2_{L^2},
\ea\ee
where
\be\la{np36}\ba
A_1(t) \triangleq \int \left( \mu(\n) |D u|^2 + |\Delta d|^2 - (\nabla d \odot \nabla d) \cdot \nabla u \right) dx.
\ea\ee
Furthermore, from $\eqref{nlc}_2$, \eqref{np01}, \eqref{np02}, and Lemma~\ref{ste}, we conclude that
\be\nonumber\ba
& \|u \|_{H^2} + \|P\|_{H^1} \\
& \le C \left( \| \n u_t \|_{L^2} + \| \n u \cdot \na u \|_{L^2}
+ \| |\na d| |\na^2 d| \|_{L^2} \right) \\
& \le C \left( \| \sqrt{\n} u_t \|_{L^2} + \| \sqrt{\n} u \|_{L^4} \| \na u \|_{L^4} + \| \na d \|_{L^4} \| \na^2 d \|_{L^4} \right) \\
& \le C \left( \| \sqrt{\n} u_t \|_{L^2} + \| \sqrt{\n} u \|_{L^4} 
\| \na u \|^{\frac{1}{2}}_{L^2}\| \na u \|^{\frac{1}{2}}_{H^1}
+ \| \na^2 d \|_{H^1} \right) \\
& \le \frac{1}{2} \|u \|_{H^2} 
+ C \left( \| \sqrt{\n} u_t \|_{L^2} + \| \sqrt{\n} u \|^2_{L^4} \| \na u \|_{L^2}
+ \| \na^2 d \|_{L^2} + \| \na \Delta d \|_{L^2} \right),
\ea\ee
which gives
\be\la{np38}\ba
\| \na u \|_{H^1}
\le C \left( \| \sqrt{\n} u_t \|_{L^2} + \| \sqrt{\n} u \|^2_{L^4} \| \na u \|_{L^2}
+ \| \na^2 d \|_{L^2} + \| \na \Delta d \|_{L^2} \right).
\ea\ee
Substituting \eqref{np38} into \eqref{np35} and applying Young's inequality leads to
\be\la{np310}\ba
& \frac{d}{dt} ( A_1(t) )
+ \frac{1}{2} \| \sqrt{\n} u_t \|^2_{L^2} + \frac{1}{2} \| \na d_t \|^2_{L^2} + \frac{1}{2} \| \na \Delta d \|^2_{L^2} \\
& \le \frac{1}{4} \left( \| \sqrt{\n} u_t \|^2_{L^2} + \| \na \Delta d \|^2_{L^2} \right)
+ C \| \na u \|^2_{L^2} \left( 1 + \| \sqrt{\n} u \|^4_{L^4} + \| \na u \|^2_{L^2} \right) + C \| \na^2 d \|^2_{L^2}.
\ea\ee
Combining this with \eqref{logbd1} yields
\be\la{np311}\ba
& \frac{d}{dt} ( A_1(t) )
+ \frac{1}{4} \| \sqrt{\n} u_t \|^2_{L^2} + \frac{1}{4} \| \na d_t \|^2_{L^2} + \frac{1}{4} \| \na \Delta d \|^2_{L^2} \\
& \le C \| \na u \|^4_{L^2} \log(2+\| \na u \|^2_{L^2})
+ C \| \na u \|^2_{L^2} + C \| \na^2 d \|^2_{L^2}.
\ea\ee
On the other hand, Young's inequality and \eqref{np02} ensure that
\be\la{np312}\ba
\left| \int (\nabla d \odot \nabla d) \cdot \nabla u dx \right|
\le \| \na d \|^2_{L^4} \| \na u \|_{L^2}
\le \frac{1}{2} \int \mu(\n) |D u|^2 dx + C_0\|\nabla d\|^4_{L^4}.
\ea\ee
We set 
\be\la{np313}\ba
A_2(t) \triangleq A_1(t) + C_0\|\nabla d\|^4_{L^4}.
\ea\ee
By virtue of \eqref{gn11}, \eqref{tygj1}, \eqref{np01}, \eqref{np36}, \eqref{np312}, and \eqref{np313}, it holds that
\be\la{np314}\ba
\frac{1}{4} \underline{\mu} \| \na u \|^2_{L^2} + \| \Delta d \|^2_{L^2} \le  A_2(t)
\le C \left( \| \na u \|^2_{L^2} + \| \Delta d \|^2_{L^2} \right).
\ea\ee
Therefore, we deduce from \eqref{np25}, \eqref{np311}, and \eqref{np314} that
\be\la{np315}\ba
& \frac{d}{dt} ( A_2(t) )
+ \frac{1}{4} \| \sqrt{\n} u_t \|^2_{L^2} + \frac{1}{4} \| \na d_t \|^2_{L^2} + \frac{1}{4} \| \na \Delta d \|^2_{L^2} \\
& \le C \| \na u \|^2_{L^2} A_2(t) \log(2+A_2(t))
+ C \| \na u \|^2_{L^2} + C \| \na^2 d \|^2_{L^2}.
\ea\ee
Dividing \eqref{np315} by $2+A_2(t)$ yields
\be\la{np316}\ba
\frac{d}{dt} \log( 2 + A_2(t) )
\le C \| \na u \|^2_{L^2} \log(2+A_2(t))
+ C \| \na u \|^2_{L^2} + C \| \na^2 d \|^2_{L^2},
\ea\ee
which together with \eqref{tygj1}, \eqref{np01}, \eqref{np314}, \eqref{np316}, and Gr\"onwall's inequality implies that
\be\la{np317}\ba
\sup_{0\leq t\leq T} \left( \| \na u \|^2_{L^2}
+ \| \na^2 d \|^2_{L^2}  \right)
+ \int_0^T \left( \|\sqrt{\n} u_t \|^2_{L^2} + \| \na^3 d \|^2_{L^2}
+ \| \na d_t \|^2_{L^2} \right) dt
\leq C.
\ea\ee
Furthermore, in view of \eqref{tygj1}, \eqref{logbd1}, \eqref{np38}, and \eqref{np317}, we have
\be\la{np318}\ba
\int_0^T \left( \| \na^2 u \|^2_{L^2} + \| P \|^2_{H^1} \right) dt \le C.
\ea\ee

Finally, multiplying \eqref{np315} by $e^{\eta_0 t}$ and integrating over $(0,T)$,
we derive \eqref{np003} after using \eqref{np314}, \eqref{np317}, \eqref{np001}, \eqref{tygj1}, and \eqref{np38}.
The proof of Lemma~\ref{Nl3} is completed.
\end{proof}

\begin{lemma}\la{Nl4}
	Let $(\n,u,P,d)$ be a strong solution to \eqref{nlc}--\eqref{bjtj1} satisfying $\eqref{n2}$. Then there exists a positive constant $C$ depending only on $q$, $\underline{\mu}$, $\bar{\mu}$, $\hat{\n}$, $\| u_0 \|_{H^1}$, $\| \na d_0 \|_{H^1}$, and $\OM$ such that
\be\ba\la{np04}
& \sup_{0\leq t\leq T} \si \left( \| \sqrt{\n} u_{t} \|_{L^{2}}^{2}
+ \|\na d_{t}\|_{L^{2}}^{2} + \| \na^3 d \|^2_{L^2} + \| \na^2 u \|^2_{L^2} + \| P \|^2_{H^1} \right) \\
& + \int_{0}^{T} \si \left( \|\nabla u_{t}\|_{L^{2}}^{2} + \|\nabla^2 d_{t}\|_{L^{2}}^{2} \right) dt
\leq C,
\ea\ee
and
\be\ba\la{np004}
& \sup_{0 \leq t\leq T} \si e^{\eta_0 t} \left( \| \sqrt{\n} u_{t} \|_{L^{2}}^{2}
+ \|\na d_{t}\|_{L^{2}}^{2} + \| \na^3 d \|^2_{L^2} + \| \na^2 u \|^2_{L^2} + \| P \|^2_{H^1} \right) \\
& + \int_{0}^{T} \si e^{\eta_0 t} \left( \|\nabla u_{t}\|_{L^{2}}^{2} + \|\nabla^2 d_{t}\|_{L^{2}}^{2} \right) dt
\leq C,
\ea\ee
with $\si=\min\{1,t\}$.
\end{lemma}
\begin{proof}
First, differentiating $\eqref{nlc}_2$ with respect to $t$ yields
\be\la{np41}\ba
&\n u_{tt} +\n u\cdot\nabla u_{t}-\div(2\mu(\rho) Du_t)+\nabla P_{t} \\
&= -\n_{t}(u_{t}+u\cdot\nabla u)-\n u_{t}\cdot\nabla u+\div(2\mu_t(\rho)Du) -\operatorname{div}(\nabla d\odot \nabla d)_{t}.
\ea\ee
Multiplying \eqref{np41} by $u_t$ and integrating by parts over $\Omega$ lead to
\begin{equation}\label{np42}
	\begin{aligned}
		& \frac{1}{2} \frac{d}{d t} \int \rho\left|u_t\right|^2 d x+\underline{\mu}\int \left|\nabla u_t\right|^2 d x \\
		& \leq-\int \rho_t\left|u_t\right|^2 d x-\int \rho_t u \cdot\nabla u\cdot u_{t}dx-\int \rho u_t\cdot\nabla u\cdot u_{t}dx \\
		& \quad - 2 \int \mu_t(\rho) D u : D u_t d x +\int (\partial_{i} d \cdot \partial_j d)_t\partial_j u^{i}_{t} d x\triangleq\sum_{i=1}^5 I_i.
	\end{aligned}
\end{equation}

\par Now, we estimate each term on the right-hand side of \eqref{np42}.
First, with the help of $\eqref{nlc}_1$, \eqref{np03}, and Poincar\'e's inequality, we have
\begin{equation}\label{np43}
	\begin{aligned}
		I_1=-2\int \rho u\cdot\nabla u_t\cdot u_{t} dx & \leq C \|u\|_{L^8} \|\nabla u_t\|_{L^2} \|\sqrt{\n} u_t\|^{\frac{1}{2}}_{L^2} \| u_t \|^{\frac{1}{2}}_{L^4} \\
		&\leq \frac{1}{16} \underline{\mu} \|\nabla u_t\|^{2}_{L^2}+C\|\sqrt{\n} u_t\|^{2}_{L^2}.
	\end{aligned}
\end{equation} 
Next, combining $\eqref{nlc}_1$ with \eqref{gn11}, \eqref{np01}, \eqref{np03}, and Poincar\'e's inequality, we obtain
\begin{equation}\label{n61}
	\begin{aligned}
		I_{2}&\leq \int \rho|u| (|\nabla u|^2 |u_t| +|u||\nabla^2 u||u_t|+|u||\nabla u| |\nabla u_t| )dx\\
		& \leq C\|\sqrt{\n} u_t\|_{L^3} (\|\na u\|^2_{L^4} \|u\|_{L^6}+\|\na^2 u\|_{L^2} \|u\|_{L^{12}}^2) + C \|\nabla u_t\|_{L^2} \|\nabla u\|_{L^6} \|u\|_{L^6}^2 \\
		& \leq C\|\sqrt{\n} u_t\|^{\frac{1}{2}}_{L^2} \| u_t\|^{\frac{1}{2}}_{L^6} \|\nabla u\|_{H^1} + C \|\nabla u_t\|_{L^2} \|\nabla u\|_{H^1} \\
		& \leq C\|\sqrt{\n} u_t\|^{\frac{1}{2}}_{L^2} \|\nabla u_t\|^{\frac{1}{2}}_{L^2} \|\nabla u\|_{H^1} + C \|\nabla u_t\|_{L^2} \|\nabla u\|_{H^1} \\
		& \leq \frac{1}{16} \underline{\mu}\|\nabla u_t\|_{L^2}^2+C\|\nabla u\|_{H^1}^2+C\|\sqrt{\n} u_t\|_{L^2}^2.
	\end{aligned}
\end{equation}
Similarly, it holds that
\begin{equation}\label{np44}
	\begin{aligned}
		I_3&\leq C\|\sqrt{\n} u_t\|^2_{L^4} \|\nabla u\|_{L^2} \leq  C\|\sqrt{\n} u_t\|^{\frac{1}{2}}_{L^2} \|u_t\|^{\frac{3}{2}}_{L^6}
        \leq \frac{1}{16} \underline{\mu} \|\nabla u_t\|_{L^2}^2+C\|\sqrt{\n}u_t\|^2_{L^2},
	\end{aligned}
\end{equation}
and
\begin{equation}\label{np444}
	\begin{aligned}
		I_4 \leq C \int|u||\nabla \mu(\rho)||\nabla u|\left|\nabla u_t\right| d x \leq & C\|u\|_{L^{\frac{4 q}{q-2}}}\|\nabla \mu(\rho)\|_{L^q}\|\nabla u\|_{L^{\frac{4 q}{q-2}}}\left\|\nabla u_t\right\|_{L^2} \\
		\leq & \frac{1}{16} \underline{\mu} \|\nabla u_t\|_{L^2}^2+C\|\nabla u\|_{H^1}^2.
	\end{aligned}
\end{equation}
Moreover, we deduce from \eqref{gn11}, \eqref{np02}, and Young's inequality that for any $\ep>0$,
\be\la{np45}\ba
I_{5}&\leq C \|\nabla d\|_{L^{4}} \|\nabla d_{t}\|_{L^{4}} \|\nabla u_{t}\|_{L^{2}}\\
& \leq \frac{1}{16} \underline{\mu} \|\nabla u_{t}\|_{L^{2}}^{2}
+ C \|\nabla d_{t}\|_{L^{2}}\|\nabla^{2} d_{t}\|_{L^{2}} \\
& \leq \frac{1}{16} \underline{\mu} \|\nabla u_{t}\|_{L^{2}}^{2}
+ \ep \|\nabla^{2} d_{t}\|_{L^{2}}^{2} + C(\ep) \|\nabla d_{t}\|_{L^{2}}^{2}.
\ea\ee
Substituting \eqref{np43}, \eqref{n61}, \eqref{np44}, \eqref{np444}, and \eqref{np45} into \eqref{np42}, we arrive at
\be\la{np46}\ba
\frac{d}{dt} \|\sqrt{\n} u_{t}\|_{L^{2}}^{2} +\underline{\mu} \|\nabla u_{t}\|_{L^{2}}^{2}
& \leq 2 \ep \|\nabla^{2} d_{t}\|_{L^{2}}^{2}
+ C \left( \|\sqrt{\n} u_{t}\|_{L^{2}}^{2} + C \|\nabla u\|_{H^1}^2 \right)
+ C(\ep) \| \na d_t \|^2_{L^2}.
\ea\ee
Next, differentiating \eqref{np33} with respect to $t$ gives
\be\la{np47}\ba
\nabla d_{tt}-\Delta\nabla d_{t} =-\nabla (u\cdot \nabla d)_{t}+\nabla (|\nabla d|^{2}d)_{t}.
\ea\ee

Multiplying \eqref{np47} by $\na d_t$, integrating over $\OM$, and using \eqref{np01}, \eqref{np03}, and Young's inequality, we derive
\be\la{np48}\ba
& \frac{1}{2}\frac{d}{dt}\|\nabla d_{t}\|_{L^{2}}^{2} + \|\Delta d_{t}\|_{L^{2}}^{2} \\
& \leq C \int |\nabla u_{t}||\nabla d||\nabla d_{t}| dx
+ C \int |\nabla u||\nabla d_{t}|^{2} dx
+ C \int |u_{t}||\nabla d| |\nabla^{2} d_{t}| dx \\
& \quad + C \int |\nabla d|^{2}|d_{t}||\nabla^{2} d_{t}| dx
+ C \int |\nabla d||\nabla d_{t}| |\nabla^{2} d_{t}| dx \\
& \le C \|\nabla u_{t}\|_{L^{2}}\|\nabla d_{t}\|_{L^{4}}\|\nabla d\|_{L^{4}}
+ C \|\nabla u\|_{L^{2}}\|\nabla d_{t}\|_{L^{4}}^{2}
+ C \| u_t \|_{L^4} \| \na d \|_{L^4} \|\nabla^{2} d_{t}\|_{L^{2}} \\
& \quad + C \| \na^2 d_t \|_{L^2} \| \na d \|^2_{L^8} \| d_t \|_{L^4}
+ C \| \na^2 d_t \|_{L^2} \| \na d \|_{L^4} \| \na d_t \|_{L^4} \\
& \le \frac{1}{4} \|\Delta d_{t}\|_{L^{2}}^{2}
+ C \| \na d_t \|^2_{L^4} + C \| \na d \|^4_{L^8} \| d_t \|^2_{L^4}
+ \frac{C_1}{2} \|\nabla u_{t}\|_{L^{2}}^{2} \\
& \le \frac{1}{4} \|\Delta d_{t}\|_{L^{2}}^{2}
+ C \| \na d_t \|_{L^2} \| \na^2 d_t \|_{L^2} + C \| d_t \|^2_{L^4}
+ \frac{C_1}{2} \|\nabla u_{t}\|_{L^{2}}^{2} \\
& \le \frac{1}{2} \|\Delta d_{t}\|_{L^{2}}^{2}
+ C \| \na d_t \|^2_{L^2} + C \| \na^2 d \|^2_{H^1}
+ \frac{C_1}{2} \|\nabla u_{t}\|_{L^{2}}^{2},
\ea\ee
where in the last inequality we have used the following estimate:
\be\la{np49}\ba
\| d_t \|^2_{L^4} & \le C \left( \| u \cdot \na d \|^2_{L^4}
+ \| \na^2 d \|^2_{L^4} + \| |\na d|^2 \|^2_{L^4} \right) \\
& \le C \left( \| u \|^2_{L^8} \| \na d \|^2_{L^8} + \| \na^2 d \|^2_{H^1}
+ \| \na d \|^4_{L^8} \right)
\le C \| \na^2 d \|^2_{H^1},
\ea\ee
due to \eqref{gn11} and \eqref{np03}.
Thus, we have
\be\la{np410}\ba
\frac{d}{dt}\|\nabla d_{t}\|_{L^{2}}^{2} + \|\Delta d_{t}\|_{L^{2}}^{2}
\le C_1 \|\nabla u_{t}\|_{L^{2}}^{2}
+ C \| \na d_t \|^2_{L^2} + C \| \na^2 d \|^2_{H^1}.
\ea\ee

Choose a positive constant $K$ such that
$K\underline{\mu}-C_1\geq1$.
Multiplying \eqref{np46} by $K$, adding the result to \eqref{np410}, and then choosing $\ep>0$ sufficiently small yield
\be\la{np411}\ba
& \frac{d}{dt} \left( K \|\sqrt{\n} u_{t}\|_{L^{2}}^{2}
+ \| \na d_t \|^2_{L^2} \right)
+ \|\nabla u_{t}\|_{L^{2}}^{2} + \frac{1}{2}\| \Delta d_t \|^2_{L^2} \\
& \leq C \left( \|\sqrt{\n} u_{t}\|_{L^{2}}^{2}
+ \| \na d_t \|^2_{L^2}+ \|\nabla u\|_{H^1}^2
+ \| \na^2 d \|^2_{H^1} \right).
\ea\ee
Multiplying by $\si$ and integrating
over $(0,T)$, we obtain from \eqref{np01} and \eqref{np03} that
\be\ba\la{np412}
\sup_{0\leq t\leq T} \si \left( \| \sqrt{\n} u_{t} \|_{L^{2}}^{2}
+ \|\na d_{t}\|_{L^{2}}^{2} \right)
+ \int_{0}^{T} \si \left( \|\nabla u_{t}\|_{L^{2}}^{2} + \|\nabla^2 d_{t}\|_{L^{2}}^{2} \right) dt
\leq C.
\ea\ee
Furthermore, by virtue of \eqref{np33}, \eqref{gn11}, \eqref{tygj1}, \eqref{np38}, and Young's inequality, it holds that
\be\ba\nonumber
\| \nabla \Delta d \|_{L^{2}}^{2}
\leq & C \left( \|\nabla d_{t}\|_{L^{2}}^{2} + \||\nabla u||\nabla d|\|_{L^{2}}^{2}
+ \||u||\nabla^{2}d|\|_{L^{2}}^{2} + \||\nabla d|^{3}\|_{L^{2}}^{2}
+ \||\nabla^{2}d||\nabla d|\|_{L^{2}}^{2} \right) \\
\leq & C \left( \|\nabla d_{t}\|_{L^{2}}^{2} + \|\nabla u\|_{L^{4}}^{2} \|\nabla d\|_{L^{4}}^{2}
+ \| u \|_{L^{4}}^{2} \| \nabla^{2} d \|^2_{L^{4}}
+ \|\nabla d\|_{L^{6}}^{6} + \| \na d \|^2_{L^4} \| \na^2 d \|^2_{L^4} \right) \\
\leq & C \left( \|\nabla d_{t}\|_{L^{2}}^{2}
+ \| \nabla u \|_{L^{2}} \| \nabla u \|_{H^1}
+ \|\nabla^{2}d\|_{L^{4}}^{2} + \| \na d \|^2_{H^1} \right) \\
\leq & C \left( \|\nabla d_{t}\|_{L^{2}}^{2}
+ \| \sqrt{\n} u_t \|^2_{L^2} + \| \na u \|^2_{L^2}
+ \| \na^3 d \|_{L^2} ( \| \na u \|_{L^2} + \| \na^2 d \|_{L^2} )
+ \| \na d \|^2_{H^1} \right) \\
\leq & \frac{1}{2} \| \nabla \Delta d \|_{L^{2}}^{2}
+ C \left( \|\nabla d_{t}\|_{L^{2}}^{2} + \| \sqrt{\n} u_t \|^2_{L^2}
+ \| \na u \|^2_{L^2} + \| \na d \|^2_{H^1} \right),
\ea\ee
which together with \eqref{tygj1} leads to
\be\la{np413}\ba
\|\nabla^{3}d\|_{L^{2}}^{2}
\le C \left( \|\nabla d_{t}\|_{L^{2}}^{2} + \| \sqrt{\n} u_t \|^2_{L^2}
+ \| \na u \|^2_{L^2} + \| \na d \|^2_{H^1} \right).
\ea\ee
In addition, it follows from \eqref{logbd1}, \eqref{np38}, \eqref{np03}, and \eqref{np413} that
\be\la{np414}\ba
\| u \|^2_{H^2} + \| P \|^2_{H^1}
\le C \left( \|\nabla d_{t}\|_{L^{2}}^{2} + \| \sqrt{\n} u_t \|^2_{L^2}
+ \| \na u \|^2_{L^2} + \| \na d \|^2_{H^1} \right).
\ea\ee
This combined with \eqref{np412}, \eqref{np413} and \eqref{np414} gives \eqref{np04}.

On the other hand, multiplying \eqref{np411} by $\si e^{\eta_0 t}$, choosing $\ep$ suitably small, integrating over $(0,T)$, and using \eqref{np001}, \eqref{np003}, \eqref{np413}, and \eqref{np414}, we obtain \eqref{np004} and complete the proof of Lemma~\ref{Nl4}.
\end{proof}

With Lemmas \ref{Nl3} and \ref{Nl4}, we are in a position to prove Proposition~\ref{n1}.

\begin{proof}[Proof of Proposition $\ref{n1}$]
	  Note that $\eqref{nlc}_1$ implies $\mu(\rho)$ satisfies
	$$
	\partial_{t}(\mu(\rho))+u\cdot\nabla \mu(\rho)=0.
	$$
	The standard calculations show that
	\begin{equation}\label{np51}
		\frac{d}{d t}\|\nabla \mu(\rho)\|_{L^q}^q \leq C\|\nabla u\|_{L^{\infty}}\|\nabla \mu(\rho)\|_{L^q}^q.
	\end{equation}
	Moreover, using \eqref{gn11}, \eqref{stegj}, \eqref{np01}, \eqref{np03}, \eqref{np004}, and H\"older's inequality, we derive for $r\in(2,\min\{6,q\})$ that
	\be\la{np52}\ba
	\|\nabla u\|_{L^{\infty}} & \leq C\|u\|_{W^{2,r}} \\
	& \leq C \left( \|\rho u_t\|_{L^r} + \|\rho u \cdot \nabla u\|_{L^r} + \| |\nabla d| |\nabla^2 d| \|_{L^r} \right) \\
	& \leq C \|\rho u_t\|_{L^2}^{\frac{6-r}{2r}} \|\rho u_t\|_{L^6}^{\frac{3r-6}{2r}}
	+ C \| u \|_{L^{2r}} \|\nabla u\|_{L^{2r}} + C \|\nabla d\|_{L^{2r}} \|\nabla^2 d\|_{L^{2r}} \\
	& \leq C \|\sqrt{\rho} u_t\|_{L^2}^{\frac{6-r}{2r}} \|\nabla u_t\|_{L^2}^{\frac{3r-6}{2r}}
	+ C \|\nabla u\|_{H^1}^2 + C \| \nabla^2 d \|^2_{H^1} \\
	& \leq C \left(\si e^{\eta_0 t} \|\nabla u_t\|_{L^2}^2\right)
	+ C (\si e^{\eta_0 t})^{-\frac{2r}{r+6}} (\si e^{\eta_0 t} \|\sqrt{\rho} u_t\|^2_{L^2})^{\frac{6-r}{r+6}}\\
	&~~+ C \|\nabla u\|_{H^1}^2 + C \| \nabla^2 d \|^2_{H^1}.
	\ea\ee
    Integrating \eqref{np52} over $(0,T)$ and using \eqref{np01}, \eqref{np03}, and \eqref{np004}, we obtain
	\be\la{np53}\ba
	\int_0^T \|\nabla u\|_{L^{\infty}} dt \leq C.
	\ea\ee
	Combining this with \eqref{np51} and Gr\"onwall's inequality, we obtain
	$$
	\|\nabla \mu(\rho)(t)\|_{L^q} \leq C\left\|\nabla \mu\left(\rho_0\right)\right\|_{L^q} \cdot \exp \left\{\int_0^T\|\nabla u\|_{L^{\infty}} d t\right\} \leq C_2\left\|\nabla \mu\left(\rho_0\right)\right\|_{L^q},
	$$
	where $C_2$ is independent of $T$.
	
	Setting $\varepsilon_0=1/(2 C_2)$, we conclude that \eqref{np552} holds, which completes the proof of Proposition~\ref{n1}.
\end{proof}

\begin{lemma}\label{Nl5}
Let $(\n,u,P,d)$ be a strong solution to \eqref{nlc}--\eqref{bjtj1} satisfying $\eqref{n2}$. For $r\in[2,q)$, there exists a positive constant $C$ depending only on $q$, $r$, $\underline{\mu}$, $\bar{\mu}$, $\| \n_0 \|_{W^{1,q}}$, $\| u_0 \|_{H^1}$, $\| \na d_0 \|_{H^1}$, and $\OM$ such that
\be\la{np05}\ba
\sup _{0 \leq t \leq T} \|\rho\|_{W^{1, q}}
+ \int_0^T \si e^{\eta_0 t} \left( \| \nabla^2 u \|_{L^r}^2 + \| P \|_{W^{1,r}}^2 + \| \na^4 d \|^2_{L^2} \right) dt \leq C.
\ea\ee
\end{lemma}
\begin{proof}
First, by virtue of $\eqref{nlc}_1$ and \eqref{np53}, we obtain
\be\la{np54}\ba
\sup _{0 \leq t \leq T} \|\nabla \rho\|_{L^q} \leq C.
\ea\ee
Furthermore, we deduce from \eqref{S1}, \eqref{stegj}, \eqref{np01} and \eqref{np03} that for any $r\in[2,q),$
\be\la{np57}\ba
\|\nabla^2 u\|_{L^r} + \| P \|_{W^{1,r}}
& \leq C \left( \|\rho u_t\|_{L^r} + \|\rho u \cdot \nabla u\|_{L^r}
+ \||\nabla d| |\nabla^2 d| \|_{L^r} \right) \\
& \leq C \left( \|\nabla u_t\|_{L^2} + \| u \|_{L^{2r}} \|\nabla u\|_{L^{2r}}
+ \|\nabla d\|_{L^{2r}} \|\nabla^2 d\|_{L^{2r}} \right) \\
& \leq C \left( \|\nabla u_t\|_{L^2} + \|\nabla u\|_{H^1} + \|\nabla^2 d\|_{L^2} + \|\nabla^3 d\|_{L^2} \right),
\ea\ee
which along with \eqref{np001}, \eqref{np003} and \eqref{np004} implies
\be\la{np58}\ba
\int_0^T \si e^{\eta_0 t} \left( \|\nabla^2 u\|_{L^r}^2 + \| P \|_{W^{1,r}}^2 \right) dt \leq C.
\ea\ee
Then, applying the $\na$ operator to \eqref{np33} leads to
\be\la{np59}\ba
\nabla^2 \Delta d = \nabla^2( d_t + u \cdot \na d - |\nabla d|^2 d).
\ea\ee
Combining this with \eqref{gn11}, \eqref{tygj1}, \eqref{np03}, and Young's inequality, we obtain, for $r_*=(q+2)/2\in(2,q)$, 

\be\la{np510}\ba
\|\nabla^2 \Delta d\|_{L^2}
& \leq C \left( \|\nabla^2 d_t\|_{L^2} + \|\nabla^2 (u \cdot \nabla d)\|_{L^2}
+ \|\nabla^2 (|\nabla d|^2 d) \|_{L^2} \right) \\
& \leq C \|\nabla^2 d_t\|_{L^2} + C \|\nabla d\|_{L^{\frac{2r_*}{r_*-2}}} \|\nabla^2 u\|_{L^{r_*}}
+ C \|\nabla u\|_{L^4} \|\nabla^2 d\|_{L^4} + C \|u\|_{L^4} \|\nabla^3 d\|_{L^4} \\
& \quad + C \|\nabla^2 d\|_{L^4}^2 + C \|\nabla d\|_{L^4} \|\nabla^3 d\|_{L^4}
+ C \|\nabla d\|_{L^8}^2 \|\nabla^2 d\|_{L^4} \\
& \leq C \|\nabla^2 d_t\|_{L^2} + C \|\nabla^2 u\|_{L^{r_*}}
+ C \|\nabla u\|_{L^2}^{\frac{1}{2}} \|\nabla u\|_{H^1}^{\frac{1}{2}} \|\nabla^2 d\|_{L^2}^{\frac{1}{2}} \|\nabla^2 d\|_{H^1}^{\frac{1}{2}} \\
& \quad + C \|\nabla^3 d\|_{L^2}^{\frac{1}{2}} \|\nabla^3 d\|_{H^1}^{\frac{1}{2}} + C \| \na^2 d \|_{L^2} \| \na^2 d \|_{H^1} + C \| \na^2 d \|_{H^1} \\
& \leq \frac{1}{2} \| \nabla^2 \Delta d \|_{L^2} + C \|\nabla^2 d_t\|_{L^2}
+ C \|\nabla u\|_{H^1} + C \|\nabla^2 u\|_{L^{r_*}} + C \| \na^2 d \|_{H^1},
\ea\ee
which yields
\be\la{np511}\ba
\|\nabla^2 \Delta d\|_{L^2} \leq C \|\nabla^2 d_t\|_{L^2}
+ C \|\nabla u\|_{H^1} + C \|\nabla^2 u\|_{L^{r_*}} + C \| \na^2 d \|_{H^1}.
\ea\ee
Therefore, we conclude from \eqref{np511}, \eqref{np001}, \eqref{np003}, \eqref{np004}, and \eqref{np58} with $r=r_*$ that
\be\la{np512}\ba
\int_0^T \si e^{\eta_0 t} \| \na^4 d \|^2_{L^2} dt
& \le C \int_0^T \si e^{\eta_0 t} \left( \| \na^2 \Delta d \|^2_{L^2} + \| \na^2 d \|^2_{H^1} \right) dt \le C.
\ea\ee
Combining this with \eqref{np01}, \eqref{np54}, and \eqref{np58}, we obtain
\eqref{np05} and complete the proof of Lemma~\ref{Nl5}.
\end{proof}

\section{A priori estimates for the Cauchy problem with vacuum far-field density}

In this section, we establish some necessary a priori estimates for local strong solutions $(\n,u,P,d)$ to the Cauchy problem \eqref{nlc}, \eqref{cztj}, \eqref{bjtj3} with $\tilde{\n}=0$.
The existence of such solutions is guaranteed by Lemma~\ref{lct}.
Let $T>0$ be a fixed time, and let $(\n,u,P,d)$ be a local strong solution to
\eqref{nlc}, \eqref{cztj}, \eqref{bjtj3} with $\tilde{\n}=0$ in $\rr \times (0,T]$, corresponding to initial data $(\n_0,u_0,d_0)$ satisfying \eqref{cpsol1}.
For simplicity, we assume that $\mu=1$.

Since $\rho_0\in L^1$ and
$\int_{\mathbb{R}^2} \rho_0\, dx=M_0>0,$
there exists a positive constant $N_0$ such that
\be\la{rho00}\ba
\int_{B_{N_0}} \rho_0 d x \geq \frac{1}{2}\int_{\mathbb{R}^2} \rho_0\, dx=\frac{M_0}{2}.
\ea\ee

We start with the following standard energy estimate for $(\n,u,d).$ 

\begin{lemma}\la{cl1}
There exists a positive constant $C$ depending only on
$\| \sqrt{\n_0} u_0 \|_{L^2}$, $\| \n_0 \|_{L^1 \cap L^\infty}$, $\|d_0-\mathbf{1}\|_{L^2}$, and $\| \na d_0 \|_{L^2}$ such that
\be\ba\la{cp01}
\sup_{0\leq t\leq T} \left( \| \n \|_{L^1 \cap L^\infty}
+ \| \sqrt{\n} u \|^2_{L^2} +\| \na d \|^2_{L^2} \right)+ \int_0^T \left( \| \na u \|^2_{L^2} +\| \na^2 d \|^2_{L^2} \right) dt
\leq C,
\ea\ee
and
\be\ba\la{cpd}
\sup_{0\leq t\leq T}\|d-\mathbf{1}\|_{L^2}^2
+\int_0^T\|\nabla d\|_{L^2}^2dt\leq C.
\ea\ee
\end{lemma}

\begin{proof}
	First, the estimate \eqref{cp01} follows from the same argument as in Lemma~\ref{Nl1}. It therefore remains to prove
	\eqref{cpd}.
	
	Next, taking the $L^2$ inner product of $\eqref{nlc}_3$ with
	$d-\mathbf{1}$, integrating over $\mathbb R^2$, we obtain
\be\ba\nonumber
\frac{1}{2}\frac{d}{dt}\|d-\mathbf{1}\|_{L^2}^2+\|\nabla d\|_{L^2}^2
&=\int |\nabla d|^2d\cdot(d-\mathbf{1}) dx \\
&\leq \|\nabla d\|_{L^4}^2\|d-\mathbf{1}\|_{L^2}\\
&\leq \|\nabla d\|_{L^2}\|\na^2 d\|_{L^2}\|d-\mathbf{1}\|_{L^2}\\
&\leq \frac{1}{2}\|\nabla d\|_{L^2}^2
+C\|\nabla^2 d\|_{L^2}^2\|d-\mathbf{1}\|_{L^2}^2,
\ea\ee
where we have used the condition that $\operatorname{div}u=0$ and $|d|=1.$ Consequently,
$$\frac{d}{dt}\|d-\mathbf{1}\|_{L^2}^2+\|\nabla d\|_{L^2}^2
\leq C\|\nabla^2 d\|_{L^2}^2\|d-\mathbf{1}\|_{L^2}^2.$$
Combining this inequality with \eqref{cp01} and applying Gr\"onwall's inequality, we obtain \eqref{cpd}.
This completes the proof of Lemma \ref{cl1}.
\end{proof}

\begin{lemma}\la{cl2}
There exists a positive constant $C$ depending only on
$\| \sqrt{\n_0} u_0 \|_{L^2}$, $\|\n_0\|_{L^1\cap L^\infty}$,
$\| \na u_0 \|_{L^2}$, and $\| \na d_0 \|_{H^1}$ such that for $m=0,1$,
\be\ba\la{cp02}
\sup_{0\leq t\leq T} t^m \left( \| \na u \|^2_{L^2} + \| \na^2 d \|^2_{L^2} \right)
+ \int_0^T t^m \left( \| \sqrt{\n} \dot{u} \|^2_{L^2} + \| \na^3 d \|^2_{L^2} \right) dt
\leq C.
\ea\ee
\end{lemma}
\begin{proof}
First, multiplying $\eqref{nlc}_2$ by $\dot{u}$, we obtain
\be\la{cp21}\ba
\int \rho |\dot{u}|^{2} dx
= & \int \Delta u \cdot \dot{u} dx
-\int \nabla P\cdot\dot{u} dx
-\int \div(\nabla d\odot\nabla d) \cdot\dot{u} dx \\
\triangleq & I_1+I_2+I_3.
\ea\ee
Integrating by parts and using \eqref{gn11} yields
\be\la{cp22}\ba
I_1
= & \int \Delta u \cdot (u_{t}+u\cdot\nabla u) dx
=-\int \nabla u : (\nabla u_{t}+\nabla (u\cdot\nabla u)) dx \\
\leq & -\frac{1}{2}\frac{d}{dt} \|\nabla u\|_{L^{2}}^{2}
-\int \partial_{i} u^{j}\partial_{i}u^{k}\partial_{k}u^{j} dx \\
\leq & -\frac{1}{2}\frac{d}{dt} \|\nabla u\|_{L^{2}}^{2} + C \|\nabla u\|_{L^{3}}^{3}
\leq - \frac{1}{2}\frac{d}{dt} \|\nabla u\|_{L^{2}}^{2} + C\|\nabla u\|_{L^{2}}^{2} \|\nabla^{2} u\|_{L^{2}}.
\ea\ee
Since $\mathcal{BMO}$ is the dual space of $\mathcal{H}^1$ (see \cite{SEM}), we apply Lemma~\ref{hm} to derive
\be\la{cp23}\ba
I_{2} = & -\int \nabla P \cdot(u_{t}+u\cdot\nabla u) dx
= \int P \nabla \cdot (u\cdot\nabla u) dx \\
= & \int P \partial_{i} u^{j} \partial_{j} u^{i} dx
\leq C \|P\|_{\mathcal{BMO}} \|\partial_{i}u\cdot\nabla u^{i}\|_{\mathcal{H}^{1}}
\le C \| \na P \|_{L^2} \| \na u \|^2_{L^2},
\ea\ee
where we have used $\div(\p_i u)=0$ and $\na^{\bot} \cdot (\na u^i)=0$ for $i=1,2$.

For $I_3$, integration by parts gives
\be\la{cp24}\ba
I_{3} = &\int (\nabla d\odot\nabla d)\cdot\nabla u_{t} dx
-\int \operatorname{div}(\nabla d\odot\nabla d)\cdot(u\cdot\nabla u) dx \\
= & \frac{d}{dt}\int (\nabla d\odot\nabla d)\cdot\nabla u dx
-\int (\nabla d_{t}\odot\nabla d)\cdot\nabla u dx
-\int (\nabla d\odot\nabla d_{t})\cdot\nabla u dx \\
& -\int\operatorname{div}(\nabla d\odot\nabla d)\cdot(u\cdot\nabla u)dx \\
\triangleq & \frac{d}{dt}\int (\nabla d\odot\nabla d)\cdot\nabla udx
+I_{31} + I_{32} + I_{33}.
\ea\ee
Applying $\na$ to $\eqref{nlc}_3$ shows that
\be\la{cp25}\ba
\na d_t - \Delta \na d = - \na(u \cdot \na d) + \na (|\na d|^2 d).
\ea\ee
Combining this with H\"older's inequality implies
\be\la{cp26}\ba
I_{31} & = - \int \left[ \left( \Delta\nabla d-\nabla(u\cdot \nabla d)
+\nabla(|\nabla d|^{2}d) \right) \odot\nabla d \right] \cdot \nabla udx \\
& = -\int [(\Delta\nabla d-\nabla u\cdot \nabla d
+ \nabla(|\nabla d|^{2}d))\odot \nabla d]\cdot\nabla udx 
+ \int u^{k}\partial_{k}\partial_{i}d\cdot \partial_{j}d \partial_{j}u^{i}dx \\
& \le C\|\nabla^{3} d\|_{L^{2}} \|\nabla d\|_{L^{6}} \|\nabla u\|_{L^{3}}
+ C\|\nabla d\|_{L^{6}}^{2}\|\nabla u\|_{L^{3}}^{2}
+ C\|\nabla d\|_{L^{6}}^{4} \|\nabla u\|_{L^{3}} \\
& \quad + C\|\nabla^{2}d\|_{L^{3}}\|\nabla d\|_{L^{6}}^{2} \|\nabla u\|_{L^{3}}
+ \int u^{k}\partial_{k}\partial_{i}d\cdot \partial_{j}d \partial_{j}u^{i}dx.
\ea\ee
Similarly, we also have
\be\la{cp27}\ba
I_{32} & \le C\|\nabla^{3} d\|_{L^{2}} \|\nabla d\|_{L^{6}} \|\nabla u\|_{L^{3}}
+ C\|\nabla d\|_{L^{6}}^{2}\|\nabla u\|_{L^{3}}^{2}
+ C\|\nabla d\|_{L^{6}}^{4} \|\nabla u\|_{L^{3}} \\
& \quad + C\|\nabla^{2}d\|_{L^{3}}\|\nabla d\|_{L^{6}}^{2} \|\nabla u\|_{L^{3}}
+ \int u^{k}\partial_{i}d\cdot \partial_{k} \partial_{j}d \partial_{j} u^{i}dx.
\ea\ee
In addition, integration by parts and the condition $\div u=0$ yield
\be\la{cp28}\ba
I_{33} & = \int\partial_{i} d \cdot\partial_{j} d \partial_{j}(u^{k}\partial_{k} u^{i})dx \\
& = \int\partial_{i} d \cdot\partial_{j} d \partial_{j} u^{k}\partial_{k} u^{i} dx
+ \int\partial_{i} d \cdot\partial_{j} d u^{k} \partial_{k} \partial_{j} u^{i} dx \\
& \le C \|\nabla d\|_{L^{6}}^{2}\|\nabla u\|_{L^{3}}^{2}
- \int u^{k}\partial_{k}\partial_{i}d\cdot \partial_{j}d \partial_{j}u^{i}dx
- \int u^{k}\partial_{i}d\cdot \partial_{k} \partial_{j}d \partial_{j} u^{i}dx.
\ea\ee
Putting \eqref{cp26}--\eqref{cp28} into \eqref{cp24} and using \eqref{gn11} and Young's inequality, we conclude that
\be\la{cp29}\ba
I_{3} \le & \frac{d}{dt}\int (\nabla d\odot\nabla d)\cdot\nabla udx
+ C\|\nabla^{3} d\|_{L^{2}} \|\nabla d\|_{L^{6}} \|\nabla u\|_{L^{3}}
+ C\|\nabla d\|_{L^{6}}^{2}\|\nabla u\|_{L^{3}}^{2} \\
& + C\|\nabla d\|_{L^{6}}^{4} \|\nabla u\|_{L^{3}}
+ C\|\nabla^{2}d\|_{L^{3}}\|\nabla d\|_{L^{6}}^{2} \|\nabla u\|_{L^{3}} \\
\leq & \frac{d}{dt}\int (\nabla d\odot\nabla d)\cdot\nabla u dx
+ \frac{\varepsilon}{4} \|\nabla^{3} d\|_{L^{2}}^{2}
+ C(\ep) \left( \|\nabla u\|_{L^{3}}^{3} + \|\nabla d\|_{L^{6}}^{6}
+ \|\nabla^{2} d\|_{L^{3}}^{\frac{3}{2}} \|\nabla d\|_{L^{6}}^{3} \right) \\
\leq & \frac{d}{dt}\int (\nabla d\odot\nabla d)\cdot\nabla udx +\frac{\varepsilon}{4} \|\nabla^{3} d\|_{L^{2}}^{2}
+ C(\ep) \left( \|\nabla u\|_{L^{3}}^{3} + \|\nabla d\|_{L^{6}}^{6}
+ \|\nabla^{3} d\|_{L^{2}}^{\frac{3}{4}} \|\nabla d\|_{L^{6}}^{\frac{15}{4}} \right) \\
\leq & \frac{d}{dt} \int (\nabla d\odot\nabla d)\cdot\nabla udx +\frac{\varepsilon}{2} \|\nabla^{3} d\|_{L^{2}}^{2}
+ C(\ep) \left( \|\nabla u\|_{L^{3}}^{3} + \|\nabla d\|_{L^{6}}^{6} \right) \\
\leq & \frac{d}{dt}\int (\nabla d \odot \nabla d)\cdot\nabla udx +\frac{\varepsilon}{2} \|\nabla^{3} d\|_{L^{2}}^{2}
+ C(\ep) \left( \|\nabla u\|_{L^{2}}^{2} \|\nabla^{2} u\|_{L^{2}}
+ \|\nabla d\|_{L^{2}}^{2}\|\nabla^{2} d\|_{L^{2}}^{4} \right).
\ea\ee
Furthermore, note that $(\n,u,P,d)$ satisfies the following Stokes system
\be\nonumber\ba
\begin{cases}
- \Delta u + \na P = -\n \dot{u} - \div(\nabla d \odot \nabla d), \quad & x \in \rr, \\
\div u = 0, \quad & x \in \rr, \\
u(x) \to 0, \quad & |x| \to \infty.
\end{cases}
\ea\ee
The standard $L^p$-estimate for the Stokes equation (see \cite{TR}) shows that for any $p \in (1,\infty)$,
\be\la{qkjste}\ba
\| \na^2 u \|_{L^p} + \| \na P \|_{L^p} \le C ( \| \n \dot{u} \|_{L^p} + \| |\na d| |\na^2 d| \|_{L^p} ).
\ea\ee
Substituting \eqref{cp22}, \eqref{cp23}, and \eqref{cp29} into \eqref{cp21}
and applying \eqref{gn11}, \eqref{qkjste}, and Young's inequality, we obtain
\be\la{cp210}\ba
& \frac{d}{dt} A(t) + \| \sqrt{\n} \dot{u} \|_{L^{2}}^{2} \\
& \leq \frac{\varepsilon}{2}\|\nabla^{3} d\|_{L^{2}}^{2}
+ C(\ep) \|\nabla^{2} d\|_{L^{2}}^{4}
+C(\ep) (\|\nabla^{2} u\|_{L^{2}}+\|\nabla P\|_{L^{2}})\|\nabla u\|_{L^{2}}^{2} \\
& \leq \frac{\varepsilon}{2} \|\nabla^{3} d\|_{L^{2}}^{2}
+ \ep \left( \| \sqrt{\n} \dot{u}\|_{L^{2}}^{2}+\||\nabla d||\nabla^{2}d|\|_{L^{2}}^{2} \right)
+ C(\ep) \left( \|\nabla^{2} d\|_{L^{2}}^{4} + \|\nabla u\|_{L^{2}}^{4} \right) \\
& \leq \frac{\varepsilon}{2} \|\nabla^{3} d\|_{L^{2}}^{2}
+ \ep \| \sqrt{\n} \dot{u}\|_{L^{2}}^{2}
+ C(\ep) \left( \|\nabla d\|_{L^{4}}^{2} \|\nabla^{2}d \|_{L^{4}}^{2}
+ \|\nabla^{2} d\|_{L^{2}}^{4} + \|\nabla u\|_{L^{2}}^{4} \right) \\
& \leq \frac{\varepsilon}{2}\|\nabla^{3} d\|_{L^{2}}^{2}
+ \ep \| \sqrt{\n} \dot{u}\|_{L^{2}}^{2}
+ C(\ep) \|\nabla d\|_{L^{2}} \|\nabla^{2}d\|_{L^{2}}^{2}\|\nabla^{3}d \|_{L^{2}} \\
& \quad + C(\ep) \left( \|\nabla^{2} d\|_{L^{2}}^{4} + \|\nabla u\|_{L^{2}}^{4} \right) \\
& \leq \varepsilon \|\nabla^{3} d\|_{L^{2}}^{2}
+ \varepsilon \| \sqrt{\n} \dot{u} \|_{L^{2}}^{2}
+ C(\ep) \left( \|\nabla^{2} d\|_{L^{2}}^{4} + \|\nabla u\|_{L^{2}}^{4} \right),
\ea\ee
where
\be\la{cp211}\ba
A(t) = \frac{1}{2} \|\nabla u\|_{L^{2}}^{2}-\int (\nabla d \odot \nabla d)\cdot\nabla u dx.
\ea\ee
It follows from \eqref{gn11}, \eqref{cp01} and Young's inequality that
\be\la{cp212}\ba
\left| \int (\nabla d\odot\nabla d)\cdot\nabla u dx \right|
\leq & C \|\nabla u\|_{L^{2}} \|\nabla d\|_{L^{4}}^{2}
\le C \|\nabla u\|_{L^{2}} \|\nabla d\|_{L^{2}} \|\nabla^2 d\|_{L^{2}} \\
\leq & C \|\nabla u\|_{L^{2}} \| \Delta d \|_{L^{2}}
\le \frac{1}{4} \|\nabla u\|_{L^{2}}^{2} + C_0 \|\Delta d\|_{L^{2}}^{2},
\ea\ee
which implies that
\be\la{cp213}\ba
\frac{1}{4} \|\nabla u\|_{L^{2}}^{2}
\leq A(t) + C_0 \| \Delta d \|_{L^{2}}^{2}
\leq \|\nabla u\|_{L^{2}}^{2} + 2 C_0 \| \Delta d \|_{L^{2}}^{2}.
\ea\ee
On the other hand, multiplying \eqref{cp25} by $-\na \Delta d$ and 
integrating by parts over $\rr$,
we derive after using \eqref{gn11} and Young's inequality that
\be\la{cp214}\ba
& \frac{1}{2}\frac{d}{dt} \|\Delta d\|_{L^{2}}^{2} + \|\nabla \Delta d\|_{L^{2}}^{2} \\
& =\int \nabla(u\cdot\nabla d)\cdot\nabla\Delta d dx
-\int \nabla (|\nabla d|^{2} d)\cdot\nabla\Delta d dx \\
& = \int (\nabla u\cdot\nabla d)\cdot \nabla\Delta d dx
+\int u^{i}\partial_{i}\partial_{j} d\cdot\partial_{j}\partial_{kk} d dx
- \int \nabla (|\nabla d|^{2} d)\cdot\nabla\Delta d\,dx \\
& = \int (\nabla u\cdot\nabla d)\cdot \nabla\Delta d dx
- \int \partial_{k} u^{i} \partial_{i}\partial_{j} d\cdot
\partial_{j} \partial_{k} d dx-\int \nabla(|\nabla d|^{2}d)\cdot\nabla\Delta d dx \\
& \leq C \|\nabla^{3} d\|_{L^{2}} \|\nabla u\|_{L^{3}} \|\nabla d\|_{L^{6}}
+ C \|\nabla u\|_{L^{3}}\|\nabla^{2} d\|_{L^{3}}^{2}
+ C \|\nabla^{3} d\|_{L^{2}} \|\nabla d\|_{L^{6}}^{3} \\
& \quad + C \|\nabla^{3} d\|_{L^{2}}\|\nabla^{2}d\|_{L^{3}} \|\nabla d\|_{L^{6}} \\
& \leq \frac{\varepsilon}{4} \|\nabla^{3} d\|_{L^{2}}^{2}
+ C(\ep) \left( \|\nabla u\|_{L^{3}}^{3} + \|\nabla d\|_{L^{6}}^{6}
+ \|\nabla^{2}d\|_{L^{3}}^{3} \right) \\
& \leq \frac{\varepsilon}{4} \|\nabla^{3} d\|_{L^{2}}^{2}
+ C(\ep) \left( \|\nabla u\|_{L^{3}}^{3}
+ \|\nabla d\|_{L^{6}}^{6} + \|\nabla^{3}d\|_{L^{2}}\|\nabla^{2} d\|_{L^{2}}^{2} \right) \\
& \leq \frac{\varepsilon}{2} \|\nabla^{3} d\|_{L^{2}}^{2}
+ C(\ep) \left( \|\nabla u\|_{L^{3}}^{3}
+ \|\nabla d\|_{L^{6}}^{6} + \|\nabla^{2} d\|_{L^{2}}^{4} \right) \\
& \leq \frac{\varepsilon}{2} \|\nabla^{3} d\|_{L^{2}}^{2}
+ C(\ep) \left( \|\nabla u\|_{L^{2}}^{2}\|\nabla^{2} u\|_{L^{2}}
+ \|\nabla d\|_{L^{2}}^{2}\|\nabla^{2}d \|_{L^{2}}^{4}
+ \|\nabla^{2} d\|_{L^{2}}^{4} \right) \\
& \leq \frac{\varepsilon}{2} \|\nabla^{3} d\|_{L^{2}}^{2}
+ \ep \left( \| \sqrt{\n} \dot{u}\|_{L^{2}}^{2}
+ \||\nabla d||\nabla^{2}d|\|_{L^{2}}^{2} \right)
+ C(\ep) \left( \|\nabla^{2} d\|_{L^{2}}^{4} + \|\nabla u\|_{L^{2}}^{4} \right) \\
& \leq \varepsilon\|\nabla^{3} d\|_{L^{2}}^{2}
+ \varepsilon \| \sqrt{\n} \dot{u}\|_{L^{2}}^{2}
+ C(\ep) \left( \|\nabla^{2} d\|_{L^{2}}^{4} + \|\nabla u\|_{L^{2}}^{4} \right).
\ea\ee
Combining \eqref{cp210}, \eqref{cp214} and \eqref{cp213} and choosing $\ep$ suitably small, we obtain
\be\la{cp215}\ba
& \frac{d}{dt} \left( A(t) + (1 + C_0) \| \Delta d \|_{L^{2}}^{2} \right)
+ \frac{1}{2} \left( \| \sqrt{\n} \dot{u} \|_{L^{2}}^{2} + \|\nabla \Delta d\|_{L^{2}}^{2} \right) \\
& \leq C \|\nabla^{2} d\|_{L^{2}}^{4} + C \|\nabla u\|_{L^{2}}^{4} \\
& \le C \left( \|\nabla^{2} d\|_{L^{2}}^{2} + \|\nabla u\|_{L^{2}}^{2} \right)
\left( A(t) + (1 + C_0) \| \Delta d \|_{L^{2}}^{2} \right),
\ea\ee
which together with \eqref{cp01} and Gr\"onwall's inequality yields
\be\ba\la{cp216}
\sup_{0\leq t\leq T} \left( \| \na u \|^2_{L^2} + \| \na^2 d \|^2_{L^2} \right)
+ \int_0^T \left( \| \sqrt{\n} \dot{u} \|^2_{L^2} + \| \na^3 d \|^2_{L^2} \right) dt
\leq C.
\ea\ee
Moreover, multiplying \eqref{cp215} by $t$ leads to
\be\ba\nonumber
& \frac{d}{dt} \left( t \left( A(t) + (1 + C_0) \| \Delta d \|_{L^{2}}^{2} \right) \right)
+ \frac{t}{2} \left( \| \sqrt{\n} \dot{u} \|_{L^{2}}^{2} + \|\nabla \Delta d\|_{L^{2}}^{2} \right) \\
& \le \left( A(t) + (1 + C_0) \| \Delta d \|_{L^{2}}^{2} \right)
+ C t \left( \|\nabla^{2} d\|_{L^{2}}^{2} + \|\nabla u\|_{L^{2}}^{2} \right)
\left( A(t) + (1 + C_0) \| \Delta d \|_{L^{2}}^{2} \right),
\ea\ee
which along with \eqref{cp01} and Gr\"onwall's inequality shows
\be\ba\la{cp217}
\sup_{0\leq t\leq T} t \left( \| \na u \|^2_{L^2} + \| \na^2 d \|^2_{L^2} \right)
+ \int_0^T t \left( \| \sqrt{\n} \dot{u} \|^2_{L^2} + \| \na^3 d \|^2_{L^2} \right) dt
\leq C.
\ea\ee
This combined with \eqref{cp216} gives \eqref{cp02} and completes the proof of Lemma~\ref{cl2}.
\end{proof}

\begin{lemma}\la{cl3}
There exists a positive constant $C$ depending only on
$\| \sqrt{\n_0} u_0 \|_{L^2}$, $\|\n_0\|_{L^1\cap L^\infty}$,
$\| \na u_0 \|_{L^2}$, and $\| \na d_0 \|_{H^1}$ such that for $m=1,2$,
\be\ba\la{cp03}
\sup_{0\leq t\leq T} t^m \left( \| \sqrt{\n} \dot{u} \|^2_{L^2}
+ \| |\na^2 d| |\na d| \|^2_{L^2} \right)
+ \int_0^T t^m \left( \| \na \dot{u} \|^2_{L^2} + \| |\Delta \na d| |\na d| \|^2_{L^2} \right) dt
\leq C,
\ea\ee
and
\be\ba\la{cp003}
\sup_{0\leq t\leq T} t^m \left( \| \na^2 u \|^2_{L^2}
+ \| \na P \|^2_{L^2} \right) \leq C.
\ea\ee
\end{lemma}
\begin{proof}
First, applying the operator $\p_t + u \cdot \na$ to $\eqref{nlc}^{j}_{2}$ yields
\be\la{cp31}\ba
\partial_{t} (\rho \dot{u}^{j}) + \div(\rho u \dot{u}^{j}) -\Delta \dot{u}^{j}
& =  -\partial_{i} (\partial_{i} u \cdot \nabla u^{j})
-\operatorname{div} (\partial_{i} u\partial_{i}u^{j})
- \partial_{t} \partial_{j} P -u\cdot\nabla\partial_{j} P \\
& \quad - \partial_{t}\partial_{i}(\partial_{i}d \cdot\partial_{j} d)
- u \cdot \nabla \partial_{i}(\partial_{i}d \cdot \partial_{j} d).
\ea\ee
Multiplying \eqref{cp31} by $\dot{u}^j$ and integrating over $\rr$ gives
\be\la{cp32}\ba
& \frac{1}{2}\frac{d}{dt} \| \sqrt{\n} \dot{u} \|_{L^{2}}^{2}
+ \|\nabla \dot{u}\|_{L^{2}}^{2} \\
& = - \int \left[\partial_{i}(\partial_{i} u\cdot\nabla u^{j})
+ \operatorname{div} (\partial_{i} u\partial_{i} u^{j})\right]\dot{u}^{j} dx
- \int (\partial_{t}\partial_{j} P+u\cdot\nabla\partial_{j} P)\dot{u}^{j} dx \\
& \quad - \left(\int \partial_{t}\partial_{i}(\partial_{i}d\cdot\partial_{j}d)\dot{u}^{j} dx
+\int u\cdot\nabla\partial_{i}(\partial_{i}d\cdot\partial_{j}d) \dot{u}^{j} dx \right) \\
& \triangleq N_1 + N_2 + N_3.
\ea\ee
Following the same arguments as in \cite{LSZ}, we have
\be\la{cp33}\ba
N_1+N_2 \leq \frac{d}{dt} \int P\partial_{j} u^{i} \partial_{i}u^{j} dx
+ C(\|P\|_{L^{4}}^{4} +\|\nabla u\|_{L^{4}}^{4}) +\frac{1}{4} \|\nabla \dot{u}\|_{L^{2}}^{2}.
\ea\ee
Integrating by parts and using \eqref{cp25} and Young's inequality, we arrive at
\be\la{cp34}\ba
N_3 = & \int \partial_{i} d_{t} \cdot \partial_{j}d \partial_{i} \dot{u}^{j} dx
+ \int \partial_{i}d \cdot \partial_{j}d_{t} \partial_{i} \dot{u}^{j} dx
- \int u^{k}\partial_{k}\partial_{i}(\partial_{i}d\cdot\partial_{j}d)\dot{u}^{j} dx \\
= & \int \partial_{i} (\Delta d-u\cdot \nabla d+|\nabla d|^{2}d) \cdot \partial_{j}d \partial_{i}\dot{u}^{j} dx \\
& + \int \partial_{i}d \cdot \partial_{j}(\Delta d-u\cdot \nabla d+|\nabla d|^{2}d) \partial_{i}\dot{u}^{j} dx
-\int u^{k}\partial_{k}\partial_{i}(\partial_{i}d\cdot\partial_{j}d)\dot{u}^{j} dx \\
=& \int (\Delta \partial_{i} d-\partial_{i}u\cdot \nabla d+\partial_{i}(|\nabla d|^{2}d))
\cdot \partial_{j}d \partial_{i}\dot{u}^{j} dx
-\int u^{k} \partial_{k}\partial_{i}d \cdot \partial_{j}d  \partial_{i}\dot{u}^{j} dx \\
& + \int \partial_{i}d\cdot(\Delta\partial_{j} d
-\partial_{j} u\cdot \nabla d+\partial_{j}(|\nabla d|^{2}d)) \partial_{i}\dot{u}^{j} dx
-\int u^{k} \partial_{i}d\cdot \partial_{k}\partial_{j}d \partial_{i}\dot{u}^{j} dx \\
& + \int \partial_{i} u^{k}\partial_{k}(\partial_{i}d\cdot\partial_{j}d)\dot{u}^{j} dx
+ \int u^{k}\partial_{k}(\partial_{i}d\cdot\partial_{j}d)\partial_{i} \dot{u}^{j} dx \\
=& \int (\Delta \partial_{i} d-\partial_{i}u\cdot \nabla d+\partial_{i}(|\nabla d|^{2}d))
\cdot \partial_{j}d \partial_{i}\dot{u}^{j} dx \\
& + \int \partial_{i}d\cdot(\Delta\partial_{j} d
-\partial_{j} u\cdot \nabla d+\partial_{j}(|\nabla d|^{2}d)) \partial_{i}\dot{u}^{j} dx
- \int \partial_{i} u^{k}\partial_{i}d\cdot\partial_{j}d \partial_{k}\dot{u}^{j} dx \\
\leq & C \|\nabla\dot{u}\|_{L^{2}} \left( \| |\Delta \nabla d| |\nabla d| \|_{L^{2}} + \|\nabla u\|_{L^{4}} \|\nabla d\|^2_{L^{8}}
+ \| \na d \|^4_{L^8} + \|\nabla d\|^2_{L^{8}} \| \na^2 d \|_{L^4} \right) \\
\leq & \frac{1}{4} \|\nabla \dot{u}\|_{L^{2}}^{2}
+ \hat{C}_1 \| |\Delta \nabla d| |\nabla d| \|^2_{L^{2}}
+ C \|\nabla u\|_{L^{4}}^{4} + C \| \nabla d \|_{L^{8}}^{8}
+ C \| \na^2 d \|^4_{L^4}.
\ea\ee
Putting \eqref{cp33} and \eqref{cp34} into \eqref{cp32} implies
\be\la{cp35}\ba
& \frac{d}{dt} \left( \frac{1}{2} \| \sqrt{\rho} \dot{u} \|_{L^{2}}^{2} 
- \int P\partial_{j} u^{i}\partial_{i}u^{j} dx \right)
+ \frac{1}{2} \|\nabla \dot{u}\|_{L^{2}}^{2}
- \hat{C}_1 \| |\Delta \nabla d| |\nabla d| \|^2_{L^{2}} \\
& \le C \left( \|P\|_{L^{4}}^{4}+\|\nabla u\|_{L^{4}}^{4}
+ \| \nabla d \|_{L^{8}}^{8} + \| \na^2 d \|^4_{L^4} \right).
\ea\ee

Next, following the approach in \cite{LLTZ}, for $a_{1},a_{2}\in \{-1,0,1\}$, we define
\be\la{cp36}\ba
(\widetilde{\nabla} d) (a_{1},a_{2})=a_{1}\partial_{1}d+a_{2}\partial_{2}d,
\quad (\widetilde{\nabla} u)(a_{1},a_{2})=a_{1}\partial_{1} u+a_{2}\partial_{2} u,
\ea\ee
which together with \eqref{cp25} yields
\be\la{cp37}\ba
\widetilde{\nabla }d_{t} -\Delta\widetilde{\nabla } d
= - \widetilde{\nabla}u\cdot\nabla d - u\cdot\nabla\widetilde{\nabla}d 
+ |\nabla d|^{2}\widetilde{\nabla}d
+ \widetilde{\nabla} |\na d|^2 d.
\ea\ee
Multiplying \eqref{cp37} by
$-4\widetilde{\nabla}d\Delta|\widetilde{\nabla} d|^{2}$
and integrating by parts over $\rr$,
we derive after using Young's inequality that
\be\ba\nonumber
&\frac{d}{dt} \|\nabla |\widetilde{\nabla }d|^{2}\|_{L^{2}}^{2}+2\|\Delta |\widetilde{\nabla}d|^{2}\|_{L^{2}}^{2}\\
=& 4 \int |\na \widetilde{\nabla} d|^2 \Delta |\widetilde{\nabla} d|^2 dx
+ 4\int (\widetilde{\nabla} u\cdot\nabla d)\cdot\widetilde{\nabla} d\Delta|\widetilde{\nabla}d|^{2} dx
+ 4\int (u\cdot\nabla \widetilde{\nabla}d)\cdot\widetilde{\nabla}d\Delta|\widetilde{\nabla }d|^{2} dx \\
&- 4\int |\nabla d|^{2}|\widetilde{\nabla}d|^{2} \Delta|\widetilde{\nabla}d|^{2} dx
- 4 \int \widetilde{\nabla} |\na d|^2 d \cdot \widetilde{\nabla}d\Delta|\widetilde{\nabla}d|^{2} dx \\
\leq & C \int \left( |\na^2 d|^2 + |\na u| |\na d|^2 \right) |\Delta|\widetilde{\nabla}d|^{2}| dx
+2\int (u\cdot\nabla |\widetilde{\nabla }d|^{2})\Delta|\widetilde{\nabla}d|^{2} dx \\
& + C \int \left( |\na d|^4 + |\na^2 d| |\na d|^2 \right) |\Delta|\widetilde{\nabla}d|^{2}| dx \\
\leq & \|\Delta|\widetilde{\nabla}d|^{2}\|_{L^{2}}^{2}
+ C( \|\nabla u\|_{L^{4}}^{4} + \| \nabla d \|_{L^{8}}^{8} + \| \na^2 d \|^4_{L^4} )
+ C \int |\na u| |\na^2 d|^2 |\na d|^2 dx \\
\leq & \|\Delta|\widetilde{\nabla}d|^{2}\|_{L^{2}}^{2}
+ C( \|\nabla u\|_{L^{4}}^{4} + \| \nabla d \|_{L^{8}}^{8} + \| \na^2 d \|^4_{L^4} ),
\ea\ee
which implies that
\be\la{cp38}\ba
&\frac{d}{dt} \|\nabla |\widetilde{\nabla }d|^{2}\|_{L^{2}}^{2}
+ \|\Delta |\widetilde{\nabla}d|^{2}\|_{L^{2}}^{2}
\leq C( \|\nabla u\|_{L^{4}}^{4} + \| \nabla d \|_{L^{8}}^{8} + \| \na^2 d \|^4_{L^4} ).
\ea\ee
Note that
\be\la{cp39}\ba
\| |\Delta \nabla d| |\nabla d| \|^2_{L^{2}}
& \le C \| \na^2 d \|^4_{L^4}
+ \| \Delta |\widetilde{\nabla}d(1,0)|^{2} \|_{L^{2}}^{2}
+ \| \Delta |\widetilde{\nabla}d(0,1)|^{2} \|_{L^{2}}^{2} \\
& \quad + \| \Delta |\widetilde{\nabla}d(1,1)|^{2} \|_{L^{2}}^{2}
+ \| \Delta |\widetilde{\nabla}d(1,-1)|^{2} \|_{L^{2}}^{2}.
\ea\ee
Next, we establish the following estimate
\be\la{cp311}\ba
C^{-1}\| |\na^2 d| |\na d| \|^2_{L^2}
\le B_1(t) \le C \| |\na^2 d| |\na d| \|^2_{L^2},
\ea\ee
where
\be\la{cp310}\ba
B_1(t) & \triangleq \|\nabla|\widetilde{\nabla}d(1,0)|^{2}\|_{L^{2}}^{2}
+ \|\nabla|\widetilde{\nabla}d(0,1)|^{2}\|_{L^{2}}^{2}
+ \|\nabla|\widetilde{\nabla}d(1,1)|^{2}\|_{L^{2}}^{2} \\
& \quad + \|\nabla|\widetilde{\nabla}d(1,-1)|^{2}\|_{L^{2}}^{2}.
\ea\ee
By Lemma~\ref{jzb}, we may write $d=(\cos\theta,\sin\theta)$.
Set $p\triangleq \nabla\theta = (p_1,p_2)$.
For $i=1,2$, we have
\be\la{ch1}\ba
\p_i d = (-\sin \theta,\cos \theta) \p_i \theta = (-\sin \theta,\cos \theta) p_i,
\ea\ee
which together with \eqref{cp36} shows that
\be\la{ch2}\ba
|\widetilde{\nabla}d(a_1,a_2)|^2=(a_1 p_1 + a_2 p_2)^2.
\ea\ee
Thus,
\be\la{ch3}\ba
B_1(t) = \| \na(p_1^2) \|_{L^2}^2 + \| \na(p_2^2) \|_{L^2}^2 + \| \na((p_1+p_2)^2) \|_{L^2}^2 + \| \na((p_1-p_2)^2) \|_{L^2}^2.
\ea\ee
A straightforward calculation yields
\be\ba\nonumber
C^{-1}\|\nabla(p\otimes p)\|_{L^2}^2
\le B_1(t) \le C\|\nabla(p\otimes p)\|_{L^2}^2.
\ea\ee
Moreover, since
\be\ba\nonumber
\partial_k(p\otimes p)=\partial_kp\otimes p+p\otimes\partial_kp,
\quad
|\partial_k(p\otimes p)|^2
=2|p|^2|\partial_kp|^2+2|p\cdot\partial_kp|^2,
\ea\ee
we obtain
\be\la{ch4}\ba
C^{-1}\int |p|^2|\nabla p|^2dx
\le B_1(t) \le C\int |p|^2|\nabla p|^2dx.
\ea\ee
On the other hand, a direct computation gives
\be\ba\nonumber
|\nabla d|^2=|p|^2, \quad
|\nabla^2d|^2=|\nabla p|^2+|p|^4,
\ea\ee
which implies that
\be\la{ch5}\ba
\||\nabla^2 d||\nabla d|\|_{L^2}^2=\int (|p|^2|\nabla p|^2+|p|^6) dx.
\ea\ee
By the Gagliardo--Nirenberg inequality and \eqref{cp01}, one gets
\be\la{ch6}\ba
\int |p|^6dx
&\le C\||p|^2\|_{L^1}\|\nabla|p|^2\|_{L^2}^2
\le C\|\nabla d\|_{L^2}^2 B_1(t) \le C B_1(t).
\ea\ee
As a consequence of \eqref{ch4}, \eqref{ch5} and \eqref{ch6}, we obtain \eqref{cp311}. 

Combining \eqref{cp38} with \eqref{cp39}, and \eqref{cp310}, we conclude that
\be\la{cp312}\ba
&\frac{d}{dt} B_1(t) + \| |\Delta \nabla d| |\nabla d| \|^2_{L^{2}}
\leq C ( \|\nabla u\|_{L^{4}}^{4} + \| \nabla d \|_{L^{8}}^{8} + \| \na^2 d \|^4_{L^4} ).
\ea\ee
We set
\be\la{cp313}\ba
B_2(t) \triangleq \frac{1}{2} \| \sqrt{\rho} \dot{u} \|_{L^{2}}^{2}
- \int P\partial_{j} u^{i}\partial_{i}u^{j} dx
+ ( \hat{C}_1 + 1 ) B_1(t).
\ea\ee
By virtue of \eqref{qkjste}, \eqref{cp311}, and Lemma~\ref{hm}, we have
\be\la{cp314}\ba
\left|\int P \partial_{j} u^{i} \partial_{i} u^{j} dx\right|
\leq & C \|P\|_{\mathcal{BMO}} \|\partial_{i}u\cdot\nabla u^{i}\|_{\mathcal{H}^{1}}
\leq C\|\nabla P\|_{L^{2}}\|\nabla u\|_{L^{2}}^{2} \\
\leq & C \left( \| \sqrt{\n} \dot{u}\|_{L^{2}} 
+\||\nabla^{2}d||\nabla d|\|_{L^{2}} \right) \|\nabla u\|_{L^{2}}^{2} \\
\leq & \frac{1}{4}\| \sqrt{\n} \dot{u}\|_{L^{2}}^{2} 
+ \frac{1}{4} B_1(t) + C\|\nabla u\|_{L^{2}}^{4},
\ea\ee
which together with \eqref{cp313} yields
\be\la{cp315}\ba
\frac{1}{4} \left( \| \sqrt{\rho} \dot{u} \|_{L^{2}}^{2}+B_1(t) \right) - C \| \na u \|^4_{L^2}
\le B_2(t) \le C \left( \| \sqrt{\rho} \dot{u} \|_{L^{2}}^{2}+B_1(t) + \| \na u \|^4_{L^2} \right).
\ea\ee
Multiplying \eqref{cp312} by $\hat{C}_1 + 1$ and adding the result to \eqref{cp35}, we arrive at
\be\la{cp316}\ba
& \frac{d}{dt} B_2(t) + \frac{1}{2} \|\nabla \dot{u}\|_{L^{2}}^{2}
+ \| |\Delta \nabla d| |\nabla d| \|^2_{L^{2}} \\
& \le C \left( \|P\|_{L^{4}}^{4}+\|\nabla u\|_{L^{4}}^{4}
+ \| \nabla d \|_{L^{8}}^{8} + \| \na^2 d \|^4_{L^4} \right) \\
& \le C \left( \| \sqrt{\n} \dot{u}\|_{L^{2}}^{2} + \|\nabla^{2} d\|_{L^{2}}^{2} \right)
\left( B_2(t) + C\|\nabla u\|_{L^{2}}^{4} \right)
+ C \|\nabla^{2}d\|_{L^{2}}^{2}\|\nabla^{3} d\|_{L^{2}}^{2},
\ea\ee
where in the second inequality we have used the following estimates:
\be\la{cp317}\ba
\|P\|_{L^{4}}^{4} + \|\nabla u\|_{L^{4}}^{4}
\leq & C (\|\nabla P\|_{L^{\frac{4}{3}}}^{4}
+ \|\nabla^{2} u\|_{L^{\frac{4}{3}}}^{4}) \\
\leq & C (\| \rho \dot{u} \|_{L^{\frac{4}{3}}}^{4}
+ \| |\nabla d| |\nabla^{2} d| \|_{L^{\frac{4}{3}}}^{4}) \\
\leq & C \| \rho \|_{L^{2}}^{2} \| \sqrt{\n} \dot{u}\|_{L^{2}}^{4}
+ C \|\nabla d\|_{L^{2}}^{4}\|\nabla^{2} d\|_{L^{4}}^{4} \\
\leq & C \| \sqrt{\n} \dot{u}\|_{L^{2}}^{2} \left( B_2(t) + C\|\nabla u\|_{L^{2}}^{4} \right)
+ C \|\nabla^{2} d\|_{L^{2}}^{2}\|\nabla^{3} d\|_{L^{2}}^{2},
\ea\ee
and
\be\la{cp318}\ba
\| \nabla d \|_{L^{8}}^{8} + \| \na^2 d \|^4_{L^4}
= & \| |\na d|^2 \|_{L^4}^4 + \| \na^2 d \|^4_{L^4} \\
\leq & C \| |\na d|^2 \|_{L^2}^2 \| \na |\na d|^2 \|_{L^2}^2
+ C \|\nabla^{2}d\|_{L^{2}}^{2}\|\nabla^{3} d\|_{L^{2}}^{2} \\
\leq & C\|\nabla d\|_{L^{2}}^{2} \|\nabla^{2} d\|_{L^{2}}^{2} \||\nabla d||\nabla^{2}d|\|_{L^{2}}^{2}
+ C \|\nabla^{2}d\|_{L^{2}}^{2}\|\nabla^{3} d\|_{L^{2}}^{2} \\
\leq & C \|\nabla^{2} d\|_{L^{2}}^{2} \left( B_2(t) + C\|\nabla u\|_{L^{2}}^{4} \right)
+ C \|\nabla^{2}d\|_{L^{2}}^{2}\|\nabla^{3} d\|_{L^{2}}^{2},
\ea\ee
due to \eqref{gn11}, \eqref{qkjste}, \eqref{cp315}, \eqref{cp02}, and H\"older's inequality.

Multiplying \eqref{cp316} by $t^m$ for $m=1,2$ yields
\be\la{cp319}\ba
& \frac{d}{dt} ( t^m B_2(t) ) + t^m \left( \frac{1}{2} \|\nabla \dot{u}\|_{L^{2}}^{2}
+ \| |\Delta \nabla d| |\nabla d| \|^2_{L^{2}} \right) \\
& \le m t^{m-1} B_2(t)
+ C t^m B_2(t) \left( \| \sqrt{\n} \dot{u}\|_{L^{2}}^{2} + \|\nabla^{2} d\|_{L^{2}}^{2} \right) \\
& \quad + C t^m \|\nabla u\|_{L^{2}}^{4}
\left( \| \sqrt{\n} \dot{u}\|_{L^{2}}^{2} + \|\nabla^{2} d\|_{L^{2}}^{2} \right)
+ C t^m \|\nabla^{2}d\|_{L^{2}}^{2}\|\nabla^{3} d\|_{L^{2}}^{2}.
\ea\ee

Moreover, we deduce from \eqref{cp315}, \eqref{cp311}, \eqref{gn11}, \eqref{cp02}, and H\"older's inequality that
\be\la{cp320}\ba
B_2(t) & \le C \left( \| \sqrt{\rho} \dot{u} \|_{L^{2}}^{2}
+ \| |\na^2 d| |\na d| \|^2_{L^2} + \| \na u \|^4_{L^2} \right) \\
& \le C \left( \| \sqrt{\rho} \dot{u} \|_{L^{2}}^{2}
+ \| \na^2 d \|^2_{L^4}  \| \na d \|^2_{L^4} + \| \na u \|^4_{L^2} \right) \\
& \le C \left( \| \sqrt{\rho} \dot{u} \|_{L^{2}}^{2}
+ \| \na^2 d \|_{L^2} \| \na^3 d \|_{L^2} \| \na d \|_{L^2}\| \na^2 d \|_{L^2} + \| \na u \|^4_{L^2} \right) \\
& \le C \left( \| \sqrt{\rho} \dot{u} \|_{L^{2}}^{2}
+ \| \na^2 d \|^4_{L^2} + \| \na^3 d \|^2_{L^2} + \| \na u \|^4_{L^2} \right).
\ea\ee
Applying Gr\"onwall's inequality to \eqref{cp319} and combining
\eqref{cp02}, \eqref{cp315} and \eqref{cp320}, we get \eqref{cp03}.

Finally, recalling from \eqref{qkjste} that
\be\la{cp321}\ba
\| \na^2 u \|^2_{L^2} + \| \na P \|^2_{L^2}
\le C \left( \| \sqrt{\n} \dot{u}\|_{L^2}^2 + \||\nabla^2 d||\nabla d|\|_{L^2}^2 \right).
\ea\ee
This, together with \eqref{cp03}, yields \eqref{cp003} and completes the proof of Lemma~\ref{cl3}.
\end{proof}

\begin{lemma}\la{cl4}
There exists a positive constant $C$ depending only on $T$,
$a$, $M_0$, $N_0$, $\|\n_0\|_{L^1\cap L^\infty}$,
$\| \sqrt{\n_0} u_0 \|_{L^2}$, $\| {\bar{x}}^a \rho_0 \|_{L^1}$,
$\| \na u_0 \|_{L^2}$, and $\| \na d_0 \|_{H^1}$ such that
\be\ba\la{cp04}
& \sup_{0\leq t\leq T} \| \rho \bar{x}^{a} \|_{L^{1}} \leq C.
\ea\ee
\end{lemma}
\begin{proof}
First, fix a smooth non-increasing function $\varphi:[0,\infty)\to[0,1]$ such that $\varphi=1$ on $[0,1]$, $\varphi=0$ on $[2,\infty)$, and $|\varphi'|\le 2$. For $N>1$, we set $\varphi_N(x)=\varphi(|x|/N)$. It is easy to see that $\varphi_{N'}\geq\varphi_N$ whenever $N'\geq N$, and that $\varphi_N(x)$ satisfies the following conditions:
\be\la{cp41}\ba
0 \le \varphi_N \le 1, \quad \varphi_N=
\begin{cases}
1,\quad &\mathrm{ if \  } |x| \le N,\\
0,\quad &\mathrm{ if \  } |x| \ge 2N,
\end{cases}
\quad |\na \varphi_N | \le 2 N^{-1}.
\ea\ee
	Multiplying $\eqref{nlc}_1$ by $\varphi_N$, integrating over $\rr$,
	and applying \eqref{cp01} and \eqref{cp41}, we obtain
	\be\ba\nonumber
	\frac{d}{dt} \int \n \varphi_N dx
	=\int \n u \cdot \na \varphi_N dx
	&\geq -2 N^{-1} M_0^{\frac{1}{2}}
	\left( \int \n |u|^2 dx \right)^{\frac{1}{2}}
	\geq -2 \hat{C} N^{-1},
	\ea\ee
	where the positive constant $\hat{C}$ depends only on
	$M_0$, $\| \sqrt{\n_0} u_0 \|_{L^2}$, and $\| \na d_0 \|_{L^2}$.
	
	Choose
	\be\ba\nonumber
	N_1\triangleq 2\left(1+N_0+8M_0^{-1}\hat{C}T\right),
	\ea\ee
	where $N_0$ is given by \eqref{rho00}. Since $N_1/2>N_0$, it follows from the above construction that $\varphi_{N_1/2}\ge\varphi_{N_0}$, and thus
	\be\la{cp43}\ba
	\inf_{0 \le t \le T} \int_{B_{N_1}} \n dx
	\ge \inf_{0 \le t \le T} \int \n \varphi_{{N_1}/2} dx
	& \ge \int \n_0 \varphi_{{N_1}/2} dx - 4 \hat{C} N_1^{-1}T \\
	& \ge \int \n_0 \varphi_{N_0} dx - 4 \hat{C} N_1^{-1}T \\
	& \ge \int_{B_{N_0}} \n_0 dx - 4 \hat{C} N_1^{-1}T
	\ge \frac{1}{4}M_0.
	\ea\ee
	Thus, $N_1$ depends only on $T$, $N_0$, $M_0$,
	$\|\bar{x}^{a}\n_0\|_{L^1}$,
	$\| \sqrt{\n_0} u_0 \|_{L^2}$, and $\| \na d_0 \|_{L^2}$.
Therefore, we deduce from \eqref{cp01}, \eqref{cp43} and \eqref{jqgj2}
that for any $v \in D^1(\rr)$,
\be\la{cp44}\ba
\| v \bar{x}^{-\eta} \|_{ L^{\frac{s}{\eta}} }
\le C \| \sqrt{\n} v \|_{L^2} + C \| \na v \|_{L^2},
\ea\ee
where $\eta \in (0,1]$ and $s>2$.

Next, multiplying $\eqref{nlc}_1$ by ${\bar{x}}^a$ and integrating over $\rr$, we use \eqref{cp01}, \eqref{cp02}, \eqref{cp44}, and H\"older's inequality to derive
\be\ba\nonumber
\frac{d}{dt} \int \n {\bar{x}}^a dx 
& \le C \int \n |u| {\bar{x}}^{a-1} \log^2(e+|x|^2) dx \\
& \le C \| \n \bar{x}^{a-1+\frac{8}{8+a}} \|_{ L^{\frac{8+a}{7+a}} }
\| u \bar{x}^{- \frac{4}{8+a} } \|_{L^{8+a}} \\
& \le C \left( 1 + \| \n \bar{x}^a \|_{L^1} \right)
\left( \| \sqrt{\n} u \|_{L^2} + \| \na u \|_{L^2} \right) \\
& \le C + C \| \n \bar{x}^a \|_{L^1},
\ea\ee
where, in the third inequality, we have used \eqref{cp44} and taken $s=4$ and $\eta=4/(8+a)$. Then, Gr\"onwall's inequality yields \eqref{cp04} and completes the proof of Lemma~\ref{cl4}.
\end{proof}

\begin{lemma}\la{cl5}
There exists a positive constant $C$ depending only on $T$,
$a$, $M_0$, $N_0$, $\|\n_0\|_{L^1\cap L^\infty}$, $\|\bar{x}^a\n_0\|_{L^1}$,
$\| \sqrt{\n_0} u_0 \|_{L^2}$, $\| \rho_0 \|_{H^{1}\cap W^{1,q}}$,
$\| \na u_0 \|_{L^2}$, and $\| \na d_0 \|_{H^1}$ such that
\be\ba\la{cp05}
\sup_{0\leq t\leq T} \| \rho \|_{H^{1}\cap W^{1,q}}
& + \int_0^T \left( \|\nabla^{2} u\|_{L^{2}}^{2} + \|\nabla P\|_{L^{2}}^{2}
+ \|\nabla^{2} u\|_{L^{q}}^{\frac{q+1}{q}} + \|\nabla P\|_{L^{q}}^{\frac{q+1}{q}} \right) dt \\
& + \int_0^T t \left( \|\nabla^{2} u\|_{L^{q}}^{2} + \|\nabla P\|_{L^{q}}^{2} \right) dt
\leq C.
\ea\ee
\end{lemma}
\begin{proof}
First, the mass equation $\eqref{nlc}_1$ implies that $\na \n$ satisfies for any $p \ge2$,
\be\la{cp51}\ba
\frac{d}{dt} \|\nabla \rho\|_{L^{p}}\leq C\|\nabla u\|_{L^{\infty}} \|\nabla \rho\|_{L^{p}}.
\ea\ee
Using \eqref{gn11}, \eqref{cp02} and \eqref{qkjste}, we have
\be\la{cp52}\ba
\|\nabla u\|_{L^{\infty}}\leq & C\|\nabla u\|_{L^{2}}^{\frac{q-2}{2(q-1)}} \|\nabla^{2} u\|_{L^{q}}^{\frac{q}{2(q-1)}}
\leq C\left( \|\rho\dot{u}\|_{L^{q}}^{\frac{q}{2(q-1)}}+\||\nabla d||\nabla^{2}d|\|_{L^{q}}^{\frac{q}{2(q-1)}}\right).
\ea\ee
Moreover, it follows from \eqref{cp01}, \eqref{cp04} and \eqref{cp44}
that for any $\eta \in (0,1]$, $s>2$ and $v \in D^1(\rr)$,
\be\la{cp53}\ba
\| \n^\eta v \|_{L^{\frac{s}{\eta}}}
& \le C \| \n^\eta \bar{x}^{ \frac{3\eta a}{4s} } \|_{L^{ \frac{4s}{3\eta} }}
\| v \bar{x}^{-\frac{3\eta a}{4s}} \|_{L^{ \frac{4s}{\eta} }} \\
& \le C \| \n \|_{L^\infty}^{ \frac{(4s-3)\eta}{4s} }
\| \n \bar{x}^a \|_{L^1}^{ \frac{3\eta}{4s} }
\left( \| \sqrt{\n} v \|_{L^2} + \| \na v \|_{L^2} \right) \\
& \le C \left( \| \sqrt{\n} v \|_{L^2} + \| \na v \|_{L^2} \right),
\ea\ee
which together with H\"older's inequality yields
\be\la{cp54}\ba
\| \rho \dot u\|_{L^q} 
& \le C\| \rho \dot u\|_{L^2}^{2(q-1)/(q^2-2)}
\| \n \dot{u} \|_{L^{q^2}}^{q(q-2)/(q^2-2)} \\ 
& \le C\| \rho \dot u\|_{L^2}^{2(q-1)/(q^2-2)} 
\left( \| \sqrt{\n} \dot u\|_{L^2}+\| \na \dot u\|_{L^2} \right)^{q(q-2)/(q^2-2)} \\ 
& \le C\| \sqrt{\n}  \dot u\|_{L^2} + C\| \sqrt{\n} \dot u\|_{L^2}^{2(q-1)/(q^2-2)}
\|\na \dot u\|_{L^2}^{q(q-2)/(q^2-2)}.
\ea\ee
Hence, we conclude from \eqref{cp03}, \eqref{cp54} and Young's inequality that
\be\la{cp55}\ba
& \int_0^T \left( \|\rho \dot u\|^{1+1 /q}_{L^q}+t\| \n \dot u\|^2_{L^q} \right) dt \\
& \le C+C \int_0^T\left( \| \sqrt{\n}  \dot u\|_{L^2}^2 +  t\|\na \dot u\|_{L^2}^2+ 
t^{-(q^3-q^2-2q)/(q^3-q^2-2q+2)} + 1 \right)dt \\ 
& \le C.
\ea\ee
On the other hand, by \eqref{gn11}, \eqref{cp01}, \eqref{cp02}, and H\"older's inequality, we derive
\be\la{cp56}\ba
\int_{0}^{T} \||\nabla d||\nabla^{2}d|\|_{L^{q}}^{2} dt
& \leq \int_{0}^{T} \|\nabla d\|^2_{L^{2q}} \|\nabla^{2}d\|^2_{L^{2q}} dt \\
& \le C \int_{0}^{T} \|\nabla d\|^2_{H^1} \|\nabla^{2}d\|^2_{H^1} dt \\
& \leq C + C \int_{0}^{T} \|\nabla^{3} d\|_{L^{2}}^{2} dt \leq C.
\ea\ee
The combination of \eqref{cp52}, \eqref{cp55}, \eqref{cp56}, and Young's inequality shows
\be\la{cp57}\ba
\int_0^T \| \na u \|_{L^\infty} dt \le C,
\ea\ee
which together with \eqref{cp51} and Gr\"onwall's inequality implies
\be\la{cp58}\ba
\sup_{0 \le t \le T} \| \na \n \|_{L^2 \cap L^q} \le C.
\ea\ee
Finally, it follows from \eqref{qkjste}, \eqref{cp01}, \eqref{cp02}, \eqref{cp55}, \eqref{cp56},
and H\"older's inequality that
\be\la{cp59}\ba
& \int_0^T \left( \|\nabla^{2} u\|_{L^{2}}^{2} + \|\nabla P\|_{L^{2}}^{2}
+ \|\nabla^{2} u\|_{L^{q}}^{\frac{q+1}{q}} + \|\nabla P\|_{L^{q}}^{\frac{q+1}{q}}
+ t \left( \|\nabla^{2} u\|_{L^{q}}^{2} + \|\nabla P\|_{L^{q}}^{2} \right)  \right) dt \\
& \leq C + C \int_0^T \left( \|\rho \dot u\|^{1+1 /q}_{L^q}+t\| \n \dot u\|^2_{L^q}
+ \| \na^3 d \|^2_{L^2} \right) dt \le C.
\ea\ee
Combining this with \eqref{cp01} and \eqref{cp58} gives \eqref{cp05},
thereby completing the proof of Lemma~\ref{cl5}.
\end{proof}

By virtue of \eqref{cp04} and \eqref{cp05}, and arguing as in
\cite[Lemma 3.6]{LSZ} (see also \cite{LX}), we can derive the
following spatial weighted estimate on the gradient of the density.
\begin{lemma}\la{cl6}
{\sloppy
There exists a positive constant $C$ depending only on $T$, $a$, $M_0$, $N_0$, $\| \sqrt{\n_0} u_0 \|_{L^2}$, $\| \na u_0 \|_{L^2}$, $\| \bar{x}^a \n_0 \|_{L^1 \cap H^{1}\cap W^{1,q}}$, and $\| \na d_0 \|_{H^1}$ such that
\par}
\be\ba\la{cp06}
\sup_{0\leq t\leq T} \| \rho \bar{x}^a \|_{L^1 \cap H^{1}\cap W^{1,q}} \leq C.
\ea\ee
\end{lemma}

\begin{lemma}\la{cl7}
There exists a positive constant $C$ depending only on $T$,
$a$, $M_0$, $N_0$, $\|\n_0\|_{L^1\cap L^\infty}$, $\|\bar{x}^a\n_0\|_{L^1}$,
$\| \sqrt{\n_0} u_0 \|_{L^2}$, $\| \na u_0 \|_{L^2}$, $\| \bar{x}^{\frac{a}{2}} \na d_0 \|_{L^2}$,
and $\| \na^2 d_0 \|_{L^2}$ such that
\be\ba\la{cp07}
\sup_{0\leq t\leq T} \|\nabla d \bar{x}^{\frac{a}{2}} \|_{L^{2}}^{2}
+ \int_{0}^{T} \|\nabla^{2} d\bar{x}^{\frac{a}{2}}\|_{L^{2}}^{2} dt
\leq C,
\ea\ee
\be\ba\la{cp007}
\sup_{0\leq t\leq T} t \|\nabla^{2} d\bar{x}^{\frac{a}{2}}\|_{L^{2}}^{2}
+ \int_{0}^{T} t\|\nabla^{3} d\bar{x}^{\frac{a}{2}}\|_{L^{2}}^{2} dt
\leq C.
\ea\ee
\end{lemma}
\begin{proof}
First, multiplying \eqref{cp25} by $\nabla d \bar{x}^{a}$ and integrating by parts over $\rr$, we derive
\be\la{cp71}\ba
\frac{1}{2}\frac{d}{dt} \|\nabla d \bar{x}^{\frac{a}{2}}\|_{L^{2}}^{2} 
+ \|\nabla^{2} d \bar{x}^{\frac{a}{2}}\|_{L^{2}}^{2}
\leq & C \int |\nabla d| |\nabla^{2} d| |\nabla \bar{x}^{a}| dx
+ C \int |\nabla u||\nabla d|^{2} \bar{x}^{a} dx \\
& + C \int |u||\nabla d|^{2} |\nabla \bar{x}^{a}| dx
+ C \int |\nabla d|^{2} |\nabla^{2} d| \bar{x}^{a}dx \\
& + C \int |\nabla d|^{4} \bar{x}^{a} dx
\triangleq \sum_{i=1}^{5}I_i.
\ea\ee
Next, we will estimate each term $I_i$ on the right-hand side of \eqref{cp71} as follows:

By applying Young's inequality, we obtain
\be\la{cp72}\ba
I_{1}\leq C \int |\nabla d| |\nabla^{2} d| \bar{x}^{a}dx\leq \frac{1}{10} \|\nabla^{2} d \bar{x}^{\frac{a}{2}}\|_{L^{2}}^{2}
+C\|\nabla d \bar{x}^{\frac{a}{2}}\|_{L^{2}}^{2}.
\ea\ee
From \eqref{gn11} and Young's inequality, we deduce that for any $\ep>0$,
\be\la{cp73}\ba
\|\nabla d \bar{x}^{\frac{a}{2}}\|_{L^{4}}^{2}
& \le C \|\nabla d \bar{x}^{\frac{a}{2}}\|_{L^{2}} \|\nabla (\nabla d \bar{x}^{\frac{a}{2}})\|_{L^{2}} \\
& \le C \|\nabla d \bar{x}^{\frac{a}{2}}\|_{L^{2}}
\left( \|\nabla^{2}d \bar{x}^{\frac{a}{2}}\|_{L^{2}}
+ \|\nabla d \nabla \bar{x}^{\frac{a}{2}}\|_{L^{2}} \right) \\
& \le \ep \|\nabla^{2}d \bar{x}^{\frac{a}{2}}\|^2_{L^{2}}
+ C(\ep) \|\nabla d \bar{x}^{\frac{a}{2}}\|^2_{L^{2}}.
\ea\ee
Combining this with H\"older's inequality yields
\be\la{cp74}\ba
I_{2} \leq C \|\nabla u\|_{L^{2}} \|\nabla d \bar{x}^{\frac{a}{2}}\|_{L^{4}}^{2}
\leq C \|\nabla d \bar{x}^{\frac{a}{2}}\|_{L^{4}}^{2}
\leq \frac{1}{10} \|\nabla^{2} d \bar{x}^{\frac{a}{2}}\|_{L^{2}}^{2}
+ C \|\nabla d \bar{x}^{\frac{a}{2}}\|_{L^{2}}^{2}.
\ea\ee
In addition, it follows from \eqref{cp44}, \eqref{cp73}, \eqref{cp01} and \eqref{cp02} that
\be\la{cp75}\ba
I_{3} \leq C\int |u||\nabla d|^2\bar{x}^{a-\frac{3}{4}}dx
& \leq C \|\nabla d \bar{x}^{\frac{a}{2}}\|_{L^{4}} \|\nabla d \bar{x}^{\frac{a}{2}}\|_{L^{2}}\|u\bar{x}^{-\frac{3}{4}}\|_{L^{4}} \\
& \leq C\|\nabla d \bar{x}^{\frac{a}{2}}\|_{L^{4}}^{2}
+ C (\| \sqrt{\n} u \|_{L^{2}}^{2}
+ \|\nabla u\|_{L^{2}}^{2})\|\nabla d \bar{x}^{\frac{a}{2}}\|_{L^{2}}^{2} \\
& \leq \frac{1}{10} \|\nabla^{2} d \bar{x}^{\frac{a}{2}}\|_{L^{2}}^{2}
+ C \|\nabla d\bar{x}^{\frac{a}{2}}\|_{L^{2}}^{2}.
\ea\ee
By virtue of \eqref{cp73}, \eqref{cp01}, \eqref{cp02}, and H\"older's inequality, it holds that
\be\la{cp76}\ba
I_{4} + I_5 \leq & C \|\nabla^{2} d \bar{x}^{\frac{a}{2}}\|_{L^{2}}
\|\nabla d\bar{x}^{\frac{a}{2}}\|_{L^{4}} \|\nabla d\|_{L^{4}}
+ \|\nabla d\|^2_{L^{4}} \|\nabla d \bar{x}^{\frac{a}{2}}\|_{L^{4}}^{2} \\
\leq & \frac{1}{20}\|\nabla^{2} d \bar{x}^{\frac{a}{2}}\|_{L^{2}}^{2}
+ C \|\nabla d\|_{H^1}^{2}\|\nabla d\bar{x}^{\frac{a}{2}}\|_{L^{4}}^{2}\\
\leq & \frac{1}{10} \|\nabla^{2} d \bar{x}^{\frac{a}{2}}\|_{L^{2}}^{2}
+ C \|\nabla d\bar{x}^{\frac{a}{2}}\|_{L^{2}}^{2}.
\ea\ee
Substituting \eqref{cp72}, \eqref{cp74}, \eqref{cp75}, and \eqref{cp76} into \eqref{cp71}, we obtain
\be\la{cp77}\ba
\frac{d}{dt} \|\nabla d \bar{x}^{\frac{a}{2}}\|_{L^{2}}^{2} 
+ \|\nabla^{2} d \bar{x}^{\frac{a}{2}}\|_{L^{2}}^{2}
\leq C \|\nabla d\bar{x}^{\frac{a}{2}}\|_{L^{2}}^{2},
\ea\ee
which together with Gr\"onwall's inequality yields \eqref{cp07}.

Furthermore, we conclude from \eqref{cp73} and \eqref{cp07} that
\be\la{cp78}\ba
\|\nabla d \bar{x}^{\frac{a}{2}}\|_{L^{4}}^{2}
& \le C \|\nabla d \bar{x}^{\frac{a}{2}}\|_{L^{2}}
\left( \|\nabla^{2}d \bar{x}^{\frac{a}{2}}\|_{L^{2}}
+ \|\nabla d \nabla \bar{x}^{\frac{a}{2}}\|_{L^{2}} \right) \\
& \le C ( 1 + \|\nabla^{2}d \bar{x}^{\frac{a}{2}}\|_{L^{2}} ).
\ea\ee

On the other hand, multiplying \eqref{cp25} by $-\Delta \nabla d \bar{x}^{a}$ and integrating by parts over $\rr$, we arrive at
\be\la{cp79}\ba
& \frac{1}{2}\frac{d}{dt} \|\nabla^{2}d\bar{x}^{\frac{a}{2}}\|_{L^{2}}^{2}
+ \|\nabla^{3} d\bar{x}^{\frac{a}{2}}\|_{L^{2}}^{2} \\
\le & -\int \nabla d_{t}\nabla^{2} d\nabla \bar{x}^{a} dx
+ \int \nabla(u\cdot\nabla d)\nabla \Delta d\bar{x}^{a} dx
-\int \nabla(|\nabla d|^{2}d)\nabla \Delta d\bar{x}^{a} dx \\
& + C \int |\nabla^{3}d| |\nabla^{2}d| |\nabla\bar{x}^{a}| dx \\
\leq & C \int |\nabla^{3}d| |\nabla^{2}d| |\nabla\bar{x}^{a}| dx
+ C \int |u| |\nabla^{2}d|^{2} |\nabla\bar{x}^{a}| dx
+ C \int |\na u| |\nabla d| |\nabla^{2} d| |\nabla\bar{x}^{a}| dx \\
& + C \int |\nabla d||\nabla^{2} d|^{2} |\nabla \bar{x}^{a}| dx
+ C \int |\nabla d|^{3}|\nabla^{2} d| |\nabla \bar{x}^{a}| dx
+ C \int |\na u| |\nabla d| |\nabla^{3} d| \bar{x}^{a} dx \\
& + C \int |\na u| |\nabla^2 d|^2 \bar{x}^{a} dx
+ C \int |\nabla d| |\nabla^2 d| |\nabla^3 d| \bar{x}^{a} dx
+ C \int |\nabla d|^{3} |\nabla^{3} d| \bar{x}^{a} dx \\
\triangleq & \sum_{i=1}^{9} J_i.
\ea\ee

Now, we will use \eqref{gn11}, \eqref{cp01}, \eqref{cp03}, and H\"older's 
inequality to estimate each term on the right-hand side of \eqref{cp79} as follows:

First, Young's inequality gives
\be\la{cp710}\ba
J_1 \le C \int |\nabla^{3}d| |\nabla^{2}d| \bar{x}^{a} dx
\le \frac{1}{8} \|\nabla^{3} d\bar{x}^{\frac{a}{2}}\|_{L^{2}}^{2}
+ C \|\nabla^{2} d \bar{x}^{\frac{a}{2}}\|_{L^{2}}^{2}.
\ea\ee
Then, using \eqref{gn11} and Young's inequality, we get for any $\ep>0$,
\be\la{cp711}\ba
\|\nabla^2 d \bar{x}^{\frac{a}{2}}\|_{L^{4}}^{2}
& \le C \|\nabla^2 d \bar{x}^{\frac{a}{2}}\|_{L^{2}} \|\nabla (\nabla^2 d \bar{x}^{\frac{a}{2}})\|_{L^{2}} \\
& \le C \|\nabla^2 d \bar{x}^{\frac{a}{2}}\|_{L^{2}}
\left( \|\nabla^{3}d \bar{x}^{\frac{a}{2}}\|_{L^{2}}
+ \|\nabla^2 d \nabla \bar{x}^{\frac{a}{2}}\|_{L^{2}} \right) \\
& \le \ep \|\nabla^{3}d \bar{x}^{\frac{a}{2}}\|^2_{L^{2}}
+ C(\ep) \|\nabla^2 d \bar{x}^{\frac{a}{2}}\|^2_{L^{2}}.
\ea\ee
Similar to \eqref{cp75}, it follows from \eqref{cp711}, \eqref{cp44}, \eqref{cp01}, \eqref{cp02}, and H\"older's inequality that
\be\la{cp712}\ba
J_{2}\leq C \int |u| |\nabla^2 d|^2 \bar{x}^{a-\frac{3}{4}}dx
\leq & C \|\nabla^2 d \bar{x}^{\frac{a}{2}}\|_{L^{4}}
\|\nabla^2 d \bar{x}^{\frac{a}{2}}\|_{L^{2}}
\|u\bar{x}^{-\frac{3}{4}}\|_{L^{4}} \\
\leq & C\|\nabla^2 d \bar{x}^{\frac{a}{2}}\|_{L^{4}}^{2}
+ C \left( \| \sqrt{\n} u \|_{L^{2}}^{2} + \|\nabla u\|_{L^{2}}^{2} \right)
\|\nabla^2 d \bar{x}^{\frac{a}{2}}\|_{L^{2}}^{2} \\
\leq & \frac{1}{8} \|\nabla^{3} d \bar{x}^{\frac{a}{2}}\|_{L^{2}}^{2}
+ C \|\nabla^2 d\bar{x}^{\frac{a}{2}}\|_{L^{2}}^{2}.
\ea\ee
Moreover, we deduce from \eqref{gn11}, \eqref{cp01}, \eqref{cp02}, \eqref{cp78}, and \eqref{cp711} that
\be\la{cp713}\ba
J_3 + J_6 + J_7
& \le C \int |\na u| |\nabla d| |\nabla^{2} d| \bar{x}^{a}
+ |\na u| |\nabla d| |\nabla^{3} d| \bar{x}^{a}
+ |\na u| |\nabla^2 d|^2 \bar{x}^{a} dx \\
& \le C \| \na d \bar{x}^{\frac{a}{2}} \|_{L^4}
\left( \| \na u \|_{L^2} \| \na^2 d \bar{x}^{\frac{a}{2}} \|_{L^4}
+ \| \na u \|_{L^4} \| \na^3 d \bar{x}^{\frac{a}{2}} \|_{L^{2}} \right) \\
& \quad + C \| \na u \|_{L^2} \| \na^2 d \bar{x}^{\frac{a}{2}} \|^2_{L^{4}} \\
& \le \frac{1}{8} \|\nabla^{3} d \bar{x}^{\frac{a}{2}}\|_{L^{2}}^{2}
+ C \left( 1 + \| \na^2 u \|^2_{L^2} + \|\nabla^2 d\bar{x}^{\frac{a}{2}}\|_{L^{2}}^{2} \right),
\ea\ee
and that
\be\la{cp714}\ba
J_4 + J_5 + J_8 + J_9
& \le C \int |\nabla d| |\nabla^{2} d| \left( |\nabla^{2} d| + |\nabla^{3} d| \right) \bar{x}^{a}
+ |\nabla d|^{3} \left( |\nabla^{2} d| + |\nabla^{3} d| \right) \bar{x}^{a} dx \\
& \le C \| \na d \|_{L^2} \| \na^2 d \bar{x}^{\frac{a}{2}} \|^2_{L^4}
+ C \| \na d \|_{L^4} \|\nabla^2 d \bar{x}^{\frac{a}{2}}\|_{L^{4}}
\|\nabla^3 d \bar{x}^{\frac{a}{2}}\|_{L^{2}} \\
& \quad + C \| \na d \|^2_{L^8} \| \na d \bar{x}^{\frac{a}{2}} \|_{L^4}
\left( \| \na^2 d \bar{x}^{\frac{a}{2}} \|_{L^2} + \| \na^3 d \bar{x}^{\frac{a}{2}} \|_{L^2} \right) \\
& \le \frac{1}{8} \|\nabla^{3} d\bar{x}^{\frac{a}{2}}\|_{L^{2}}^{2}
+ C \left( 1 + \|\nabla^{2} d \bar{x}^{\frac{a}{2}}\|_{L^{2}}^{2} \right).
\ea\ee
Substituting \eqref{cp710}, \eqref{cp712}, \eqref{cp713}, and \eqref{cp714}
into \eqref{cp79} yields
\be\la{cp715}\ba
\frac{d}{dt} \|\nabla^{2}d\bar{x}^{\frac{a}{2}}\|_{L^{2}}^{2}
+ \|\nabla^{3} d\bar{x}^{\frac{a}{2}}\|_{L^{2}}^{2}
\le C \left( 1 + \| \na^2 u \|^2_{L^2} + \|\nabla^2 d\bar{x}^{\frac{a}{2}}\|_{L^{2}}^{2} \right).
\ea\ee

Finally, multiplying \eqref{cp715} by $t$ and integrating over $(0,T)$,
we obtain \eqref{cp007} after using \eqref{cp07}, which completes the proof of Lemma~\ref{cl7}.
\end{proof}

\begin{lemma}\la{cl8}
There exists a positive constant $C$ depending only on $T$,
$a$, $M_0$, $N_0$, $\|\n_0\|_{L^1\cap L^\infty}$, $\|\n_0\|_{H^1\cap W^{1,q}}$, $\|\bar{x}^a\n_0\|_{L^1}$,
$\| \sqrt{\n_0} u_0 \|_{L^2}$, $\| \na u_0 \|_{L^2}$, $\| \bar{x}^{\frac{a}{2}} \na d_0 \|_{L^2}$,
and $\| \na^2 d_0 \|_{L^2}$ such that
\be\ba\la{cp08}
\sup_{0\leq t\leq T} t \left( \| \sqrt{\n} u_{t} \|_{L^{2}}^{2}
+ \|d_{t}\|_{H^{1}}^{2} + \|\nabla^{3} d\|_{L^{2}}^{2}\right)
+ \int_{0}^{T} t \left( \|\nabla u_{t}\|_{L^{2}}^{2} + \|\nabla d_{t}\|_{H^{1}}^{2} \right) dt
\leq C.
\ea\ee
\end{lemma}
\begin{proof}
First, it follows from \eqref{gn11}, \eqref{cp01}, \eqref{cp02}, \eqref{cp53}, and H\"older's inequality that
\be\la{cp81}\ba
\|\sqrt{\n} u_{t}\|_{L^{2}}^{2}
\leq & C\|\sqrt{\n}\dot{u}\|_{L^{2}}^{2}+C\|\sqrt{\n}|u||\nabla u|\|_{L^{2}}^{2} \\
\leq & C\|\sqrt{\n}\dot{u}\|_{L^{2}}^{2}+C\|\sqrt{\n} u\|_{L^{6}}^{2} \|\nabla u\|_{L^{3}}^{2} \\
\leq & C\|\sqrt{\n}\dot{u}\|_{L^{2}}^{2}+C(\|\sqrt{\n} u\|_{L^{2}}^{2}
+ \|\nabla u\|_{L^{2}}^{2}) \|\nabla u\|_{L^{2}}^{\frac{4}{3}} \|\nabla^{2} u\|_{L^{2}}^{\frac{2}{3}} \\
\leq & C\|\sqrt{\n}\dot{u}\|_{L^{2}}^{2}+C(1+ \|\nabla^{2} u\|_{L^{2}}^{2}).
\ea\ee
By virtue of \eqref{gn11}, \eqref{cp01}, \eqref{cp02}, \eqref{cp44}, and \eqref{cp78}, it holds that
\be\la{cp082}\ba
\| d_t \|^2_{L^2} & \le C \left( \| u \cdot \na d \|^2_{L^2}
+ \| \na^2 d \|^2_{L^2} + \| \na d \|^4_{L^4} \right) \\
& \le C \left( \| u \bar{x}^{-\frac{a}{2}} \|^2_{L^4} \| \na d \bar{x}^{\frac{a}{2}}\|^2_{L^4} + 1 \right) \\
& \le C \left( \left( \|\sqrt{\n} u\|_{L^{2}}^{2} + \|\nabla u\|_{L^{2}}^{2} \right) \|\nabla^{2}d \bar{x}^{\frac{a}{2}}\|_{L^{2}} + 1 \right) \\
& \le C \left( 1 + \|\nabla^{2}d \bar{x}^{\frac{a}{2}}\|^2_{L^{2}} \right).
\ea\ee
On the other hand, combining \eqref{gn11}, \eqref{cp01}, \eqref{cp02}, and \eqref{cp25}, we obtain
\be\la{cp82}\ba
\|\nabla d_{t}\|_{L^{2}}^{2}
\leq & C \left( \|\nabla^{3} d\|_{L^{2}}^{2}+\||\nabla u| |\nabla d|\|_{L^{2}}^{2}
+ \||u||\nabla^{2}d|\|_{L^{2}}^{2} + \||\nabla d|^{3}\|_{L^{2}}^{2}
+ \||\nabla d||\nabla^{2}d|\|_{L^{2}}^{2} \right) \\
\leq & C \left( \|\nabla^{3} d\|_{L^{2}}^{2}+\|\nabla u\|_{L^{4}}^{2} \|\nabla d\|_{L^{4}}^{2}
+ \|u\bar{x}^{-\frac{a}{4}}\|_{L^{8}}^{2} \|\nabla^{2}d\bar{x}^{\frac{a}{2}}\|_{L^{2}}\|\nabla^{2}d\|_{L^{4}} \right) \\
& + C \left( \|\nabla d\|_{L^{6}}^{6} + \||\nabla d||\nabla^{2}d|\|_{L^{2}}^{2} \right) \\
\leq & C \left( \|\nabla^{3} d\|_{L^{2}}^{2} + \|\nabla^{2} u\|_{L^{2}}^{2}
+ \|\nabla^{2}d\bar{x}^{\frac{a}{2}}\|_{L^{2}}^{2} + 1 \right),
\ea\ee
where in the last inequality we have used the following estimate:
\be\la{cp83}\ba
\| u \bar{x}^{-\frac{a}{4}} \|_{L^{8}}^{4}
\le C \left( \|\sqrt{\n} u\|_{L^{2}}^{2}+\|\nabla u\|_{L^{2}}^{2} \right)^2 \le C,
\ea\ee
owing to \eqref{cp44}, \eqref{cp01} and \eqref{cp02}.

Therefore, we conclude from \eqref{cp81}, \eqref{cp082}, \eqref{cp82}, \eqref{cp02}, \eqref{cp05}, \eqref{cp07}, and \eqref{cp007} that
\be\la{cp84}\ba
\sup_{0\leq t\leq T} \left( t \|d_{t}\|_{L^{2}}^{2} \right)
+ \int_{0}^{T} \left(\|\sqrt{\n} u_{t}\|_{L^{2}}^{2}+\| d_{t}\|_{H^{1}}^{2}\right)dt \leq C.
\ea\ee
Next, differentiating $\eqref{nlc}_2$ with respect to $t$ shows that
\be\la{cp85}\ba
\n u_{tt} +\n u\cdot\nabla u_{t}-\Delta u_{t}= -\n_{t}(u_{t}+u\cdot\nabla u)-\n u_{t}\cdot\nabla u -\nabla P_{t} -\operatorname{div}(\nabla d\odot \nabla d)_{t}.
\ea\ee
Multiplying \eqref{cp85} by $u_t$ and integrating by parts over $\rr$, we derive after using $\eqref{nlc}_1$ that
\be\la{cp86}\ba
\frac{1}{2}\frac{d}{dt} \|\sqrt{\n} u_{t}\|_{L^{2}}^{2} + \|\nabla u_{t}\|_{L^{2}}^{2}
& \leq C \int \n |u| |u_{t}| \left( |\nabla u_{t}|+ |\nabla u |^{2} +|u||\nabla^{2} u|\right) dx \\
& \quad + C\int \n |u|^{2}|\nabla u| |\nabla u_{t}| dx
+ C \int \n |u_{t}|^{2}|\nabla u| dx \\
& \quad + C\int|\nabla d||\nabla d_{t}| |\nabla u_{t}| dx
\triangleq L_{1} + L_{2} + L_{3} + L_{4}.
\ea\ee
Now, we use \eqref{gn11}, \eqref{cp53}, \eqref{cp01}, \eqref{cp02}, and Young's inequality to estimate each term on the right-hand side of \eqref{cp86} as follows:
\be\la{cp87}\ba
L_{1} \leq & C\|\sqrt{\n} u\|_{L^{6}} \|\sqrt{\n} u_{t}\|_{L^{2}}^{\frac{1}{2}}
\|\sqrt{\n} u_{t}\|_{L^{6}}^{\frac{1}{2}} \left( \|\nabla u_{t}\|_{L^{2}}+\|\nabla u\|_{L^{4}}^{2} \right) \\
& + C\|\n^{\frac{1}{4}} u\|_{L^{12}}^{2} \|\sqrt{\n} u_{t}\|_{L^{2}}^{\frac{1}{2}}
\|\sqrt{\n} u_{t}\|_{L^{6}}^{\frac{1}{2}} \|\nabla^{2} u\|_{L^{2}}\\
\leq & C \left( 1 + \|\nabla u\|_{L^{2}} \right) \|\sqrt{\n} u_{t}\|_{L^{2}}^{\frac{1}{2}}
\left( \|\sqrt{\n} u_{t}\|_{L^{2}} +\|\nabla u_{t}\|_{L^{2}} \right)^{\frac{1}{2}}
\left( \|\nabla u_{t}\|_{L^{2}} + \|\nabla^{2}u\|_{L^{2}} \right) \\
\leq & \frac{1}{6} \|\nabla u_{t}\|_{L^{2}}^{2}
+ C \left( \|\sqrt{\n} u_{t}\|_{L^{2}}^{2} + \|\nabla^{2} u\|_{L^{2}}^{2} \right),
\ea\ee

\be\la{cp88}\ba
L_{2} + L_{3} & \leq C \|\sqrt{\n} u\|_{L^{8}}^{2} \|\nabla u\|_{L^{4}}\|\nabla u_{t}\|_{L^{2}}
+ C \|\sqrt{\n} u_{t}\|_{L^{6}}^{\frac{3}{2}} \|\sqrt{\n} u_{t}\|_{L^{2}}^{\frac{1}{2}}\|\nabla u\|_{L^{2}} \\
& \leq C \|\nabla^2 u\|^{\frac{1}{2}}_{L^{2}} \|\nabla u_{t}\|_{L^{2}}
+ C \left( \|\sqrt{\n} u_{t}\|_{L^{2}} +\|\nabla u_{t}\|_{L^{2}} \right)^{\frac{3}{2}}
\|\sqrt{\n} u_{t}\|_{L^{2}}^{\frac{1}{2}} \\
& \leq \frac{1}{6} \|\nabla u_{t}\|_{L^{2}}^{2}
+ C \left( 1+\|\sqrt{\n} u_{t}\|_{L^{2}}^{2} + \|\nabla^{2} u\|_{L^{2}}^{2} \right),
\ea\ee
and that for any $\ep>0$,
\be\la{cp89}\ba
L_{4}\leq C \|\nabla d\|_{L^{4}} \|\nabla d_{t}\|_{L^{4}} \|\nabla u_{t}\|_{L^{2}}
& \leq \frac{1}{6} \|\nabla u_{t}\|_{L^{2}}^{2}
+ C \|\nabla^{2} d\|_{L^{2}} \|\nabla d_{t}\|_{L^{2}}\|\nabla^{2} d_{t}\|_{L^{2}} \\
& \leq \frac{1}{6} \|\nabla u_{t}\|_{L^{2}}^{2}
+ \ep \|\nabla^{2} d_{t}\|_{L^{2}}^{2} + C(\ep) \|\nabla d_{t}\|_{L^{2}}^{2}.
\ea\ee
Substituting \eqref{cp87}--\eqref{cp89} into \eqref{cp86} leads to
\be\la{cp810}\ba
\frac{d}{dt} \|\sqrt{\n} u_{t}\|_{L^{2}}^{2} + \|\nabla u_{t}\|_{L^{2}}^{2}
& \leq \ep \|\nabla^{2} d_{t}\|_{L^{2}}^{2}
+ C(\ep) \left( 1+\|\sqrt{\n} u_{t}\|_{L^{2}}^{2} + \|\nabla^{2} u\|_{L^{2}}^{2} + \|\nabla d_{t}\|_{L^{2}}^{2} \right).
\ea\ee
Differentiating \eqref{cp25} with respect to $t$ gives
\be\la{cp814}\ba
\nabla d_{tt}-\Delta\nabla d_{t} =-\nabla (u\cdot \nabla d)_{t}+\nabla (|\nabla d|^{2}d)_{t}.
\ea\ee
Multiplying \eqref{cp814} by $\na d_t$ and integrating over $\rr$, we obtain
\be\la{cp815}\ba
\frac{1}{2}\frac{d}{dt}\|\nabla d_{t}\|_{L^{2}}^{2} + \|\nabla^{2} d_{t}\|_{L^{2}}^{2}
& \leq C \int |\nabla u_{t}||\nabla d||\nabla d_{t}| dx
+ C \int |\nabla u||\nabla d_{t}|^{2} dx \\
& \quad + C \int |u_{t}||\nabla d| |\nabla^{2} d_{t}| dx
+ C \int |\nabla d|^{2}|d_{t}||\nabla^{2} d_{t}| dx \\
& \quad + C \int |\nabla d||\nabla d_{t}| |\nabla^{2} d_{t}| dx
\triangleq \sum_{i=1}^{5} J_i.
\ea\ee
From \eqref{gn11}, \eqref{cp01}, \eqref{cp02}, \eqref{cp07}, \eqref{cp44}, and Young's inequality, we deduce that
\be\la{cp816}\ba
& J_1+J_2+J_4+J_5 \\
& \le C \|\nabla u_{t}\|_{L^{2}}\|\nabla d_{t}\|_{L^{4}}\|\nabla d\|_{L^{4}}
+ C \|\nabla u\|_{L^{2}}\|\nabla d_{t}\|_{L^{4}}^{2} \\
& \quad + C \| \na^2 d_t \|_{L^2} \| \na d \|^2_{L^8} \| d_t \|_{L^4}
+ C \| \na^2 d_t \|_{L^2} \| \na d \|_{L^4} \| \na d_t \|_{L^4} \\
& \le \frac{1}{2} \|\nabla u_{t}\|_{L^{2}}^{2}
+ \frac{1}{5} \|\nabla^2 d_{t}\|_{L^{2}}^{2}
+ C \| \na d_t \|^2_{L^4} + C \| d_t \|^2_{L^4} \\
& \le \frac{1}{2} \|\nabla u_{t}\|_{L^{2}}^{2}
+ \frac{1}{5} \|\nabla^2 d_{t}\|_{L^{2}}^{2}
+ C \| \na d_t \|_{L^2} \| \na^2 d_t \|_{L^2}
+ C \| d_t \|_{L^2} \| \na d_t \|_{L^2} \\
& \le \frac{1}{2} \|\nabla u_{t}\|_{L^{2}}^{2}
+ \frac{1}{4} \|\nabla^2 d_{t}\|_{L^{2}}^{2}
+ C \| \na d_t \|^2_{L^2} + C \| d_t \|^2_{L^2},
\ea\ee
and
\be\la{cp817}\ba
J_{3} & \leq C \int |u_{t}\bar{x}^{-\frac{2a-1}{4}}| |\nabla d \bar{x}^{\frac{a}{2}}|^{\frac{2a-1}{2a}}
|\nabla d|^{\frac{1}{2a}}|\nabla^{2} d_{t}| dx \\
& \leq C\|u_{t}\bar{x}^{-\frac{2a-1}{4}}\|_{L^{8a}}
\|\nabla d\bar{x}^{\frac{a}{2}}\|_{L^{2}}^{\frac{2a-1}{2a}}
\|\nabla d\|_{L^{4}}^{\frac{1}{2a}} \|\nabla^{2} d_{t}\|_{L^{2}} \\
& \leq \frac{1}{4} \|\nabla^{2}d_{t}\|_{L^{2}}^{2}
+ C \|u_{t}\bar{x}^{-\frac{2a-1}{4}}\|_{L^{8a}}^{2} \\
& \leq \frac{1}{4} \|\nabla^{2}d_{t}\|_{L^{2}}^{2}
+ C \| \sqrt{\n} u_{t} \|^2_{L^{2}}
+ \frac{\hat{C_0}}{2} \|\nabla u_{t}\|_{L^{2}}^{2}.
\ea\ee
Substituting \eqref{cp816} and \eqref{cp817} into \eqref{cp815} shows that
\be\la{cp818}\ba
\frac{d}{dt}\|\nabla d_{t}\|_{L^{2}}^{2} + \|\nabla^{2} d_{t}\|_{L^{2}}^{2}
\le ( \hat{C_0}+1 ) \|\nabla u_{t}\|_{L^{2}}^{2}
+ C \left( \| d_t \|^2_{H^1} + \| \sqrt{\n} u_{t} \|^2_{L^{2}} \right).
\ea\ee
Consequently, multiplying \eqref{cp810} by $\hat{C_0} + 2$
and adding it to \eqref{cp818}, and then choosing $\ep$ sufficiently small, we derive
\be\la{cp819}\ba
& \frac{d}{dt} \left( (\hat{C_0}+2) \|\sqrt{\n} u_{t}\|_{L^{2}}^{2}
+ \| \na d_t \|^2_{L^2} \right)
+ \|\nabla u_{t}\|_{L^{2}}^{2} + \frac{1}{2} \| \na^2 d_t \|^2_{L^2} \\
& \leq C \left( 1+\|\sqrt{\n} u_{t}\|_{L^{2}}^{2} + \|\nabla^{2} u\|_{L^{2}}^{2} + \| d_t \|^2_{H^1} \right).
\ea\ee
Multiplying the above inequality by $t$, integrating over $(0,T)$ and using \eqref{cp05} and \eqref{cp84}, we arrive at
\be\ba\la{cp820}
\sup_{0\leq t\leq T} t \left( \| \sqrt{\n} u_{t} \|_{L^{2}}^{2}
+ \|d_{t}\|_{H^{1}}^{2} \right)
+ \int_{0}^{T} t \left( \|\nabla u_{t}\|_{L^{2}}^{2} + \|\nabla d_{t}\|_{H^{1}}^{2} \right) dt
\leq C.
\ea\ee
Finally, we apply the standard elliptic estimate to \eqref{cp25}
and use \eqref{gn11}, \eqref{cp01}, \eqref{cp02}, \eqref{cp07},
\eqref{cp83}, and H\"older's inequality to obtain
\be\la{cp821}\ba
\|\nabla^{3}d\|_{L^{2}}^{2}
\leq & C \left( \|\nabla d_{t}\|_{L^{2}}^{2} + \||\nabla u||\nabla d|\|_{L^{2}}^{2}
+ \||u||\nabla^{2}d|\|_{L^{2}}^{2} + \||\nabla d|^{3}\|_{L^{2}}^{2}
+ \||\nabla^{2}d||\nabla d|\|_{L^{2}}^{2} \right) \\
\leq & C \left( \|\nabla d_{t}\|_{L^{2}}^{2} + \|\nabla u\|_{L^{4}}^{2} \|\nabla d\|_{L^{4}}^{2}
+ \|u\bar{x}^{-\frac{a}{4}}\|_{L^{8}}^{2} \|\nabla^{2}d\bar{x}^{\frac{a}{2}}\|_{L^{2}}\|\nabla^{2}d\|_{L^{4}} \right) \\
& + C \left( \|\nabla d\|_{L^{6}}^{6} + \| \na d \|^2_{L^4} \| \na^2 d \|^2_{L^4} \right) \\
\leq & C \left( \|\nabla d_{t}\|_{L^{2}}^{2}+ \|\nabla^{2} u\|_{L^{2}}^{2}
+ \|\nabla^{2}d\bar{x}^{\frac{a}{2}}\|_{L^{2}}^{2}
+ \|\nabla^{2}d\|_{L^{4}}^{2} + \| \na^3 d \|_{L^2} + 1 \right) \\
\leq & C \left( \|\nabla d_{t}\|_{L^{2}}^{2}+ \|\nabla^{2} u\|_{L^{2}}^{2}
+ \|\nabla^{2}d\bar{x}^{\frac{a}{2}}\|_{L^{2}}^{2} + \| \na^3 d \|_{L^2} + 1 \right) \\
\leq & \frac{1}{2} \|\nabla^{3}d\|_{L^{2}}^{2}
+ C \left( \|\nabla d_{t}\|_{L^{2}}^{2}+ \|\nabla^{2} u\|_{L^{2}}^{2}
+ \|\nabla^{2}d\bar{x}^{\frac{a}{2}}\|_{L^{2}}^{2} + 1 \right),
\ea\ee
which implies
\be\la{cp822}\ba
\|\nabla^{3}d\|_{L^{2}}^{2}
\le C \left( \|\nabla d_{t}\|_{L^{2}}^{2}+ \|\nabla^{2} u\|_{L^{2}}^{2}
+ \|\nabla^{2}d\bar{x}^{\frac{a}{2}}\|_{L^{2}}^{2} + 1 \right).
\ea\ee
This combined with \eqref{cp003}, \eqref{cp007}, and \eqref{cp820} leads to
\be\la{cp823}\ba
\sup_{0\leq t\leq T} t \|\nabla^{3}d\|_{L^{2}}^{2} \leq C,
\ea\ee
which together with \eqref{cp820} yields \eqref{cp08} and completes the proof of Lemma~\ref{cl8}.
\end{proof}

\section{A priori estimates for the Cauchy problem with non-vacuum far-field density}

In this section, for initial data $(\n_0,u_0,d_0)$ satisfying \eqref{2cpsol1}, we derive some a priori estimates for the strong solutions to the Cauchy problem \eqref{nlc}, \eqref{cztj}, and \eqref{bjtj3} with $\tilde{\n}>0$.
Without loss of generality, we assume that $\mu=1$.

\begin{lemma}\la{bl1}
There exists a positive constant $C$ depending only on $\tilde{\n}$, $s$,
$\| \n_0 - \tilde{\n} \|_{L^s \cap L^\infty}$,
$\| u_0 \|_{H^1}$, $\|d_0-\mathbf{1}\|_{L^2}$, and $\| \na d_0 \|_{H^1}$ such that
\be\la{2cp01}\ba
&\sup_{0\leq t\leq T} \left( \| \n - \tilde{\n} \|_{L^s \cap L^\infty}
+ \| \sqrt{\n} u \|^2_{L^2} + \| \na d \|^2_{L^2} + \|d-\mathbf{1}\|^2_{L^2} \right)\\
&\quad
+ \int_0^T \left( \| \na u \|^2_{L^2} + \| \na^2 d \|^2_{L^2} \right) dt
\leq C,
\ea\ee
and
\be\la{2cp02}\ba
\| u \|_{H^1} \le C \left( 1 + \| \na u \|_{L^2} \right), \quad t \in [0,T].
\ea\ee
\end{lemma}
\begin{proof}
First, from the mass equation $\eqref{nlc}_1$ and the divergence-free condition $\div u = 0$, we conclude that $\n - \tilde{\n}$ satisfies
\be\la{2cp11}\ba
(\n - \tilde{\n})_t + \div ((\n - \tilde{\n}) u) = 0.
\ea\ee
The standard argument in \cite{L1} yields
\be\la{2cp12}\ba
\| \n(t) - \tilde{\n} \|_{L^p} = \| \n_0 - \tilde{\n} \|_{L^p} \quad \text{ for all } p \in [s,\infty]
\text{ and } t \in [0,T].
\ea\ee
Furthermore, by adapting the arguments in Lemmas \ref{Nl1} and \ref{cl1}, we have
\be\la{2cp13}\ba
\sup_{0\leq t\leq T} \left( \| \sqrt{\n} u \|^2_{L^2} + \| \na d \|^2_{L^2} + \|d-\mathbf{1}\|^2_{L^2} \right)
+ \int_0^T \left( \| \na u \|^2_{L^2} + \| \na^2 d \|^2_{L^2} \right) dt
\leq C,
\ea\ee
which together with \eqref{2cp12} gives \eqref{2cp01}.

Next, for any $v \in H^1(\rr)$, we use \eqref{2cp12}, \eqref{gn11}, and Young's inequality to obtain
\be\la{2cp14}\ba
\int |v|^2 dx & \le C \int \n |v|^2 dx + C \int |\n - \tilde{\n}| |v|^2 dx \\
& \le C \| \sqrt{\n} v \|^2_{L^2} + C \| \n - \tilde{\n} \|_{L^{s+1}} \| v \|^2_{ L^{\frac{2(s+1)}{s}} } \\
& \le C \| \sqrt{\n} v \|^2_{L^2}
+ C \| v \|^{\frac{2s}{s+1}}_{L^2} \| \na v \|^{\frac{2}{s+1}}_{L^2} \\
& \le \frac{1}{2} \| v \|^2_{L^2} + C \| \sqrt{\n} v \|^2_{L^2} + C \| \na v \|^2_{L^2}.
\ea\ee
This implies that
\be\la{b6}\ba
\|v\|_{L^2} \leq C\|\sqrt{\rho} v\|_{L^2}+C\|\nabla v\|_{L^2}.
\ea\ee
Combining \eqref{b6} with \eqref{2cp13} leads to \eqref{2cp02} and completes the proof of Lemma~\ref{bl1}.
\end{proof}

By virtue of \eqref{2cp01}, the following estimates can be derived by an argument similar to that in Lemma~\ref{cl2}.

\begin{lemma}\la{bl11}
There exists a positive constant $C$ depending only on $\tilde{\n}$, $s$,
$\| \n_0 - \tilde{\n} \|_{L^s \cap L^\infty}$,
$\| u_0 \|_{H^1}$, and $\| \na d_0 \|_{H^1}$ such that for $m=0,1$,
\be\la{b2}\ba
\sup_{0\leq t\leq T} t^m \left( \| \na u \|^2_{L^2} + \| \na^2 d \|^2_{L^2} \right)
+ \int_0^T t^m \left( \| \sqrt{\n} \dot{u} \|^2_{L^2} + \| \na^3 d \|^2_{L^2} \right) dt \leq C.
\ea\ee
\end{lemma}

\begin{lemma}\la{bl2}
There exists a positive constant $C$ depending only on $\tilde{\n}$, $s$,
$\| \n_0 - \tilde{\n} \|_{L^s \cap L^\infty}$,
$\| u_0 \|_{H^1}$, and $\| \na d_0 \|_{H^1}$ such that for $m=1,2$,
\be\ba\la{b3}
\sup_{0\leq t\leq T} t^m \left( \| \sqrt{\n} \dot{u} \|^2_{L^2}
+ \| |\na^2 d| |\na d| \|^2_{L^2} \right)
+ \int_0^T t^m \left( \| \na \dot{u} \|^2_{L^2} + \| |\Delta \na d| |\na d| \|^2_{L^2} \right) dt
\leq C,
\ea\ee
and
\be\ba\la{b4}
\sup_{0\leq t\leq T} t^m \left( \| \na^2 u \|^2_{L^2}
+ \| \na P \|^2_{L^2} \right) \leq C.
\ea\ee
\end{lemma}
\begin{proof}
First, applying the operator $\dot{u}^j[\p_t + u \cdot \na]$ to $\eqref{nlc}^{j}_{2}$, summing with respect to $j$, and integrating the resulting equation over $\rr,$ one obtains
\be\la{b10}\ba
& \frac{1}{2}\frac{d}{dt} \| \sqrt{\n} \dot{u} \|_{L^{2}}^{2}
+ \|\nabla \dot{u}\|_{L^{2}}^{2} \\
& = - \int \dot{u}^{j} \left[\partial_{i}(\partial_{i} u\cdot\nabla u^{j})
+ \operatorname{div} (\partial_{i} u\partial_{i} u^{j})\right] dx
-\int \dot{u}^{j}(\partial_{t}\partial_{j} P+u\cdot\nabla\partial_{j} P) dx \\
& \quad - \left(\int \dot{u}^{j} \partial_{t}\partial_{i}(\partial_{i}d\cdot\partial_{j}d) dx
+\int \dot{u}^{j} u\cdot\nabla\partial_{i}(\partial_{i}d\cdot\partial_{j}d)  dx \right) \\
& \triangleq M_1 + M_2 + M_3.
\ea\ee
It follows from integration by parts and Young's inequality that
\begin{equation}\label{b11}
\begin{aligned}
M_1&\leq \frac{1}{8}\|\nabla \dot{u}\|^2_{L^2}+C\|\nabla u\|^4_{L^4}.
\end{aligned}
\end{equation}
By virtue of \eqref{qkjste} and \eqref{2cp01}, we have
\be\la{2cp21}\ba
\|\nabla^2 u\|_{L^2} + \|\nabla P\|_{L^2}
\leq & C \left( \|\n \dot{u}\|_{L^2} + \||\nabla d||\nabla^2 d|\|_{L^2} \right) \\
\leq & C \left( \| \sqrt{\n} \dot{u}\|_{L^2} + \||\nabla d||\nabla^2 d|\|_{L^2} \right).
\ea\ee
Integration by parts combined with \eqref{2cp21} leads to
\be\la{2cp22}\ba
M_2= & -\int\left(u_t^j \partial_t \partial_j P+u^i \partial_i u^j \partial_t \partial_j P+\dot{u}^j u^i \partial_i \partial_j P\right) d x \\
= & -\frac{d}{d t} \int u^i \partial_i u^j \partial_j P d x+\int\left(\partial_t u^i \partial_i u^j \partial_j P+u^i \partial_t \partial_i u^j \partial_j P\right) d x \\
& +\int\left(\partial_j \dot{u}^j u^i \partial_i P+\dot{u}^j \partial_j u^i \partial_i P\right) d x \\
= & -\frac{d}{d t} \int u^i \partial_i u^j \partial_j P d x+2 \int\left(\dot{u}^i \partial_i u^j \partial_j P-u^l \partial_l u^i \partial_i u^j \partial_j P\right) d x \\
& +\int\left(\partial_j u^l \partial_l u^j u^i \partial_i P+\dot{u}^j \partial_j u^i \partial_i P\right) d x \\
\leq & \frac{d}{d t} \int P \partial_j u^i \partial_i u^j dx
+ C\|\dot{u}\|_{\mathcal{BMO}}\|\nabla u\|_{L^2}\|\nabla P\|_{L^2}
+ C \|u\|_{L^{6}}\|\nabla u\|_{L^6}^2\|\nabla P\|_{L^2} \\
\leq & \frac{d}{d t} \int P \partial_j u^i \partial_i u^j dx
+ C\|\na \dot{u}\|_{L^2}\|\nabla u\|_{L^2}\|\nabla P\|_{L^2}
+ C \|u\|_{L^{6}}\|\nabla u\|_{L^6}^2\|\nabla P\|_{L^2} \\
\leq & \frac{d}{dt} \int P \partial_j u^i \partial_i u^j dx
+ \varepsilon\|\nabla \dot{u}\|_{L^2}^2
+ C \|\nabla u\|_{L^2}^2 \left( \| \sqrt{\n} \dot{u}\|_{L^2} + \||\nabla d||\nabla^2 d|\|_{L^2} \right)^2 \\
& + C \| \na u \|_{L^2}^{\frac{2}{3}} \left( \| \sqrt{\n} \dot{u} \|_{L^2} + \| |\na d| |\na^2 d| \|_{L^2} \right)^{\frac{8}{3}},
\ea\ee
where we have used the following estimate:
\be\nonumber\ba
\|u\|_{L^{6}}\|\nabla u\|_{L^6}^2 \|\nabla P\|_{L^2}
& \le C \| u \|_{L^2}^{\frac{2}{3}} \| \na^2 u \|_{L^2}^{\frac{1}{3}}
\| \na u \|_{L^2}^{\frac{2}{3}} \| \na^2 u \|_{L^2}^{\frac{4}{3}} \|\nabla P\|_{L^2} \\
& \le C \| \na u \|_{L^2}^{\frac{2}{3}} \| \na^2 u \|_{L^2}^{\frac{5}{3}} \|\nabla P\|_{L^2} \\
& \le C \| \na u \|_{L^2}^{\frac{2}{3}} \left( \| \sqrt{\n} \dot{u} \|_{L^2} + \| |\na d| |\na^2 d| \|_{L^2} \right)^{\frac{8}{3}}.
\ea\ee
This follows from \eqref{2cp21} and the Gagliardo-Nirenberg inequality.

From \eqref{2cp22} and Young's inequality, we conclude that
\begin{equation}\label{b12}
\begin{aligned}
M_2\leq & \frac{d}{dt} \int \partial_j u^k \partial_k u^j  P dx
+ \frac{1}{8}\|\nabla \dot{u}\|_{L^2}^2
+ C \|\nabla u\|_{L^2}^6 \\
& + C \| \na u \|_{L^2}^{\frac{2}{3}} \left( \| \sqrt{\n} \dot{u} \|_{L^2} + \| |\na d| |\na^2 d| \|_{L^2} \right)^{\frac{8}{3}}.
\end{aligned}
\end{equation}
Similar to the derivation of \eqref{cp34}, we arrive at
\be\la{b13}\ba
M_3 \leq & \frac{1}{8} \|\nabla \dot{u}\|_{L^{2}}^{2}
+ \hat{C}_1 \| |\Delta \nabla d| |\nabla d| \|^2_{L^{2}}
+ C \|\nabla u\|_{L^{4}}^{4} + C \| \nabla d \|_{L^{8}}^{8}
+ C \| \na^2 d \|^4_{L^4}.
\ea\ee
Note that \eqref{gn11}, \eqref{2cp21}, and Young's inequality ensure
\be\la{2cp23}\ba
\|\nabla u\|^4_{L^4}\leq & C\| \nabla u\|^{2}_{L^2}\|\nabla^2 u\|_{L^2}^{2} \\
\leq & C\|\nabla u\|^{2}_{L^2} \left( \| \sqrt{\n} \dot{u} \|_{L^2} + \| |\na d| |\na^2 d| \|_{L^2} \right)^{2} \\
\le & C \| \na u \|_{L^2}^6 + C \| \na u \|_{L^2}^{\frac{2}{3}} \left( \| \sqrt{\n} \dot{u} \|_{L^2} + \| |\na d| |\na^2 d| \|_{L^2} \right)^{\frac{8}{3}}.
\ea\ee
Combining \eqref{b10}, \eqref{b11}, \eqref{2cp22}, \eqref{b12}, \eqref{b13}, and \eqref{2cp23}, we obtain
\be\la{b14}\ba
& \frac{d}{dt} \left( \frac{1}{2} \| \sqrt{\rho} \dot{u} \|_{L^{2}}^{2} 
- \int P\partial_{j} u^{i}\partial_{i}u^{j} dx \right)
+ \frac{1}{2} \|\nabla \dot{u}\|_{L^{2}}^{2}
- \hat{C}_1 \| |\Delta \nabla d| |\nabla d| \|^2_{L^{2}} \\
& \le C \| \nabla d \|_{L^{8}}^{8} + C\| \na^2 d \|^4_{L^4}
+ C \| \na u \|_{L^2}^6 + C \| \na u \|_{L^2}^{\frac{2}{3}} \left( \| \sqrt{\n} \dot{u} \|_{L^2} + \| |\na d| |\na^2 d| \|_{L^2} \right)^{\frac{8}{3}}.
\ea\ee
Moreover, following a procedure similar to \eqref{cp36}--\eqref{cp312}, we have
\be\la{b15}\ba
&\frac{d}{dt} B_1(t) + \| |\Delta \nabla d| |\nabla d| \|^2_{L^{2}}
\leq C ( \|\nabla u\|_{L^{4}}^{4} + \| \nabla d \|_{L^{8}}^{8} + \| \na^2 d \|^4_{L^4} ),
\ea\ee
where $B_1(t)$ is defined as in \eqref{cp310} and satisfies
\be\la{2cp24}\ba
C^{-1}\||\nabla^2 d||\nabla d|\|^2_{L^2}
\le B_1(t) \le C \||\nabla^2 d||\nabla d|\|^2_{L^2}.
\ea\ee
We set
\be\la{2cp25}\ba
D_1(t) \triangleq  \| \sqrt{\rho} \dot{u} \|_{L^{2}}^{2}+\||\nabla^2 d||\nabla d|\|^2_{L^2},
\ea\ee
and
\be\la{b16}\ba
D_2(t) \triangleq \frac{1}{2} \| \sqrt{\rho} \dot{u} \|_{L^{2}}^{2}
- \int P\partial_{j} u^{i}\partial_{i}u^{j} dx
+ ( \hat{C}_1 + 1 ) B_1(t).
\ea\ee
By virtue of $\eqref{qkjste}$, \eqref{2cp24}, Lemma~\ref{hm}, and Young's inequality, we have
\be\la{b17}\ba
\left|\int P \partial_{j} u^{i} \partial_{i} u^{j} dx\right|
\leq & C \|P\|_{\mathcal{BMO}} \|\partial_{i}u\cdot\nabla u^{i}\|_{\mathcal{H}^{1}}
\leq C\|\nabla P\|_{L^{2}}\|\nabla u\|_{L^{2}}^{2} \\
\leq & C \left( \| \sqrt{\n} \dot{u}\|_{L^{2}} 
+\||\nabla^{2}d||\nabla d|\|_{L^{2}} \right) \|\nabla u\|_{L^{2}}^{2} \\
\leq & \frac{1}{4}\| \sqrt{\n} \dot{u}\|_{L^{2}}^{2} 
+ \frac{1}{4} B_1(t) + \tilde{C} \|\nabla u\|_{L^{2}}^{4},
\ea\ee
which together with \eqref{b16} yields
\be\la{b18}\ba
C^{-1} D_1(t) - \tilde{C} \| \na u \|^4_{L^2}
\le D_2(t) \le C \left( D_1(t) + \| \na u \|^4_{L^2} \right).
\ea\ee
Next, multiplying \eqref{b15} by $\hat{C}_1 + 1$, adding the result to \eqref{b14}, and using \eqref{b18}, we obtain
\begin{equation}\label{b19}
\begin{aligned}
& \frac{d}{dt} D_2(t) + \frac{1}{2} \|\nabla \dot{u}\|_{L^{2}}^{2}
+ \| |\Delta \nabla d| |\nabla d| \|^2_{L^{2}} \\
& \le C \| \nabla d \|_{L^{8}}^{8} + C\| \na^2 d \|^4_{L^4}
+ C \| \na u \|_{L^2}^6 + C \| \na u \|_{L^2}^{\frac{2}{3}} \left( \| \sqrt{\n} \dot{u} \|_{L^2} + \| |\na d| |\na^2 d| \|_{L^2} \right)^{\frac{8}{3}} \\
& \le C \|\nabla^{2} d\|_{L^{2}}^{6}
+ C \|\nabla^{2}d\|_{L^{2}}^{2}\|\nabla^{3} d\|_{L^{2}}^{2}
+ C \| \na u \|_{L^2}^6 + C \| \na u \|_{L^2}^{\frac{2}{3}} \left( D_2(t) + \tilde{C} \| \na u \|_{L^2}^4 \right)^{\frac{4}{3}},
\end{aligned}
\end{equation}
where in the second inequality we have used the following estimates:
\be\nonumber\ba
\| \nabla d \|_{L^{8}}^{8} + \| \na^2 d \|^4_{L^4}
& \leq C \|\nabla d\|_{L^{2}}^{2} \|\nabla^{2} d\|_{L^{2}}^{6}
+ C \|\nabla^{2}d\|_{L^{2}}^{2}\|\nabla^{3} d\|_{L^{2}}^{2} \\
& \leq C \|\nabla^{2} d\|_{L^{2}}^{6}
+ C \|\nabla^{2}d\|_{L^{2}}^{2}\|\nabla^{3} d\|_{L^{2}}^{2},
\ea\ee
due to \eqref{gn11} and \eqref{2cp01}.

Multiplying \eqref{b19} by $t^m$ for $m=1,2$ and using \eqref{b2}, we arrive at
\be\la{b192}\ba
& \frac{d}{dt} ( t^m D_2(t) ) + t^m \left( \frac{1}{2} \|\nabla \dot{u}\|_{L^{2}}^{2}
+ \| |\Delta \nabla d| |\nabla d| \|^2_{L^{2}} \right) \\
& \le m t^{m-1} D_2(t) + C \| \na^2 d \|_{L^2}^2 + C t \| \na^3 d \|_{L^2}^2 + C \| \na u \|_{L^2}^2 \\
& \quad + C t^m \| \na u \|_{L^2}^{\frac{2}{3}} \left( D_2(t) + \tilde{C} \| \na u \|_{L^2}^4 \right)^{\frac{4}{3}}.
\ea\ee
By virtue of \eqref{gn11}, \eqref{b2}, \eqref{b18}, \eqref{2cp25}, and H\"older's inequality, it holds that
\be\la{b20}\ba
D_2(t) & \le C \left( \| \sqrt{\rho} \dot{u} \|_{L^{2}}^{2}
+ \| |\na^2 d| |\na d| \|^2_{L^2} + \| \na u \|^4_{L^2} \right) \\
& \le C \left( \| \sqrt{\rho} \dot{u} \|_{L^{2}}^{2}
+ \| \na^2 d \|^2_{L^4}  \| \na d \|^2_{L^4} + \| \na u \|^4_{L^2} \right) \\
& \le C \left( \| \sqrt{\rho} \dot{u} \|_{L^{2}}^{2}
+ \| \na^2 d \|^2_{L^2} \| \na^3 d \|_{L^2} \| \na d \|_{L^2}+ \| \na u \|^4_{L^2} \right) \\
& \le C \left( \| \sqrt{\rho} \dot{u} \|_{L^{2}}^{2}
+ \| \na^2 d \|^2_{L^2} + \| \na^3 d \|^2_{L^2} + \| \na u \|^4_{L^2} \right),
\ea\ee
which together with \eqref{2cp01} and \eqref{b2} gives
\be\la{2cp26}\ba
\int_0^T \left( D_2(t) + t D_2(t) \right) dt \le C.
\ea\ee
Integrating \eqref{b192} over $(0,T)$ and using \eqref{2cp26}, \eqref{2cp01}, and \eqref{b2}, we derive
\be\la{2cp27}\ba
& \sup_{0 \le t \le T} ( t^m D_2(t) ) + \int_0^T t^m \left( \frac{1}{2} \|\nabla \dot{u}\|_{L^{2}}^{2}
+ \| |\Delta \nabla d| |\nabla d| \|^2_{L^{2}} \right) dt \\
& \le C + C \int_0^T t^m \| \na u \|_{L^2}^{\frac{2}{3}} \left( D_2(t) + \tilde{C} \| \na u \|_{L^2}^4 \right)^{\frac{4}{3}} dt \\
& \le C + C \sup_{0 \le t \le T} \left( t^m \left( D_2(t) + \tilde{C} \| \na u \|_{L^2}^4 \right) \right)^{\frac{2}{3}}
\int_0^T t^\frac{m}{3} \| \na u \|_{L^2}^{\frac{2}{3}}
\left( D_2(t) + \tilde{C} \| \na u \|_{L^2}^4 \right)^{\frac{2}{3}} dt \\
& \le C + C \sup_{0 \le t \le T} \left( t^m \left( D_2(t) + \tilde{C} \| \na u \|_{L^2}^4 \right) \right)^{\frac{2}{3}} \\
& \le \frac{1}{2} \sup_{0 \le t \le T} \left( t^m D_2(t) \right) + C,
\ea\ee
where we have used the following estimate:
\be\la{2cp28}\ba
& \int_0^T t^\frac{m}{3} \| \na u \|_{L^2}^{\frac{2}{3}}
\left( D_2(t) + \tilde{C} \| \na u \|_{L^2}^4 \right)^{\frac{2}{3}} dt \\
& \le \left( \int_0^T \| \na u \|_{L^2}^{2} dt \right)^{\frac{1}{3}}
\left( \int_0^T t^\frac{m}{2}\left( D_2(t) + \tilde{C} \| \na u \|_{L^2}^4 \right) dt \right)^{\frac{2}{3}} \\
& \le C \left( \int_0^T (1+t)\left( D_2(t) + \tilde{C} \| \na u \|_{L^2}^4 \right) dt \right)^{\frac{2}{3}} \\
& \le C,
\ea\ee
owing to \eqref{b18}, \eqref{2cp01}, \eqref{b2}, and \eqref{2cp26}.
Hence, combining \eqref{2cp27}, \eqref{2cp25}, \eqref{b18}, and \eqref{b2} yields \eqref{b3}.

Furthermore, from \eqref{qkjste}, we have
\be\nonumber\ba
\| \na^2 u \|^2_{L^2} + \| \na P \|^2_{L^2}
\le C \left( \| \sqrt{\n} \dot{u}\|_{L^2}^2 + \||\nabla d||\nabla^{2}d|\|_{L^2}^2 \right),
\ea\ee
which together with \eqref{b3} implies \eqref{b4} and completes the proof of Lemma~\ref{bl2}.
\end{proof}

\begin{lemma}\la{bl3}
There exists a positive constant $C$ depending only on $T$, $\tilde{\n}$, $s$, $q$, $\| \n_0 - \tilde{\n} \|_{L^s \cap L^\infty}$,
$\| \na \rho_0 \|_{L^{q}}$,
$\| u_0 \|_{H^1}$, and $\| \na d_0 \|_{H^1}$ such that
\be\ba\la{b21}
\sup_{0\leq t\leq T} \| \na \rho \|_{L^{q}}
& + \int_0^T \left( \|\nabla^{2} u\|_{L^{2}}^{2} + \|\nabla P\|_{L^{2}}^{2}
+ \|\nabla^{2} u\|_{L^{q}}^{\frac{q+1}{q}} + \|\nabla P\|_{L^{q}}^{\frac{q+1}{q}} \right) dt \\
& + \int_0^T t \left( \|\nabla^{2} u\|_{L^{q}}^{2} + \|\nabla P\|_{L^{q}}^{2} \right) dt
\leq C.
\ea\ee
\end{lemma}
\begin{proof}
First, it follows from the mass equation $\eqref{nlc}_1$ that $\na \n$ satisfies, for any $p \ge2$,
\be\la{b22}\ba
\frac{d}{dt} \|\nabla \rho\|_{L^{p}}\leq C\|\nabla u\|_{L^{\infty}} \|\nabla \rho\|_{L^{p}}.
\ea\ee
By virtue of \eqref{gn11}, \eqref{b2}, and \eqref{qkjste}, it holds that
\be\la{b23}\ba
\|\nabla u\|_{L^{\infty}}\leq & C\|\nabla u\|_{L^{2}}^{\frac{q-2}{2(q-1)}} \|\nabla^{2} u\|_{L^{q}}^{\frac{q}{2(q-1)}}
\leq C\left( \|\rho\dot{u}\|_{L^{q}}^{\frac{q}{2(q-1)}}+\||\nabla d||\nabla^{2}d|\|_{L^{q}}^{\frac{q}{2(q-1)}}\right).
\ea\ee
On the one hand, using \eqref{b6} and H\"older's inequality, we arrive at
\be\la{b24}\ba
\| \rho \dot u\|_{L^q} 
& \le C\| \rho \dot u\|_{L^2}^{\frac{2(q-1)}{q^2-2}}
\| \dot{u} \|_{L^{q^2}}^{\frac{q(q-2)}{q^2-2}} \\ 
& \le C\| \rho \dot u\|_{L^2}^{\frac{2(q-1)}{q^2-2}} 
\left( \| \sqrt{\n} \dot u\|_{L^2}+\| \na \dot u\|_{L^2} \right)^{\frac{q(q-2)}{q^2-2}} \\ 
& \le C\| \sqrt{\n}  \dot u\|_{L^2} + C\| \sqrt{\n} \dot u\|_{L^2}^{\frac{2(q-1)}{q^2-2}}
\|\na \dot u\|_{L^2}^{\frac{q(q-2)}{q^2-2}},
\ea\ee
which together with \eqref{b2}, \eqref{b3}, \eqref{b24}, and Young's inequality yields
\be\la{b25}\ba
& \int_0^T \left( \|\rho \dot u\|^{\frac{q+1}{q}}_{L^q}+t\| \n \dot u\|^2_{L^q} \right) dt \le C.
\ea\ee
On the other hand, in view of \eqref{gn11}, \eqref{b2}, \eqref{b3}, and H\"older's inequality, we have
\be\la{b26}\ba
\int_{0}^{T} \||\nabla d||\nabla^{2}d|\|_{L^{q}}^{2} dt
& \leq \int_{0}^{T} \|\nabla d\|^2_{L^{2q}} \|\nabla^{2}d\|^2_{L^{2q}} dt \\
& \le C \int_{0}^{T} \|\nabla d\|^2_{H^1} \|\nabla^{2}d\|^2_{H^1} dt \\
& \leq C.
\ea\ee
Combining \eqref{b23}, \eqref{b25}, and \eqref{b26} leads to
\be\la{b27}\ba
\int_0^T \| \na u \|_{L^\infty} dt \le C,
\ea\ee
which together with \eqref{b22} and Gr\"onwall's inequality gives
\be\la{b28}\ba
\sup_{0 \le t \le T} \| \na \n \|_{L^q} \le C.
\ea\ee
Finally, by \eqref{qkjste}, \eqref{b2}, \eqref{b3}, \eqref{b25}, \eqref{b26}, and H\"older's inequality, we obtain
\be\la{b29}\ba
& \int_0^T \left( \|\nabla^{2} u\|_{L^{2}}^{2} + \|\nabla P\|_{L^{2}}^{2}
+ \|\nabla^{2} u\|_{L^{q}}^{\frac{q+1}{q}} + \|\nabla P\|_{L^{q}}^{\frac{q+1}{q}}
+ t \left( \|\nabla^{2} u\|_{L^{q}}^{2} + \|\nabla P\|_{L^{q}}^{2} \right)  \right) dt \\
& \leq C+C \int_0^T \left( \|\rho \dot u\|^{\frac{q+1}{q}}_{L^q}+t\| \n \dot u\|^2_{L^q} \right) dt \le C.
\ea\ee
This, combined with \eqref{b28}, implies \eqref{b21} and completes the proof of Lemma~\ref{bl3}.
\end{proof}

\begin{lemma}\la{bl4}
There exists a positive constant $C$ depending only on $T$, $\tilde{\n}$, $s$, $q$, $\| \n_0 - \tilde{\n} \|_{L^s \cap L^\infty}$,
$\| \na \rho_0 \|_{L^{q}}$,
$\| u_0 \|_{H^1}$, and $\| \na d_0 \|_{H^1}$ such that
\be\ba\la{b30}
\sup_{0\leq t\leq T} t \left(\| \sqrt{\n} u_{t} \|_{L^{2}}^{2}
+ \|\na d_{t}\|_{L^2}^{2} + \|\nabla^{3} d\|_{L^{2}}^{2}\right)
+ \int_{0}^{T} t \left( \|\nabla u_{t}\|_{L^{2}}^{2} + \|\nabla d_{t}\|_{H^{1}}^{2} \right) dt
\leq C.
\ea\ee
\end{lemma}
\begin{proof}
First, we use \eqref{gn11}, \eqref{2cp01}, \eqref{2cp02}, \eqref{b2}, and H\"older's inequality to derive
\be\la{b31}\ba
\|\sqrt{\n} u_{t}\|_{L^{2}}^{2}
\leq & C\|\sqrt{\n}\dot{u}\|_{L^{2}}^{2}+C\|\sqrt{\n}|u||\nabla u|\|_{L^{2}}^{2} \\
\leq & C\|\sqrt{\n}\dot{u}\|_{L^{2}}^{2}+C\|\n\|_{L^\infty}\| u\|_{L^{6}}^{2} \|\nabla u\|_{L^{3}}^{2} \\
\leq & C\|\sqrt{\n}\dot{u}\|_{L^{2}}^{2}+C(1+ \|\nabla^{2} u\|_{L^{2}}^{2}),
\ea\ee
and
\be\la{b33}\ba
\|\nabla d_{t}\|_{L^{2}}^{2}
\leq & C \left( \|\nabla^{3} d\|_{L^{2}}^{2}+\||\nabla u| |\nabla d|\|_{L^{2}}^{2}
+ \||u||\nabla^{2}d|\|_{L^{2}}^{2} + \||\nabla d|^{3}\|_{L^{2}}^{2}
+ \||\nabla d||\nabla^{2}d|\|_{L^{2}}^{2} \right) \\
\leq & C \left( \|\nabla^{3} d\|_{L^{2}}^{2}+\|\nabla u\|_{L^{4}}^{2} \|\nabla d\|_{L^{4}}^{2}
+ \|u\|_{L^{8}}^{2} \|\nabla^{2}d\|_{L^{2}}\|\nabla^{2}d\|_{L^{4}} \right) \\
& + C \left( \|\nabla d\|_{L^{6}}^{6} + \||\nabla d||\nabla^{2}d|\|_{L^{2}}^{2} \right) \\
\leq & C \left( \|\nabla^{3} d\|_{L^{2}}^{2} + \|\nabla^{2} u\|_{L^{2}}^{2} + 1 \right),
\ea\ee
which together with \eqref{b2} and \eqref{b21} yields
\be\la{b34}\ba
\int_{0}^{T} \left(\|\sqrt{\n} u_{t}\|_{L^{2}}^{2}+\| d_{t}\|_{H^{1}}^{2}\right)dt \leq C.
\ea\ee
Differentiating $\eqref{nlc}_2$ with respect to $t$ leads to
\be\la{b35}\ba
\n u_{tt} +\n u\cdot\nabla u_{t}-\Delta u_{t}= -\n_{t}(u_{t}+u\cdot\nabla u)-\n u_{t}\cdot\nabla u -\nabla P_{t} -\operatorname{div}(\nabla d\odot \nabla d)_{t}.
\ea\ee
Multiplying \eqref{b35} by $u_t$, integrating by parts over $\rr$, and using \eqref{gn11}, \eqref{2cp01}, \eqref{2cp02}, and \eqref{b2}, we obtain
\be\la{b36}\ba
& \frac{1}{2}\frac{d}{dt} \|\sqrt{\n} u_{t}\|_{L^{2}}^{2} + \|\nabla u_{t}\|_{L^{2}}^{2} \\
& \leq C \int \n |u| |u_{t}| \left( |\nabla u_{t}|+ |\nabla u |^{2} +|u||\nabla^{2} u|\right) dx + C\int \n |u|^{2}|\nabla u| |\nabla u_{t}| dx \\
& \quad + C \int \n |u_{t}|^{2}|\nabla u| dx + C\int|\nabla d||\nabla d_{t}| |\nabla u_{t}| dx \\
& \le C \|\n\|_{L^{\infty}}^\frac{3}{4} \|u\|_{L^{8}} \| \sqrt{\n} u_t \|_{L^2}^\frac{1}{2} \| u_{t}\|_{L^{4}}^\frac{1}{2} \|\nabla u_{t}\|_{L^{2}}
+ C \|\n\|_{L^{\infty}} \| u \|_{L^{4}} \| u_t \|_{L^{4}} \|\nabla u\|_{L^{4}}^{2} \\
& \quad + C\|\rho\|_{L^\infty}\|u\|_{L^{12}}^{2} \|u_{t}\|_{L^{3}} \|\nabla^{2} u\|_{L^{2}} + C\|\n\|_{L^\infty} \|u\|_{L^{8}}^{2} \|\nabla u\|_{L^{4}}\|\nabla u_{t}\|_{L^{2}} \\
&\quad + C\|\n\|^{\frac{3}{4}}_{L^\infty} \| u_{t}\|_{L^{6}}^{\frac{3}{2}} \|\sqrt{\n} u_{t}\|_{L^{2}}^{\frac{1}{2}}\|\nabla u\|_{L^{2}} + C \|\nabla d\|_{L^{4}} \|\nabla d_{t}\|_{L^{4}} \|\nabla u_{t}\|_{L^{2}} \\
& \le C \left( \|\sqrt{\n} u_{t}\|_{L^{2}} +\|\nabla u_{t}\|_{L^{2}} \right)^{\frac{3}{2}} \|\sqrt{\n} u_{t}\|_{L^{2}}^{\frac{1}{2}}
+ C \left( \|\sqrt{\n} u_{t}\|_{L^{2}} + \|\nabla u_{t}\|_{L^{2}} \right)
\left( \|\nabla^{2}u\|_{L^{2}} + 1 \right) \\
& \quad + \frac{1}{6} \|\nabla u_{t}\|_{L^{2}}^{2}
+ C \|\nabla^{2} d\|_{L^{2}} \|\nabla d_{t}\|_{L^{2}} \|\nabla^{2} d_{t}\|_{L^{2}} \\
& \le \frac{1}{2} \|\nabla u_{t}\|_{L^{2}}^{2} + \frac{\ep}{2} \|\nabla^{2} d_{t}\|_{L^{2}}^{2}
+ C(\ep) \left( 1+\|\sqrt{\n} u_{t}\|_{L^{2}}^{2} + \|\nabla^{2} u\|_{L^{2}}^{2} + \|\nabla d_{t}\|^2_{L^{2}} \right),
\ea\ee
which gives
\be\la{b40}\ba
\frac{d}{dt} \|\sqrt{\n} u_{t}\|_{L^{2}}^{2} + \|\nabla u_{t}\|_{L^{2}}^{2}
& \leq \ep \|\nabla^{2} d_{t}\|_{L^{2}}^{2}
+ C(\ep) \left( 1+\|\sqrt{\n} u_{t}\|_{L^{2}}^{2} + \|\nabla^{2} u\|_{L^{2}}^{2} + \|\nabla d_{t}\|^2_{L^{2}} \right).
\ea\ee
On the other hand, differentiating $\eqref{nlc}_3$ with respect to $t$ and then applying $\na$, we see that $\na d_t$ satisfies
\be\la{2cp51}\ba
\nabla d_{tt}-\Delta\nabla d_{t} =-\nabla (u\cdot \nabla d)_{t}+\nabla (|\nabla d|^{2}d)_{t}.
\ea\ee
Multiplying \eqref{2cp51} by $\na d_t$, integrating by parts over $\rr$, and using \eqref{gn11}, \eqref{2cp01}, \eqref{2cp02}, \eqref{b2}, and Young's inequality, we arrive at
\be\la{b41}\ba
& \frac{1}{2}\frac{d}{dt}\|\nabla d_{t}\|_{L^{2}}^{2} + \|\nabla^{2} d_{t}\|_{L^{2}}^{2} \\
& \leq C \int |\nabla u_{t}||\nabla d||\nabla d_{t}| dx
+ C \int |\nabla u||\nabla d_{t}|^{2} dx + C \int |u_{t}||\nabla d| |\nabla^{2} d_{t}| dx \\
& \quad + C \int |\nabla d|^{2}|d_{t}||\nabla^{2} d_{t}| dx + C \int |\nabla d||\nabla d_{t}| |\nabla^{2} d_{t}| dx \\
& \le C \|\nabla u_{t}\|_{L^{2}}\|\nabla d_{t}\|_{L^{4}}\|\nabla d\|_{L^{4}}
+ C \|\nabla u\|_{L^{2}}\|\nabla d_{t}\|_{L^{4}}^{2}
+ C \|u_t\|_{L^4}\|\nabla d\|_{L^4}\|\nabla^2 d_t\|_{L^2} \\
& \quad + C \| \na^2 d_t \|_{L^2} \| \na d \|^2_{L^8} \| d_t \|_{L^4}
+ C \| \na^2 d_t \|_{L^2} \| \na d \|_{L^4} \| \na d_t \|_{L^4} \\
& \le \frac{1}{4} \|\nabla^2 d_{t}\|_{L^{2}}^{2}
+ C \| \na d_t \|^2_{L^4} + C \| d_t \|^2_{L^4} + C \| \sqrt{\n} u_{t} \|^2_{L^{2}} + C \|\nabla u_{t}\|_{L^{2}}^{2} \\
& \le \frac{1}{4} \|\nabla^2 d_{t}\|_{L^{2}}^{2}
+ C \| \na d_t \|_{L^2} \| \na^2 d_t \|_{L^2} + C \| d_t \|^2_{H^1} + C \| \sqrt{\n} u_{t} \|^2_{L^{2}} + C \|\nabla u_{t}\|_{L^{2}}^{2} \\
& \le \frac{1}{2} \|\nabla^2 d_{t}\|_{L^{2}}^{2}
+ C \| d_t \|^2_{H^1} + C \| \sqrt{\n} u_{t} \|^2_{L^{2}} + \frac{\hat{C}_3}{2} \|\nabla u_{t}\|_{L^{2}}^{2},
\ea\ee
which implies
\be\la{b44}\ba
\frac{d}{dt}\|\nabla d_{t}\|_{L^{2}}^{2} + \|\nabla^{2} d_{t}\|_{L^{2}}^{2}
\le \hat{C}_3 \|\nabla u_{t}\|_{L^{2}}^{2}
+ C \left( \| d_t \|^2_{H^1} + \| \sqrt{\n} u_{t} \|^2_{L^{2}} \right).
\ea\ee
Multiplying \eqref{b40} by $\hat{C}_3 + 1$,
adding the resulting equation to \eqref{b44}, and taking $\ep$ sufficiently small, we obtain
\be\la{b45}\ba
& \frac{d}{dt} \left( (\hat{C}_3+1) \|\sqrt{\n} u_{t}\|_{L^{2}}^{2}
+ \| \na d_t \|^2_{L^2} \right)
+ \|\nabla u_{t}\|_{L^{2}}^{2} + \frac{1}{2} \| \na^2 d_t \|^2_{L^2} \\
& \leq C \left( 1+\|\sqrt{\n} u_{t}\|_{L^{2}}^{2} + \|\nabla^{2} u\|_{L^{2}}^{2} + \| d_t \|^2_{H^1} \right).
\ea\ee
Multiplying \eqref{b45} by $t$, integrating over $(0,T)$ and using \eqref{b21} and \eqref{b34} lead to
\be\ba\la{b46}
\sup_{0\leq t\leq T} t \left( \| \sqrt{\n} u_{t} \|_{L^{2}}^{2}
+ \|d_{t}\|_{H^{1}}^{2} \right)
+ \int_{0}^{T} t \left( \|\nabla u_{t}\|_{L^{2}}^{2} + \|\nabla d_{t}\|_{H^{1}}^{2} \right) dt
\leq C.
\ea\ee
Finally, applying the standard elliptic estimate to \eqref{cp25}
and using \eqref{gn11}, \eqref{2cp01}, \eqref{2cp02}, \eqref{b2}, and Young's inequality, we derive
\be\la{b47}\ba
& \|\nabla^{3}d\|_{L^{2}}^{2} \\
& \leq C \left( \|\nabla d_{t}\|_{L^{2}}^{2} + \||\nabla u||\nabla d|\|_{L^{2}}^{2}
+ \||u||\nabla^{2}d|\|_{L^{2}}^{2} + \||\nabla d|^{3}\|_{L^{2}}^{2}
+ \||\nabla^{2}d||\nabla d|\|_{L^{2}}^{2} \right) \\
& \leq C \left( \|\nabla d_{t}\|_{L^{2}}^{2} + \|\nabla u\|_{L^{4}}^{2} \|\nabla d\|_{L^{4}}^{2}
+ \|u\|_{L^{4}}^{2} \|\nabla^{2}d\|^2_{L^{4}} + \|\nabla d\|_{L^{6}}^{6} + \| \na d \|^2_{L^4} \| \na^2 d \|^2_{L^4} \right) \\
& \leq C \left( \|\nabla d_{t}\|_{L^{2}}^{2}+ \|\nabla^{2} u\|_{L^{2}}^{2}
+ \|\nabla^{2}d\|_{L^{2}}^{2}
+ \|\nabla^{2} d\|_{L^{2}} \| \na^3 d \|_{L^2} + 1 \right) \\
& \leq \frac{1}{2} \|\nabla^{3}d\|_{L^{2}}^{2}
+ C \left( \|\nabla d_{t}\|_{L^{2}}^{2}+ \|\nabla^{2} u\|_{L^{2}}^{2}
+ 1 \right),
\ea\ee
which gives
\be\la{b48}\ba
\|\nabla^{3}d\|_{L^{2}}^{2}
\le C \left( \|\nabla d_{t}\|_{L^{2}}^{2}+ \|\nabla^{2} u\|_{L^{2}}^{2}
+ 1 \right).
\ea\ee
From \eqref{b4}, \eqref{b46}, and \eqref{b48}, we conclude that
\be\la{b49}\ba
\sup_{0\leq t\leq T} t \|\nabla^{3}d\|_{L^{2}}^{2} \leq C.
\ea\ee
Combining \eqref{b46} and \eqref{b49} implies \eqref{b30} and completes the proof of Lemma~\ref{bl4}.
\end{proof}

\section{Proofs of Theorems \ref{thnp1}--\ref{th2cp1}}
In this section, with all the a priori estimates in Sections 3--5 at hand, we are ready to prove the main results of this paper, Theorems \ref{thnp1}--\ref{th2cp1}.

Proof of Theorem~\ref{thnp1}.
According to the local existence result (Lemma~\ref{lct}), there exists a $T_*>0$ such that the system \eqref{nlc}--\eqref{bjtj1}
admits a unique local strong solution $(\n,u,P,d)$ on $\OM \times (0,T_*]$. We now extend this local strong solution globally in time. It follows from \eqref{npsol01} that there exists a $T_1 \in (0,T_*] $ such that \eqref{n2} holds for $T=T_1$.

Next, we set
\be\la{y522}\ba
T^* \triangleq \sup\{T\,| (\n,u,P,d) \text{ is a strong solution on } \OM \times (0,T] \text{ and } \,\eqref{n2} \  \text{holds} \}.
\ea\ee
Then $T^*\ge T_1>0$. Hence, for any $0<\tau<T\leq T^*$ with $T$ finite, we deduce
from \eqref{np03} and \eqref{np04} that, for any $2 \le r <\infty$,
\be\la{y523}\ba
\na u,\ \na^2 d,\ \na d,\ d \in C \left([\tau ,T];L^r\right),
\ea\ee
where we have used the standard embedding
\be\nonumber\ba
L^\infty(\tau,T;H^1) \cap H^1(\tau,T;L^2) \hookrightarrow
C\left([\tau,T];L^r\right),\ \text{ for any } r \in [2,\infty).
\ea\ee
Furthermore, combining $\eqref{nlc}_1$ with \eqref{np05} and applying the standard argument in \cite[Lemma 2.3]{L1}, we obtain
\be\la{y524}\ba
\n \in C \left([0,T];W^{1,q} \right).
\ea\ee

Finally, we claim that
\be\la{y527}\ba
T^*=\infty.
\ea\ee
Otherwise, $T^*<\infty$.
Then Proposition~\ref{n1} implies that \eqref{np552} holds for $T=T^*$.
From \eqref{y523} and \eqref{y524}, we conclude that
\be\nonumber\ba
(\n,u,d)(x,T^*) = \lim_{t \to T^*}(\n,u,d)(x,t)
\ea\ee
satisfies \eqref{npsol1}.
Hence, we may take $(\n,u,d)(x,T^*)$ as the initial data, and Lemma~\ref{lct} shows that there exists some $T^{**}>T^*$ such that
\eqref{n2} holds for $T=T^{**}$.
This contradicts \eqref{y522}, thereby proving \eqref{y527}.
The uniqueness of $(\n,u,P,d)$ satisfying \eqref{npsol2} follows from arguments similar to those in \cite{LLTZ}.
This completes the proof of Theorem~\ref{thnp1}.

Theorems \ref{thcp1} and \ref{th2cp1} can be proved by combining the a priori estimates for the Cauchy problem established in Sections 4 and 5 with standard arguments.
Since the proofs are similar to those in \cite{LSZ,LLTZ}, the details are omitted.

\bigskip

\noindent\textbf{Data availability.} No data was used for the research described in the article.

\bigskip

\noindent\textbf{Conflict of interest.} The authors declare that they have no conflict of interest.

\end{document}